\documentclass[10 pt,reqno]{amsart}
\usepackage{amsmath}
\usepackage{amsfonts}
\usepackage{amssymb}
\usepackage{amsthm}
\usepackage{amscd}
\usepackage{a4wide}
\usepackage{enumerate}
\usepackage{pifont}
\usepackage{dsfont} % i added for blackboard bold 1's
\usepackage{xcolor}
\usepackage{amsmath}
\usepackage{tikz}
\usepackage{tikz-cd}
\usepackage{booktabs}
\usetikzlibrary{shapes.geometric}
\usepackage{amscd}
\usepackage{mathabx}
\numberwithin{equation}{section}
\usepackage[original]{imakeidx}
\makeindex % will get the same name as the main file
\usepackage{mathrsfs}
\usepackage{graphics}
\usepackage{eucal}
\usepackage{mathrsfs}
\usepackage{mdwlist}
\usepackage{color}
\usepackage[pdftex,bookmarks,colorlinks,breaklinks]{hyperref}  % PDF hyperlinks, with coloured links
\usepackage{mdwlist}
\usepackage{environ}
\hypersetup{linkcolor=red,citecolor=blue,filecolor=dullmagenta,urlcolor=darkblue} 
\usepackage{enumitem}
\usepackage{multicol}
\setenumerate[1]{label=\thesection.\arabic*.}
\newif\ifhinting\hintingtrue % hints are on by defautl% coloured links
\usepackage{fancyhdr}
\fancypagestyle{plain}{ %
  \fancyhf{} % remove everything
}
\newtheorem {theorem}    {Theorem}[section]
\newtheorem {lemma}      [theorem]    {Lemma}

\newtheorem {corollary}  [theorem]    {Corollary}
\newtheorem {proposition}[theorem]    {Proposition}

\theoremstyle{definition}
\newtheorem {definition} [theorem]    {Definition}
\newtheorem {example}    [theorem]    {Example}
\newtheorem {remark}    [theorem]    {Remark}

\newcounter{AbcT}

\numberwithin{equation}{section}

\newcommand{\IGNORE}[1]{}

\renewcommand{\liminf}{\varliminf}
\renewcommand{\limsup}{\varlimsup}

\DeclareMathOperator{\Span}{Span}

 \DeclareMathOperator{\SL}{SL}

\DeclareMathOperator{\GL}{GL}

\newcommand{\defeq}{\stackrel{\text{def}}{=}}

\newcommand{\diag}{\operatorname{diag}}

\begin{document}
\title{Inhomogeneous Khintchine-Groshev type theorems on manifolds over function fields}

\begin{abstract} The goal of this paper is to establish a complete Khintchine-Groshev type theorem in both homogeneous and inhomogeneous settings, on analytic nondegenerate manifolds over a local field of positive characteristic. The dual form of Diophantine approximation has been considered here. The core of our argument is a sharp estimate for flows on homogeneous spaces, which had been originally invented by Kleinbock and Margulis in \cite{KM}. Our treatise provides the function field analogs of the various earlier results of this type, studied in the Euclidean and $S$-adic framework, by Bernik, Kleinbock and Margulis \cite{BKM}, Beresnevich, Bernik, Kleinbock and Margulis \cite{BBKM}, Badziahin, Beresnevich and Velani \cite{BBV}, Mohammadi and Golsefidy \cite{MG1, MG2} and Datta and Ghosh \cite{DG1}.

\end{abstract}
\subjclass[2020]{11J61, 11J83, 11K60, 37D40, 37A17, 22E40} \keywords{
Diophantine approximation in positive characteristic, Khintchine type theorem, Dynamical systems, Homogeneous flows}

\author{Sourav Das}

\author{Arijit Ganguly}
 
\address{Department of Mathematics and Statistics, Indian Institute of Technology Kanpur, Kanpur, 208016, India}
\email{iamsouravdas1@gmail.com, arijit.ganguly1@gmail.com}
\maketitle
\section{Introduction}\label{section:intro}
%Diophantine approximation in the above context deals with the quantitative aspects of the approximation of a Laurent series in $\mathbb{F}_{q}((X^{-1}))$ by rational functions in $X$, and its higher dimensional analogues. It emerges as an active field of research when Mahler developed the geometry of numbers in this context in 1940 (see \cite{M}), and since then an extensive study follows. The theory over function fields has many interesting parallels with that of over $\mathbb{R}$. Indeed, exactly similar to the scenario in the Euclidean spaces, one can use geometry of numbers to prove the version of the Dirichlet's theorem here (for an elementary proof of the most general multiplicative form see \cite{GG}, and also for a global version see \cite[Appendix]{BGM}), and furthermore, can ask for its improvability. We refer the reader to \cite{Spr1} for more analogous results. Sooner or later, the metrical questions inevitably demand a lot of considerations. The present project is devoted to establish one of the most fundamental results in the metric theory, namely Khintchine's theorem, for nondegenerate manifolds in positive characteristic. In fact, we prove the `zero-one law' in both homogeneous and inhomogeneous cases. Before we turn to the statements of our main results, let us recall the  over $\mathbb{R}$.  \\

Diophantine approximation is the quantitative study of the density of rational numbers inside real numbers, and its higher dimensional analogs. To begin with, we recall the classical Khintchine-Groshev theorem. Let  $\psi: \mathbb{Z}_{\geq 0} \rightarrow \mathbb{R}_+$ and $n\in \mathbb{N}$.
%which we refer to as the \emph{approximation function} in what follows.
We say that $\mathbf{x}\in \mathbb{R}^n$ is \emph{$\psi$-approximable} if for infinitely many $\mathbf{q}\in\mathbb{Z}^n \setminus \{\mathbf{0}\}$, one has the following:
\[|\mathbf{x}\cdot \mathbf{q}+p|<\psi(\|\mathbf{q}\|^n), \text{ for some }p\in \mathbb{Z},\] where $\|\mathbf{q}\|$ stands for the supremum norm of $\mathbf{q}$. In the special case $\psi(t)=\frac{1}{t^{1+\varepsilon}}$, where $\varepsilon>0$, we call $\mathbf{x}$ a \emph{very well approximable} vector. Denote the set of $\psi$-approximable vectors  by $\mathcal{W}_n(\psi)$. The classical Khintchine-Groshev theorem says:
\[\mathcal{W}_n(\psi)=\begin{cases}
\text{Lebesgue null}  & \text{if } \displaystyle \sum_{\mathbf{q}\in \mathbb{Z}^n\setminus \{\mathbf{0}\}} \psi(\|\mathbf{q}\|^n)<\infty\\
\text{Lebesgue full}   & \text{if }\displaystyle \sum_{\mathbf{q}\in \mathbb{Z}^n\setminus \{\mathbf{0}\}}\psi(\|\mathbf{q}\|^n)=\infty, \text{ provided } \psi \text{ is decreasing.}
\end{cases}\]  Note that the monotonicity assumption in the divergence case is absolutely necessary when $n=1$, but it can be dropped for $n\geq 2$. In fact, the above theorem can be generalized to $m\times n$ matrices as well, for any $m,n\in \mathbb{N}$, and if $nm>1$, it is no longer necessary that $\psi$ is decreasing to establish the divergence case (see \cite[Theorem 1]{BV2}). \\

The central question of ``Diophantine approximation on manifolds" is to which extent an embedded submanifold $\mathcal{M}$ of $\mathbb{R}^n$
inherits the diophantine properties that are prevalent in $\mathbb{R}^n$. For instance, given a submanifold $\mathcal{M}\subseteq \mathbb{R}^n$, one can inquire about the very well approximable vectors lying in it or the validity of Khintchine type theorems etc. 
The situation turns out to be more delicate here as almost all the classical techniques seem to fail. Not only that, it is also imperative to take correct topological, geometric, analytic etc. assumptions with regard to $\mathcal{M}$ into account. This branch originates with the following conjecture of K. Mahler in 1932:\\

\emph{For Lebesgue almost every $x\in \mathbb{R}$, the vector $(x,x^2,\dots,x^n)$ is not very well approximable.}\\

\noindent Later in 1960's, while resolving the conjecture of Mahler, V. Sprind\v{z}uk extended the conjecture to any nondegenerate analytic submanifold of $\mathbb{R}^n$. An input from Baker  further strengthens the same, which we usually refer to as the `Baker-Sprind\v{z}uk conjecture' in literature. \\

This long-standing conjecture has been proved in the landmark work \cite{KM} by D.Y. Kleinbock and G.A. Margulis. The game changer is the machinery of the homogeneous dynamics that the authors have deployed to attack the conjecture. Their strategy can be put in brief as follows. The authors have shown that the property of a point of $\mathbb{R}^n$ being very well approximable can be translated to a particular cusp excursion behavior of points in the  space of
unimodular lattices  under
the diagonal flow. In view of this, proving the conjecture thus amounts to showing such orbits are rare. Then to accomplish that, they proved a very strong quantitative estimate on nondivergence of trajectories refining all such estimates that used to exist before in the literature.   \\

The `Quantitative nondivergence' estimate proved in \cite{KM} comes out to be very influential in the study of metrical aspects of Diophantine approximation from then onwards. The subsequent works \cite{K}, \cite{BKM}, and \cite{KT} extend and generalize the same to different directions and contexts so as to yield more interesting number theoretic phenomena. Indeed, with an appropriate modification to that of \cite{KM}, V. Bernik, D.Y. Kleinbock and G. A. Margulis proved the convergence case of Khintchine's theorem on nondegenerate submanifolds of Euclidean spaces in \cite{BKM}. This has further been made effective in \cite{Ad} by  F. Adiceam et al.  The divergence case of the same has been proved by V. Beresnevich, V. Bernik, D.Y. Kleinbock, and G.A. Margulis in \cite{BBKM}.  A complete inhomogeneous version of the Khintchine-Groshev theorem for nondegenerate submanifolds of $\mathbb{R}^n$  has been established by D. Badziahin, V. Beresnevich and S. Velani in \cite{BBV}.\\

%It is quite pertinent to inquire whether Khinchine-Groshv type theorems hold for affine subspaces as well. A. Ghosh established the convergence case of Khitchine's theorem for affine hyperplanes and more generally, for affine subspaces and their nondegenrate submanifolds in \cite{G-hyper} and \cite{G-monat} respectively. A subsequent work \cite{GG-quant} with the second named author provides a quantitative version of the latter. The inhomogeneous versions of both the convergence and divergence cases in this context are proved by  V. Beresnevich, A. Ganguly, A. Ghosh and S. Velani in \cite{BGGV}. We would like to refer the reader to \cite{G-survey} for a survey of some of the recent developments in Diophantine approximation in subspaces of $\mathbb{R}^n$ and its connection with homogeneous dynamics.\\

Sooner rather than later, Diophantine approximation has also been explored in the context of other local fields (ultrametric) leading to the natural analogs of some of the results mentioned above. \vspace{0.1 cm} 
\begin{itemize}
    \item[\ding{43}] \textbf{$p$-adic case}: V. Sprind\v{z}uk established the $p$-adic version of Mahler's conjecture in \cite{Spr2}. Modifying the techniques of \cite{KM}, Kleinbock and Tomanov proved the $S$-adic version of Baker-Sprind\v{z}uk conjecture in \cite{KT}. The convergence case of $S$-adic Khintchine theorem for non-degenerate analytic manifolds was proved by A. Mohammadi and A. Salehi Golsefidy in \cite{MG1}. They also proved the divergence case of the same for $\mathbb{Q}_p$ in \cite{MG2}. Inhomogeneous $S$-adic convergence and divergence Khintchine type theorems  has been proved by S. Datta and A. Ghosh in \cite{DG1}. For various other important results in the context of $S$-adic Diophantine approximation see \cite{DG2,DG3,DR,DaG}. \vspace{0.2cm}
    \item[\ding{43}] \textbf{Over function fields}: Diophantine approximation over function fields deals with the quantitative aspects of the approximation of a Laurent series by rational functions and its higher dimensional analogs. The theory of Diophantine approximation over local fields of positive characteristic has many interesting parallels with that of the classical Euclidean scenario. The classical theory of regular continued fractions carries over to fields of formal power series. We refer the reader to \cite{L1,L2, deM, schmidt,  Fuch1} for a variety of results in that. K. Mahler developed the geometry of numbers in this context in \cite{M}. Using this one can easily prove the analog of Dirichlet's theorem. An elementary proof of the same, in fact the general multiplicative version, can be found in \cite{GG}. See also \cite[Appendix]{BGM} for a general form of Dirichlet’s theorem for finite extension of function fields. \\ 
    
     The Khintchine theorem in the positive characteristic setup was first proved by B. deMathan in \cite{deM}. This has been extended subsequently to systems of linear forms by S. Kristensen in \cite{Kr1},  and further examination of the exceptional set, in terms of Hausdorff dimension, has been carried out. See also \cite{AGP} for the multiplicative version of Khitchine-Groshev theorem (in this setup) and \cite{GR} for an analog of Khitchine-Groshev theorem for imaginary quadratic extensions of function fields. Chao Ma and Wei-Yi Su proved the inhomogeneos version of Khinchinte theorem and also of Jarn\'{i}k–Besicovitch theorem \cite{Chao-Su}. In a later work, viz. \cite{Kr2}, S. Kristensen proved certain asymptotic formulae for the number of solutions to systems of inhomogeneous linear Diophantine inequalities from which inhomogeneous versions of Khintchine-Groshev and Jarn\'{i}k can be derived. The reader is referred to \cite{KN, Fuch2} for further developments in the inhomogeneous theory in positive characteristic.\\
    
   We now turn our attention to the metric theory over manifolds in this function field setting. The analogs of the conjectures of Mahler and Baker-Sprind\v{z}uk had been proved by V. Sprind\v{z}uk (\cite{Spr1}) and A. Ghosh (\cite{G-pos}) respectively. Furthermore, in the spirit of \cite{KLW}, the extremality of a large class of natural measures, namely the \emph{friendly} one, has been discussed in the very recent paper \cite{DaG2} by the authors. In \cite{G-JNT}, the second named author has proved the convergence case of Khintchine's theorem for affine hyperplanes along with a quantitative version of the same. There had been some advancement in the corresponding inhomogeneous theory as well, for instance, the inhomogeneous Baker-Sprind\v{z}uk conjecture has been proved by the second named author and A. Ghosh in \cite{GG-contemp}, and using the ``Inhomogeneous transference principle" of Beresnevich and Velani  (see \cite{BV, BBV}) it has been proved in \cite{DaG2} that \emph{contracting} extremal measures on the space of matrices are indeed inhomogeneously extremal. An upper bound for the inhomogeneous Diophantine exponent in the non-extremal case has also been provided in \cite{DaG2}. \\
   
 We would also like to take this occasion to make the reader aware that despite many results holding in both Euclidean and positive characteristic realms together, there are a few striking contrasts in place as well, such as the theory of badly approximable vectors are quite different, there is no analog of Roth's theorem, and much to one's surprise, the rational functions are precisely all elements for which the Dirichlet's theorem can be improved (\cite[Theorem 2.4]{GG1}), to name but a few.  

   % Beyond that, theorems of the above type have not been explored much to the best of authors' knowledge. Recently in \cite{G-JNT}, the second named author has proved the convergence case of Khintchine's theorem for affine hyperplanes, and furthermore, a quantitative version of the same has also been provided. However, like a missing link, the veracity of any of the Khinchine-Groshev type theorems, either convergence or divergence, for analytic \emph{nonplanar} (equivalently, nondegenerate) submanifolds in positive characteristic remains unconfirmed. \\  
\end{itemize}\vspace{0.2 cm}

To the best of author's knowledge, Khintchine type theorems have not been explored much on manifolds over function fields, except for hyperplanes in \cite{G-JNT}. The paper under consideration settles the complete analog of Khintchine-Groshev law in the generality of inhomogeneous framework for analytic \emph{nonplanar} (equivalently, nondegenerate) submanifolds. These extend the earlier results of \cite{MG1,MG2,BKM,BBKM,BBV,DG1} to positive characteristic. The quantitative nondivergence estimate and the inhomogeneous transference principle are the main ingredients to prove the convergence case. The divergence case is proved using techniques from another source \cite{BBV,BDV}, namely \emph{ubiquitous systems}. It is noteworthy to mention that the divergence case has been proved for general Hausdorff measures (see Theorem \ref{thm:inhomo-div}).

%Before we get into the actual content of this paper, we would like to bring the following to the notice of the reader. \\    
 
%\textbf{Organization of the paper}:  \S \ref{subsection:mainres} consists of the statements of the main results of this project. We begin \S \ref{section:calpha} by recalling some definitions and basic facts from ultrametric calculus, and after that some of the basic properties of $(C,\alpha)$-good functions from \cite{KT}. We prepare ourselves with the aim to establish Theorem \ref{thm:inhomo-conv} and \ref{thm:inhomo-div} in \S \ref{preparation}. Here we prove some technical results, concerning $(C,\alpha)$-good functions and skew gradient of certain types of functions, that are crucial to prove the convergence part of homogeneous Khintchine's theorem. \S \ref{section:estimate_big_grad} and \S \ref{section:estimate_small_grad} deals with estimating the measure of `large' gradient (Theorem \ref{thm:big-grad}) and `small' gradient sets (Theorem \ref{thm:small-grad}), which are key ingredients for proving the convergence part of homogeneous Khintchine's theorem. The objective of \S \ref{section:inhomoconv} is to settle the convergence case of inhomogeneous version of Khintchine's theorem (Theorem \ref{thm:inhomo-conv}).  The ``inhomogeneous transference principle" introduced in \cite[\S 5]{BV} has been instrumental in this endeavour. \S \ref{section:divg}, the final section, provides the proof of the divergence case of inhomogeneous Khintchine's theorem (Theorem \ref{thm:inhomo-div}) using ubiquitous systems.

\section{The setup and main results}\label{subsection:mainres}
We begin with  the function 
field $\mathbb{F}_{q}(X)$, where $q\defeq p^b,$ $p$ is a prime and $b\in \mathbb{N}$. One can define  a nontrivial, non-archimedean, and discrete absolute value  $|\cdot|$
on $\mathbb{F}_{q}(X)$ as follows: 
\[ |0|\defeq 0\,\,  \text{ and} \,\, \left|\frac{P}{Q}\right|\defeq q^{\displaystyle \deg P- \deg Q}
\text{ \,\,\,for all nonzero } P, Q\in \mathbb{F}_{q}[X]\,.\] 
The completion field of $\mathbb{F}_{q}(X)$ is $\mathbb{F}_{q}((X^{-1}))$, i.e., the field of Laurent series in $X^{-1}$ over the finite field $\mathbb{F}_q$. The absolute value of $\mathbb{F}_{q}((X^{-1}))$, which we again denote by $|\cdot |$, is given as follows. 
Let $a \in \mathbb{F}_{q}((X^{-1}))$. For $a=0$, 
define $|a|=0$. If $a \neq 0$, then we can write 
$$a=\displaystyle \sum_{k\leq k_{0}} a_k X^{k},\,\,\mbox{where}\,\,\,\,k_0 \in \mathbb{Z},\,a_k\in \mathbb{F}_{q}\,\,\mbox{and}\,\, a_{k_0}
\neq 0\,. $$

\noindent We define $k_0$ as the \textit{degree} of $a$, which will be denoted by $\deg a$,  and $|a|\defeq q^{\deg a}$. This clearly extends the absolute value $|\cdot|$ of $\mathbb{F}_{q}(X)$ to $\mathbb{F}_{q}((X^{-1}))$ and moreover, 
the extension remains non-archimedian and discrete. This makes $\mathbb{F}_{q}((X^{-1}))$ a complete and separable metric space which is ultrametric and hence totally disconnected. It is worth mentioning that any local field of positive characteristic happens to be isomorphic to some $\mathbb{F}_{q}((X^{-1}))$ (\cite[Part I, Chapter I, Theorem 8]{W}). \\

Let $\Lambda$ and $F$ 
denote $\mathbb{F}_{q}[X]$ and $\mathbb{F}_{q}((X^{-1}))$ respectively from now on. For any $n\in \mathbb{N}$, $F^n$ is throughout assumed to be equipped 
with the supremum norm which is defined as follows
\[||\mathbf{x}||\defeq \displaystyle \max_{1\leq i\leq n} |x_i|\text{ \,\,for all \,} \mathbf{x}=(x_1,\dots,x_n)\in F^{n}\,,\]

\noindent and with the topology induced by this norm. Clearly $\Lambda^n$ is discrete in $F^n$. Since the topology on $F^n$ considered here 
is the usual product topology on $F^n$, it follows that  $F^n$ is locally compact as $F$ is locally compact. As a custom, we denote the Haar measure on $F^n$ which takes the value 1 on the closed unit ball $ \{ \mathbf{x} \in F^n: \|\mathbf{x}\| \leq 1\} $ by $|\cdot |$,  although this is an abuse of notation. Let $\mathscr{O}=\mathbb{F}_{q}[[X^{-1}]]$ denotes the ring of all formal power series in $X^{-1}.$ Note that $\mathscr{O}=\{x \in \mathbb{F}_{q}((X^{-1})) : |x| \leq 1\}.$ We will use these notations throughout this article. It easy to check that for any $t\in \mathbb{N}$
\begin{equation}
    \# \{ \mathbf{a} \in \Lambda^n : \|\mathbf{a}\|=q^t\}= q^{(t+1)n}-q^{tn}=q^{tn}(q^n-1).
\end{equation}

 Let $ \Psi : \Lambda^n \rightarrow \mathbb{R}_+ $  be a function satisfying
\begin{equation}\label{equ:psi}
\Psi (a_1,\dots,a_i,\dots,a_n) \geq \Psi(b_1,\dots,b_i,\dots,b_n) \,\, \text{if} \,\, |a_i| \leq |b_i| \,\, \text{for all} \,\, i=1,\dots,n. 
\end{equation}
Such a function is called a \textit{multivariable approximating function}. Given a multivariable approximating function $\Psi$ and a function $\theta : F^n \rightarrow F$,  we say that a vector $ \mathbf{y} \in F^n $ is $(\Psi,\theta)$-approximable if there exist infinitely many $(\mathbf{a},a_0) \in \Lambda^n \setminus \{0\} \times \Lambda $ such that 
$$\left| \mathbf{y} \cdot \mathbf{a} + a_0 + \theta(\mathbf{y})\right| < \Psi(\mathbf{a}).$$
The set of all $(\Psi,\theta)$-approximable vectors of $F^n$ is denoted by $\mathcal W_n(\Psi,\theta).$ When $\theta\equiv 0$, we are brought to the homogeneous situation where simply call such vectors $\Psi$ approximable, and the set of all such vectors is denoted by $\mathcal W_n(\Psi)$. \\

Before going to our main results, we recall some notation, terminologies, and assumptions from \cite{MG1,MG2,BBV,DG1}. We will always take the domain $U$ to be an open box in $F^d$ with respect to the sup norm on $F^d$ and furthermore, we assume the following:
\begin{enumerate}[label=({{\Roman*}})]
       \item\label{II} We take analytic maps $f=(f_1,\dots,f_n): U \subseteq F^d \rightarrow F^n,$ which  can be analytically extended to the boundary of $U$. In view of the implicit function theorem for analytic manifolds over function fields (see chapter $2$ of \cite{WHS}), without loss in generality we may assume that $f_1(\mathbf{x})=x_1$
    \item\label{III} Let the restrictions of $1,f_1,\dots,f_n$ to any open subset of $U$ are linearly independent over $F.$
    
    \item\label{V} We always assume $\Psi : \Lambda^n \rightarrow \mathbb R_+$ to be a \emph{multivariable approximating function}. 
    \suspend{enumerate}\vspace{0.1 cm}
Without loss in generality, for the sake of simplification, we further make some boundedness assumptions: \vspace{0.1 cm}
   \resume{enumerate}[{[label=({{\Roman*}})]}]
    \item\label{IV} $\|f(\mathbf{x})\| \leq 1,$ $\|\nabla f(\mathbf{x})\| \leq 1,$ and $|\bar\Phi_{\beta}f(\mathbf{y_1},\mathbf{y_2},\mathbf{y_3})|\leq 1$ for any second difference quotient $\Phi_{\beta}$ and $\mathbf{x},\mathbf{y_1},\mathbf{y_2},\mathbf{y_3} \in U$ (see \S \ref{section:calpha} for the definitions).
    \item\label{VI} $\theta : U \subseteq F^d \rightarrow F$ is analytic and can be extended analytically to the boundary of $U.$ We also assume $|\theta(\mathbf{x})| \leq 1,$ $\|\nabla \theta(\mathbf{x})\| \leq 1$ and $|{\bar\Phi_{\beta}}\theta(\mathbf{y_1},\mathbf{y_2},\mathbf{y_3})|\leq 1$ for any second difference quotient $\Phi_{\beta}$ and $\mathbf{x},\mathbf{y_1},\mathbf{y_2},\mathbf{y_3} \in U.$ 
    \end{enumerate}
\subsection{Main results}
For a subset $X$ of $F^d$ and $s>0,$ $\mathscr{H}^s(X)$  denotes  the $s$-dimensional Hausdorff measure of  $X$. Now we are ready to state the main results of this article. 

%\begin{theorem}\label{thm:conv} Suppose $U$ is an open subset of $F^d$ and $f: U  \rightarrow F^n$ satisfies (\ref{II}), (\ref{III}) and (\ref{IV}). Let $\Psi$ be a multivariable approximating function, then 
%$$\mathcal W_f(\Psi) \defeq \{\mathbf{x} \in U : f(\mathbf{x}) \,\, is \,\, \Psi-approximable \} $$
%has measure zero if $\displaystyle \sum_{\mathbf{a} \in \Lambda^n  \setminus \{0\} }  \Psi (\mathbf{a}) < \infty.$
%\end{theorem}

\begin{theorem}\label{thm:inhomo-conv} Suppose $U$ is an open subset of $F^d$, $f: U  \rightarrow F^n$ satisfies \ref{II}, \ref{III}, \ref{IV} and $\theta : U \rightarrow F$ satisfies \ref{VI}.  Let $\Psi$ be a multivariable approximating function and     
$$\mathcal W_f(\Psi,\theta) \defeq \{\mathbf{x} \in U : f(\mathbf{x}) \,\, is \,\, (\Psi,\theta)-approximable \}. $$ Then $\mathcal W_f(\Psi,\theta)$
has measure zero if  $\displaystyle \sum_{\mathbf{a} \in \Lambda^n  \setminus \{0\} }  \Psi (\mathbf{a}) < \infty.$
\end{theorem}
 The divergence counterpart of the Khintchine-Groshev theorem is proved in the generality of the Hausdorff measure: 
 
\begin{theorem}\label{thm:inhomo-div} 
 Suppose $U$ is an open subset of $F^d$, $f: U  \rightarrow F^n$ satisfies \ref{II}, \ref{III}, \ref{IV} and $\theta : U \rightarrow F$ satisfies \ref{VI}. Let $\Psi(\mathbf{a}) \defeq \psi(\|\mathbf{a}\|),$ where $\psi$ is a decreasing positive function. Then for $s>d-1$, we have $$\mathscr{H}^s(\mathcal W_f(\Psi,\theta) \cap U)=\mathscr{H}^s(U) \,\,\,\, \text{if} \,\, \displaystyle \sum_{\mathbf{a} \in \Lambda^n  \setminus \{0\} }  \|\mathbf{a}\| \left(\frac{\Psi (\mathbf{a})}{\|\mathbf{a}\|}\right)^{s+1-d} = \infty. $$In particular, when $s=d$, $\mathcal W_f(\Psi,\theta) $ has full measure if $\displaystyle \sum_{\mathbf{a} \in \Lambda^n  \setminus \{0\} } \Psi (\mathbf{a}) = \infty $.
\end{theorem}
Now combining Theorem \ref{thm:inhomo-conv} and \ref{thm:inhomo-div}, one obtains the following corollary.
\begin{corollary}[Complete Khintchine-Groshev theorem for manifolds in positive characteristic]
     Suppose $U$ is an open subset of $F^d$ and $f: U  \rightarrow F^n$ satisfies \ref{II}, \ref{III}, and \ref{IV}. Let $\Psi(\mathbf{a}) \defeq \psi(\|\mathbf{a}\|)$, where $\psi$ is a decreasing positive function. Then
 \[\mathcal W_f(\Psi)=\begin{cases}
\text{null}  & \text{ if }\displaystyle \sum_{\mathbf{a} \in \Lambda^n  \setminus \{0\} }  \Psi (\mathbf{a}) < \infty\\
\text{co-null}& \text{ if }\,\,\displaystyle \sum_{\mathbf{a} \in \Lambda^n  \setminus \{0\} }  \Psi (\mathbf{a}) = \infty.
\end{cases}\]
\end{corollary}

\begin{example}$\ $
    \begin{enumerate}[label=({{\arabic*}})]
        \item\label{exam:1} Let $U=B(0;1/q^4) \subset F$ be the ball of radius $1/q^{4}$ containing the point $0,$ $f_0:U\rightarrow F^2$ be given by $f_0(x)=(x,x^2)$ and $\theta_0: U \rightarrow F$ is defined by $\theta_0(x)=x^4$. First of all note that $f_0,\theta_0$ satisfies the hypothesis of Theorem \ref{thm:inhomo-conv} and \ref{thm:inhomo-div}. Let $\Psi_1, \Psi_2 : \Lambda^2 \rightarrow \mathbb R_+$ be given by $\Psi_1(\|\mathbf{a}\|)=\frac{1}{\|\mathbf{a}\|}$ and $\Psi_2(\|\mathbf{a}\|)=\frac{1}{\|\mathbf{a}\|^2}.$ Then by applying Theorem \ref{thm:inhomo-conv} and \ref{thm:inhomo-div}, we get that $\mathcal W_{f_0}(\Psi_1,\theta_0)$ is a co-null set and $\mathcal W_{f_0}(\Psi_2,\theta_0)$ is a null set.\vspace{0.2 cm}
         \item More generally, consider $U=B(0;1) \subset F$ to be the ball of radius $1$ containing the point $0$ and $f_1:U \rightarrow F^n$ be given by $f_1(x)=(p_1(x),\dots,p_n(x)),$ where $p_1,\dots,p_n$ are polynomials in $x$ such that they satisfies the hypothesis of Theorem \ref{thm:inhomo-conv} and \ref{thm:inhomo-div}. Note that if the polynomials $p_1,\dots,p_n$ satisfy \ref{II} and \ref{III}, then by rescaling and normalizing $p_i$'s, we can make $f_1(x)=(p_1(x),\dots,p_n(x))$ satisfy \ref{IV}. Similarly if we take $\theta_1 : U \rightarrow F$ to be $\theta_1(x)=p(x)$ for some polynomial $p,$ then we can also make $\theta_1$ to satisfy \ref{VI}. The upshot is that $f_1,\theta_1$ satisfies the hypothesis of Theorem \ref{thm:inhomo-conv} and \ref{thm:inhomo-div}. Let $\Psi_1$ and $\Psi_2$ be as defined in example \ref{exam:1}. Then by Theorem \ref{thm:inhomo-conv} and \ref{thm:inhomo-div}, we have $\mathcal W_{f_1}(\Psi_1,\theta_1)$ is a co-null set and $\mathcal W_{f_1}(\Psi_2,\theta_1)$ is a null set.
    \end{enumerate}
\end{example}
\begin{remark}$\ $
\begin{enumerate}[label=({{\arabic*}})]
    \item The proof of Theorem \ref{thm:inhomo-div} can be further modified as it deems necessary to deal with the more general case of multivariable approximating functions satisfying a convexity condition, which has been introduced and referred to as the property $\mathbf{P}$ in \cite[\S 1]{BBV}. Let   $\mathbf{v}=(v_1,\dots,v_n)$ be an $n$-tuple of positive numbers with $v_1+\dots+v_n=n$. We set
        $\|\mathbf{x}\|_{\mathbf{v}} \defeq \max_{1 \leq i \leq n} |x_i|^{1/v_i}$. The function  $\|\cdot\|_{\mathbf{v}}$  thus defined is called the $\mathbf{v}$-quasinorm on $F^n$.  This clearly generalises the supremum norm $\|\mathbf{x}\|$, for  $\|\mathbf{x}\|_{\mathbf{v}}=\|\mathbf{x}\|$,  when $\mathbf{v}=(1,\dots,1)$. 
    We say that a multivariable approximating function $\Psi$ satisfies property $\mathbf{P}$ if $\Psi(\mathbf{a})=\psi(\|\mathbf{a}\|_{\mathbf{v}})$, for some approximating function $\psi$ (i.e., $\psi$ is decreasing) and $\mathbf{v}$ as above. Any multivariable approximating function $\mathbf{a} \mapsto \psi(\|\mathbf{a}\|)$ obviously satisfies property $\mathbf{P}$. We leave it to the interested reader to make appropriate modifications of the arguments provided in \S \ref{section:divg} to establish the following general statement:\\
    
    \emph{Let $U, f, s, \theta$ be as in the hypothesis of Theorem \ref{thm:inhomo-div}. Assume that $\Psi$ is a multivariable approximating function that satisfies property $\mathbf{P}$. Then one has 
    \[ \mathscr{H}^s(\mathcal W_f(\Psi,\theta) \cap U)=\mathscr{H}^s(U) \text{ if } \,\, \displaystyle \sum_{\mathbf{a} \in \Lambda^n  \setminus \{0\} } \|\mathbf{a}\| \left(\frac{\Psi (\mathbf{a})}{\|\mathbf{a}\|}\right)^{s+1-d} = \infty.\]  }
    
    \item  It seems to be an interesting direction to explore whether or not a Convergence Khintchine-Groshev type theorem can be settled for a more general class of measures, for example, friendly measures or under some more modest assumptions, etc. To the best of our knowledge, it remains unanswered to date. 
\end{enumerate}
\end{remark}
\subsection{Strategy and outline of the proofs}\label{stra}
To prove Theorem \ref{thm:inhomo-conv}, first note that
\begin{equation*}
    \mathcal W_f(\Psi,\theta)= \displaystyle \limsup_{\mathbf{a} \rightarrow \infty} \mathcal W_f(\mathbf{a},\Psi,\theta),
\end{equation*}
where for $\mathbf{a} \in \Lambda^n,$
\begin{equation*}
    \mathcal W_f(\mathbf{a},\Psi,\theta)\defeq \{ \mathbf{x} \in U : |a_0 + \mathbf{a} \cdot f(\mathbf{x})+ \theta(\mathbf{x})| < \Psi(\mathbf{a}) \,\, \text{for some } \,\, a_0 \in \Lambda \}.
\end{equation*}
In the next step, we decompose the set $\mathcal W_f(\mathbf{a},\Psi,\theta)$ into two subsets $\mathcal W_f^{\text{big}}(\mathbf{a},\Psi,\theta)$ and $\mathcal W_f^{\text{small}}(\mathbf{a},\Psi,\theta)$ (given by (\ref{equ:w_large_f_a}) and (\ref{equ:w_small_f_a}) respectively) based on the size of the gradient $\nabla(\mathbf{a} \cdot f(\mathbf{x})+ \theta(\mathbf{x}))$. Hence we can express $\mathcal W_f(\Psi,\theta)$ as the union of two lim-sup sets $\mathcal W_f^{\text{big}}(\Psi,\theta)$ and $\mathcal W_f^{\text{small}}(\Psi,\theta)$.  Now we separately show that these two sets have measure zero. To this end, we first look at the homogeneous counterpart of these sets. The corresponding homogeneous `big' gradient part is easy to estimate using calculus on ultrametric spaces and the analogous estimate for the `small' gradient part is proved using the famous `quantitative non-divergence' technique. Next, from the homogeneous big gradient estimate, with the help of the Borel-Cantelli lemma, we show that $|\mathcal W_f^{\text{big}}(\Psi,\theta)|=0$. Finally combining the inhomogeneous transference principle and the estimate obtained previously for homogeneous small gradient one shows $\mathcal W_f^{\text{small}}(\Psi,\theta)$ has zero measure. This completes the proof of the convergence case. \\

To prove the divergence case, i.e., Theorem \ref{thm:inhomo-div}, we use the technique of ubiquitous systems. In fact, the generality of this technique yields the divergence counterpart for general Hausdorff measures. 
\subsection{A brief summary of the following sections} In \S \ref{section:calpha}, we recall some definitions and basic facts from ultrametric calculus, and after that few of the basic properties of $(C,\alpha)$-good functions. We prepare ourselves intending to establish Theorem \ref{thm:inhomo-conv} and \ref{thm:inhomo-div} in \S \ref{preparation}. Here we prove some technical results, concerning $(C,\alpha)$-good functions and skew gradient of certain types of functions. \S \ref{section:estimate_big_grad} and \S \ref{section:estimate_small_grad} deals with estimating the measure of `big' gradient (Theorem \ref{thm:big-grad}) and `small' gradient sets (Theorem \ref{thm:small-grad}), which are key ingredients for proving Theorem \ref{thm:inhomo-conv}. The objective of \S \ref{section:inhomoconv} is to settle the convergence case of the inhomogeneous version of Khintchine's theorem (Theorem \ref{thm:inhomo-conv}).  The ``inhomogeneous transference principle" introduced in \cite[\S 5]{BV} has been instrumental in this endeavor. \S \ref{section:divg}, the final section, provides the proof of the divergence case of inhomogeneous Khintchine's theorem (Theorem \ref{thm:inhomo-div}) using ubiquitous systems.

\subsection*{Acknowledgement}  The authors are deeply grateful to the Department of Mathematics and Statistics at the Indian Institute of Technology Kanpur for providing a friendly environment conducive to research. The authors would also like to thank the anonymous referees for the careful reading of the paper, detailed reports, and numerous valuable suggestions on its presentation, which have significantly improved the clarity and quality of the exposition.

\section{Ultametric calculus and $ (C,\alpha) $-good functions}\label{section:calpha}
 \subsection{Calculus of functions on local fields} Now we recall, following \cite{WHS}, the notion of $C^k$ functions in the ultrametric case. We consider the local field ${F}$. Let $U$ be an open subset of ${F}$ and $g: U \rightarrow {F}$ be a function.  Let 
 $$\nabla^k U \defeq \{(x_1,\dots,x_k) \in U^k \,\, |\,\, x_i \neq x_j \,\, for \,\, i \neq j\},$$
 and define the $k$-th order difference quotient $\Phi^k g: \nabla^{k+1} U \rightarrow {F}$ of $g$ inductively by $\Phi^0 g=g$ and \\
 $$\Phi^k g(x_1,x_2,\dots,x_{k+1})\defeq\frac{\Phi^{k-1}g(x_1,x_3,\dots,x_{k+1})-\Phi^{k-1}g(x_2,x_3,\dots,x_{k+1})}{x_1 - x_2}.$$\\
 It is clear that $\Phi^k g$ is a symmetric function of its $k+1$ variables. Now $g$ is called $C^k$ at $a \in U$ if the following limit exists:
 $$\lim_{(x_1,\dots,x_{k+1}) \rightarrow (a,\dots,a)} \Phi^k g(x_1,\dots,x_{k+1}),$$ 
 and $g$ is called $C^k$ on $U$ if it is $C^k$ on every point $a \in U.$ This is equivalent to $\Phi^k g$ being extendable to $\bar\Phi^k g : U^{k+1} \rightarrow {F}.$ This extension, if it exists, is indeed unique. The $C^k$ functions are $k$ times differentiable, and 
 $$g^{(k)}(x)=k! \bar\Phi^k g(x,\dots,x).$$

 Let now $g$ be an $ F$-valued function of several variables defined on $U_1 \times \dots \times U_d,$ where each $U_i$ is an open subset of $ F.$ Denote by $\Phi_i^k g$ the $k^{th}$ order difference quotient of $g$ with respect to the $i^{th}$ coordinate. Then for any multi-index $\beta =(i_1,\dots,i_d)$ let
 $$\Phi_{\beta}g \defeq \Phi_1^{i_1} \circ \dots \circ \Phi_d^{i_d} g.$$
 It is defined on $\nabla^{i_1 +1}U_1 \times \dots \times \nabla^{i_d +1} U_d$. The function $g$ is called $C^k$ on $U_1 \times \dots \times U_d$ if for any multi-index $\beta$ with $|\beta|=\sum_{j=1}^{d} i_j$ at most $k,$ $\Phi_{\beta}g$ is extendable to a continuous function $\bar\Phi_{\beta}g: U_1^{i_1 +1} \times \dots \times U_d^{i_d +1} \rightarrow F.$ Similar to the one-variable scenario, we can show that partial derivatives $\partial_{\beta}g \defeq \partial_{1}^{i_1} \circ \dots \circ \partial_{d}^{i_d}$ of a $C^k$ function $g$ exists and are continuous as long as $|\beta | \leq k.$ Moreover, we have
 \begin{equation}\label{equ:part_beta_g}
     \partial_{\beta}g(x_1,\dots,x_d)= \beta! \bar \Phi_{\beta} g(x_1,\dots,x_1,\dots,x_d,\dots,x_d),
 \end{equation}
 where $\beta!=\prod_{j=1}^{d}i_j !$ and each of the variables $x_j$ in the right-hand side of \eqref{equ:part_beta_g} is repeated $i_j+1$ times.

 If $g=(g_1,\dots,g_n):F^d \rightarrow F^n$ is a $C^1$ map, then by $\nabla g(\mathbf{x})$ we denote the $d \times n$ matrix whose $(i,j)$ th entry is $\partial_j g_i(\mathbf{x})$ and $\|\nabla g(\cdot)\|$ denotes the max norm.\\
  
 \noindent \textbf{Notation:} Throughout this article, for $\mathbf{x} \in F^d$ and $r>0,$ $B(\mathbf{x};r)$ stands for the open ball $\{\mathbf{y} \in F^d:\|\mathbf{y}-\mathbf{x}\|<r\}$ in $F^d$ of radius $r$ containing the point $\mathbf{x}$. As $F^d$ is ultrametric,  every point of an open ball can be regarded as centre of the same. Indeed, a ball is uniquely determined by any point in it and the radius. Hence we refrain from referring to centre of a ball. If $B=B(\mathbf{x};r)$ and $\tau>0$ be given, by $\tau B$ we mean the ball $B(\mathbf{x};\tau r)$. Also note that any two balls in $F^d$, (or more generally in an ultrametric space) are either disjoint or one contains the other. 
\subsection{$(C,\alpha)$-good functions and their properties}
In this section, we recall some terminology  
introduced in the papers of Kleinbock and Margulis, and Kleinbock, Lindenstrauss, and Weiss and used in several subsequent works by many 
authors.

%The following is taken from  \S 1 and 2 of \cite{KT}. 

%\begin{definition}\label{defn:besicovitch}Let  $X$ be a metric space. We will say $X$ is \emph{Besicovitch} if there exists a constant $N_X$ such that  for any bounded subset $A$ of $X$ and for any collection $\mathcal{C}$ of nonempty open balls in $X$ such that every $a\in A$ is the center of some ball in $\mathcal{C}$, there exists a countable subcollection $\{B_i\}$ of $\mathcal{C}$ with\[1_A\leq \displaystyle \sum_i 1_{B_i} \leq N_X\,.\]\end{definition}Standard examples of Besicovitch spaces are the Euclidean spaces $\mathbb{R}^n$ (see Theorem 2.7, \cite{Mat}), $F^n$ with metric induced by supremum norm as defined in Section \ref{section:intro}. In fact, the constant $N_{F^n}=1$ as this is an ultrametric space. In other words, any covering by balls of any bounded subset of $F^n$ admits a countable subcover consisting of pairwise mutually disjoint balls.\\

 Suppose  $U\subseteq F^d$ is open, $\nu$ is a Radon measure
on $F^d$, and $f: F^d \longrightarrow {F}$ is a given function such that $|f|$ is measurable. For any $B\subseteq F^d$, we set
$$ ||f||_{\nu,B} \defeq \displaystyle \sup_{x\in B\cap \text{ supp }(\nu)} |f(x)|.$$

\begin{definition}\label{defn:C,alpha}
 For $C,\alpha >0$, $f$ is said to be $(C,\alpha)$-\emph{good} on $U$ with respect to $\nu$ if for every ball $B\subseteq U$ containing some point of $\text{supp }(\nu)$, one has
 \[\nu(\{x\in B: |f(x)| < \varepsilon\})\leq C\left(\frac{\varepsilon}{||f||_{\nu,B}}\right)^{\alpha} \nu(B)\,.\]
\end{definition}
The following properties are immediate from Definition \ref{defn:C,alpha}. 
\begin{lemma} \label{lem:C,alpha} Let $U,\nu, f, C,\alpha$ be as given above. Then one has 
\begin{enumerate}[label=({{\roman*}})]
   \item $f$ \text{ is } $(C,\alpha)$-good \text{ on  }$U$ with respect to $\nu \Longleftrightarrow \text{ so is } |f|$.
   \item $f$ is $(C,\alpha)$-good on $U$  with respect to $\nu$ $\Longrightarrow$ so is $c f$ for all $c \in {F}$.
   \item \label{item:sup} Let $I$ be an index set and for each $i \in I$, $f_i : U \rightarrow F$ be functions such that $|f_i|$ is measurable. For each $ i\in I, f_i$ are $(C,\alpha)$-good on $U$ with respect to $\nu$  $\Longrightarrow$ so is $\sup_{i\in I} |f_i|$.
   \item $f$ is $(C,\alpha)$-good on $U$ with respect to $\nu$ and $g :U \longrightarrow \mathbb{R}$ is a continuous function such 
   that $c_1\leq |\frac{f}{g}|\leq c_2$ 
 for some $c_1,c_2 > 0\Longrightarrow g$ is $(C(\frac{c_2}{c_1})^{\alpha},\alpha)$-good on $U$ with respect to $\nu$.
 \item Let $C_2 >1$ and 
 $\alpha _2>0$. $f$ is $(C_1,\alpha_1)$-good on $U$ with respect to $\nu$, $C_1 \leq C_2$ and $ \alpha_2 \leq \alpha_1 \Longrightarrow f$ is
 $(C_2,\alpha_2)$-good on $U$ with respect to $\nu$.
  \end{enumerate}
\end{lemma}  
 We say a  map $\mathbf{f}=(f_1,f_2,...,f_n)$ from $U$ to ${F}^n$, where $n\in \mathbb{N}$,
is $(C,\alpha)$-good on $U$ with respect to $\nu$, or simply $(\mathbf{f},\nu)$ is 
$(C,\alpha)$-good on $U$,  if every ${F}$-linear combination of $1,f_1,...,f_n$ 
is $(C,\alpha)$-good on $U$ with respect to $\nu$. \\

The following is a particular case of \cite[Lemma 3.4]{KT}:
\begin{lemma}\label{lem:poly-C,alpha}
Any polynomial $f \in F[x_1,\dots,x_m]$ of degree not greater than $k$ is $(C,1/mk)$-good on $F^m$ with respect to natural Haar measure of $F^m,$ where $C$ is a constant depending only on $m$ and $k.$
\end{lemma}
\begin{note}
 Using a standard argument \cite[Lemma 1.4]{PS}, it is easy to show that the ultrametric space $F^d,$  has the following nice covering property. Any covering by balls of any bounded subset of $F^d$ admits a countable subcover consisting of pairwise mutually disjoint balls.
\end{note}
%Following \cite{MG1}, we can define the notion of \emph{orthogonality} on $F^m$. 
%\begin{definition}
 %A set of vectors $x_1,\dots,x_n$ in $F^m,$ is called orthonormal if $\|x_1\|=\dots=\|x_n\|=\|x_1 \wedge \dots \wedge x_n\|=1.$ 
%\end{definition}

\section{Preparation for the proof of Theorem \ref{thm:inhomo-conv} and \ref{thm:inhomo-div}}\label{preparation}
In this section, we prove a few technical results providing sufficient conditions for certain functions to be $(C,\alpha)$-good, which will be useful in the proof of Theorem \ref{thm:inhomo-conv} and \ref{thm:inhomo-div}. Let us first recall the following Theorem from \cite[Theorem 3.2]{KT}, which is a useful criterion for checking certain functions to be $(C,\alpha)$-good.
\begin{theorem}\label{thm:C_alpha_criteria}
    Let $V_1,\dots,V_d$ be non-empty open subsets of $F.$ Let $k \in \mathbb N$, $c_1,\dots,c_d>0$ and $f \in C^k(V_1 \times \dots \times V_d)$ be such that
    \begin{equation}\label{equ:C_alpha_criteria}
        |\Phi_{j}^{k}f| \equiv c_j \,\, \text{on}\,\, \nabla^{k+1}V_j \times \prod_{i\neq j}V_i, \,\, j=1,\dots,d.
    \end{equation}
    Then $f$ is $(dk^{3-1/k},1/dk)$-good on $V_1 \times \dots \times V_d$.
\end{theorem}
\noindent Now we prove the  following function field  analogue of \cite[Proposition 1]{BBV} and \cite[Proposition 5.1]{DG1}:
\begin{proposition}\label{prop:grad_C}
Let $U \subseteq F^d$ be  open, $\mathbf{x}_0 \in U $ and let $\mathcal{H} \subseteq C^l(U)$ be a compact family of functions $h:U \rightarrow F$ for some $l \geq 2.$  Also suppose that 
\begin{equation}\label{equ:h_inf_max_greater_0}
    \inf_{h \in \mathcal{H}} \left(\max_{0< |\beta| \leq l} |\partial_{\beta} h(\mathbf{x}_0)|\right) > 0.
\end{equation}
Then there exists a neighbourhood $V \subseteq U$ of $\mathbf{x}_0$ and constants $C,\delta>0$ satisfying the following property. For any $\theta \in C^l(U)$ such that 
\begin{equation}\label{equ:theta_sup_max_less_delta}
    \sup_{\mathbf{x} \in U}\left( \max_{0< |\beta| \leq l} |\partial_{\beta} \theta(\mathbf{x})| \right)\leq \delta
\end{equation}
and any $h \in \mathcal H$ we have the following
\begin{enumerate}[label=({{\roman*}})]
    \item $h + \theta$ is $(C,\frac{1}{dl})$-good on $V,$ 
    \item $\|\nabla (h + \theta )\|$ is $(C,\frac{1}{d(l-1)})$-good on $V.$
\end{enumerate}
\end{proposition} 
\begin{proof}
    Following the approach of \cite[Proposition 1]{BBV} and \cite[Proposition 5.1]{DG1}, we give a brief sketch of the proof. In view of \eqref{equ:h_inf_max_greater_0}, there exists a constant $c_1>0$ such that for any $h \in \mathcal{H}$ one can find a multiindex $\beta$ with $0 < |\beta|=k \leq l,$ where $k=k(h),$ such that
\begin{equation}\label{equ:partial_beta_h_greater_c_1}
|\partial_{\beta} h(\mathbf{x}_0)| \geq c_1.
\end{equation}
Since there are finitely many $\beta,$ without loss in generality we may assume that $\beta$ appearing in \eqref{equ:partial_beta_h_greater_c_1} is the same for all $h \in \mathcal H.$ We want to apply Theorem \ref{thm:C_alpha_criteria} to prove our results. To do so we first show that $h+\theta$ satisfies the hypothesis of Theorem \ref{thm:C_alpha_criteria}. Now we aim to show that there exists $A \in \GL_d(\mathscr O)$ such that $(h+ \theta) \circ A$ satisfies \eqref{equ:C_alpha_criteria}. It is not very difficult to see that for $h \in \mathcal H$ there exists $A_h \in \GL_d(\mathscr O)$ and $c>0$ such that
\begin{equation}\label{equ:min_partial_h_A_h}
    \min_{i=1,\dots,d} |\partial_i^{k}(h \circ A_h )(A_h^{-1}(\mathbf{x}_0))| >c.
\end{equation}
Consider the map $\eta : \GL_d(F) \times C^l(U) \times U \rightarrow F $ defined by
\begin{equation*}
    \eta(A,h,\mathbf{x})=\min_{i=1,\dots,d} |\partial_i^{k}(h \circ A) (A^{-1}(\mathbf{x}))|.
\end{equation*}
 Using \eqref{equ:min_partial_h_A_h} and the continuity of $\eta,$ we get an open neighbourhood $U_{A_h} \times U_h \times U_{(\mathbf{x}_0,h)}$  of $(A_h,h,\mathbf{x}_0)$ such that 
 \begin{equation*}
    \eta(A,f,\mathbf{x})>c \,\, \,\, \text{for all} \,\, (A,f,\mathbf{x}) \in U_{A_h} \times U_h \times U_{(\mathbf{x}_0,h)}.
 \end{equation*}
In particular, 
\begin{equation}\label{equ:eta_A_f_x_greater_c}
\eta(A_h,f,\mathbf{x})>c \,\, \,\, \text{for all} \,\, f \in U_{h} \,\, \text{and} \,\, \mathbf{x} \in U_{(\mathbf{x}_0,h)}.
\end{equation}
$\{U_h\}_{h \in \mathcal H}$ is an open cover of $\mathcal{H},$ hence it must have a finite subcover $\{U_{h_i}\}_{i=1}^{r}.$ By \eqref{equ:eta_A_f_x_greater_c} we have that for every $\mathbf{x} \in U_{\mathbf{x}_0}=\cap_{i=1}^{r} U_{(\mathbf{x}_0,h_i)}$ and $h \in \mathcal H,$ there exists $i \in \{1,\dots,r\}$ such that
\begin{equation}
\eta(A_{h_i},h,\mathbf{x})>c.
\end{equation}
By \eqref{equ:theta_sup_max_less_delta}, there exists a constant $c_0>0$ such that 
\begin{equation}
|\partial_i^k (\theta \circ A)(A^{-1}\mathbf x)| \leq c_0\delta 
\end{equation}
for some $A \in \GL_d(\mathscr O).$ Thus for any $\theta$ satisfying \eqref{equ:theta_sup_max_less_delta}, we have
\begin{equation}
    \eta(A_{h_i},h+\theta,\mathbf x) >c' \,\, \forall \,\, \mathbf{x} \in U_{\mathbf{x}_0},
\end{equation}
    for some constant $c'>0.$ Also by the compactness of the family $\mathcal H$ and \eqref{equ:theta_sup_max_less_delta} there is an uniform upper bound for every $h \in \mathcal H$ and $\theta$ of the above type. Now by applying Theorem \ref{thm:C_alpha_criteria}, we get that $(h+\theta) \circ A_{h_i}$ is $(dk^{3-1/k},1/dk)$-good on $A_{h_i}^{-1}U_{\mathbf x_0}.$ Hence $h+\theta $ is $(dk^{3-1/k},1/dk)$-good on $U_{\mathbf x_0}.$ This completes the proof of part $(i).$

    For the second part consider the compact sets $\mathcal{H}_{A_{h_i}}=\{h \in \mathcal H:\eta(A_{h_i},h,\mathbf{x}_0) \geq c\}.$ Hence $\{\partial_j(h \circ A_{h_i}):h \in \mathcal{H}_{A_{h_i}} \}$ is a compact family of functions, being a continuous image of a compact set. Without loss in generality we can take the same $A$ for each $h \in \mathcal H,$ since $\mathcal H \subseteq \cup_{i=1}^{r} \mathcal H_{A_{h_i}}.$ Suppose $|\beta|>1$ in \eqref{equ:partial_beta_h_greater_c_1}. Now we apply part $(i)$ to the compact family of functions $\{\partial_j(h\circ A):h \in \mathcal H \}$ of $C^{(l-1)}(U).$ Clearly this family satisfies \eqref{equ:h_inf_max_greater_0}. Hence by part $(i),$ $\partial_j((h+\theta) \circ A)$ is $(C,\frac{1}{d(l-1)})$-good on an open neighbourhood of $A^{-1}(\mathbf x_0),$ for each $j=1,\dots,d.$ Therefore for each $j=1,\dots,d$, $\partial_j(h + \theta)$ is $(C,\frac{1}{d(l-1)})$-good on an open neighbourhood of $\mathbf{x}_0$ and so in $\|\nabla(h +\theta)\|,$ by using Lemma \ref{lem:C,alpha}. The case $|\beta|=1$ in \eqref{equ:partial_beta_h_greater_c_1} is trivial. This completes the proof of part $(ii).$
\end{proof}
 
%As a corollary, we have
\begin{corollary}\label{coro:c-alpha-grad}
Let $U \subseteq F^d$ be  open, $\mathbf{x}_0 \in U$ be fixed and assume that $\mathbf h=(h_1,h_2,\dots,h_n) : U \rightarrow F^n$ satisfies \ref{II}, \ref{III} and \ref{IV} and that $\theta$ satisfies \ref{VI}. Then there exists a neighbourhood $V \subseteq U$ of $\mathbf{x}_0$ and  constants $C,M_0 >0$ and $l \in \mathbb N$ such that for any $(\mathbf{a},a_0) \in F^{n+1}$ with $\|\mathbf{a}\| \geq M_0$
\begin{enumerate}[label=({{\arabic*}})]
    \item $a_0 + \mathbf{a} \cdot \mathbf h + \theta $ is $(C,\frac{1}{dl})$-good on $V,$ and
    \item $\|\nabla (\mathbf{a} \cdot \mathbf h + \theta)\| $ is $(C,\frac{1}{d(l-1)})$-good on $V.$
\end{enumerate}
\end{corollary}
%The proof of this is similar to the proof of Corollary 3 of \cite{BBV} and Corollary 5.1. of \cite{DG1}, hence we skip it here. As a consequence of this and by following arguments in the proof of Corollary 3 of \cite{BBV}, we get the following corollary:
\begin{corollary}\label{coro:c-alpha-grad-homo}
Let $U \subseteq F^d$ be  open, $\mathbf{x}_0 \in U$ be fixed and assume that $\mathbf h=(h_1,h_2,\dots,h_n) : U \rightarrow F^n$ satisfies \ref{II}, \ref{III} and \ref{IV}. Then there exists a neighbourhood $V \subseteq U$ of $\mathbf{x}_0,$   constants $C>0$ and $l \in \mathbb N$ such that for any $(\mathbf{a},a_0) \in F^{n+1}$ 
\begin{enumerate}[label=({{\arabic*}})]
    \item $a_0 + \mathbf{a} \cdot \mathbf h  $ is $(C,\frac{1}{dl})$-good on $V,$ and
    \item $\|\nabla (\mathbf{a} \cdot \mathbf h) \| $ is $(C,\frac{1}{d(l-1)})$-good on $V.$
\end{enumerate}
\end{corollary}
The proof of Corollary \ref{coro:c-alpha-grad} is exactly the same as that of \cite[Corollary 3]{BBV} or \cite[Corollary 5.1]{DG1}; and Corollary \ref{coro:c-alpha-grad-homo} can be easily proved using Corollary \ref{coro:c-alpha-grad} and proceeding along the same line with the proof of \cite[Corollary 3]{BBV}. Hence, leaving the details to the interested reader, we omit the proofs here. \\

We also mention now another result which will be useful to us later. The following can be regarded as the ultrametric analog of  \cite[Corollary 4]{BBV}, and furthermore, its proof is also adaptive to the function field setting and hence omitted. 
\begin{proposition}\label{prop:non_zero_c_alpha}
Let $U,\mathbf{x}_0,\mathbf h$ and $\theta$ be as in Corollary \ref{coro:c-alpha-grad}. Then for all sufficiently small neighbourhood of $V \subseteq U$ of $\mathbf{x}_0,$ there exists $M_0 >1$ such that
\begin{equation*}
    \inf_{\tiny{\begin{array}{cc}
      (\mathbf{a},a_0) \in F^{n+1} \\  \|\mathbf{a}\| \geq M_0
      \end{array}}} \left(\sup_{\mathbf{x} \in V}  |a_0+\mathbf{a} \cdot \mathbf h(\mathbf{x})+\theta(\mathbf{x})|\right)>0.
\end{equation*}
\end{proposition}
The following theorem is the main result of this section. We will use it in the proof of Theorem \ref{thm:small-grad}. We first recall the notion of $\textit{skew gradient}$ from \cite{BKM} and \cite{MG1}. For two $C^1$ functions $g_i : F^d \rightarrow F ,$ $i=1,2,$ define $\widetilde{\nabla}(g_1,g_2)\defeq g_1 \nabla g_2 - g_2 \nabla g_1.$ The following theorem is the function field analog of Theorem $4.7$ from \cite{MG1}.
\begin{theorem}\label{thm:skew-calpha}
Let $U$ be a neighbourhood of $\mathbf{x}_0 \in F^d.$ Also let $f=(f_1,f_2,\dots,f_n) : U \rightarrow F^n$ be an analytic map  satisfying \ref{II}, \ref{III} and \ref{IV}. Consider the set
$$\mathcal{F} = \{(\mathbf{a} \cdot f, \mathbf{a'} \cdot f + a_0) : \|\mathbf{a}\|=\|\mathbf{a'}\|=\|\mathbf{a} \wedge \mathbf{a'}\|=1, \,\, \mathbf{a},\mathbf{a'} \in F^n, \,\, a_0 \in F \}.$$
Then there is a neighbourhood $V \subseteq U $ of $\mathbf{x}_0$ such that
\begin{enumerate} [label=({{\roman*}})]
    \item For any neighbourhood $B \subseteq V$ of $\mathbf{x}_0,$ there exists $\rho = \rho(\mathcal F, B)$ such that $\displaystyle\sup_{\mathbf{x} \in B} \|\widetilde{\nabla }g(\mathbf{x})\| \geq \rho  $ for any $g \in \mathcal{F}.$
    \item \label{ii} There exist $C,\alpha>0$ such that $\|\widetilde{\nabla}g\|$ is $(C,\alpha)$-good on $V,$ for all $g \in \mathcal{F}.$
\end{enumerate}
\end{theorem}
\begin{proof}
The proof follows the approach of \cite{MG1}. To prove part $(i),$ we will proceed by contradiction. Let us suppose that $(i)$ does not hold. Then we can find a neighbourhood $B$ of $\mathbf{x}_0$ such that  there exists $\{g_n\}_n \subseteq \mathcal{F} $ with $\|\widetilde{\nabla}g_n (\mathbf{x})\| \leq 1/n$ for all $\mathbf{x} \in B,$ where $g_n = (\mathbf{a}_n \cdot f, \mathbf{a}_n' \cdot f + a_{0n})$ for any $n \in \mathbb{N}.$ Now if $\{ a_{0n}\}$ has a bounded subsequence, then going to a subsequence, we may assume that $g_n$ is converging to $g \in \mathcal{F}.$ Thus we have $\|\widetilde{\nabla}g(\mathbf{x}) \|=0$ for any $\mathbf{x} \in B$ and this contradicts the linear independence of $1,f_1,\dots,f_n.$ Therefore we may assume that $a_{0n} \rightarrow \infty.$ Note that $\inf_{\|\mathbf{a}\|=1} \sup_{\mathbf{x} \in B} \|  \nabla f(\mathbf{x})\cdot \mathbf{a}\| = \delta,  $ for some $\delta >0$. Using these we get that $\sup_{\mathbf{x} \in B} \|\widetilde{\nabla}g_n(\mathbf{x})\| \rightarrow \infty,$ and this a contradiction to our assumption. Hence the proof of part $(i) $ is done.\\

 The proof of part $(ii)$ is divided into two parts, namely for the `compact part' of $\mathcal{F}$ (i.e., there is some upper bound on $|a_0|$) and the `unbounded part' of $\mathcal{F}.$
 
 \begin{lemma}\label{lem:bounded_skew_C}
 Let $U$ and $f$ be as in the Theorem \ref{thm:skew-calpha} and for $R>0$, define $\mathcal{F}_{\leq R}$ as follows
 \begin{equation*}
     \mathcal{F}_{\leq R} \defeq \{(\mathbf{a} \cdot f, \mathbf{a'} \cdot f + a_0) : \|\mathbf{a}\|=\|\mathbf{a'}\|=\|\mathbf{a} \wedge \mathbf{a'}\|=1, \,\, \mathbf{a},\mathbf{a'} \in F^n, \,\, a_0 \in F , \,\, |a_0| \leq R \}
 \end{equation*}
 Then there exists a neighbourhood $V=V_R$ of $\mathbf{x}_0,$ $C=C_R$ and $\alpha = \alpha_R >0$ such that $\|\widetilde{\nabla}g\| $ is $(C,\alpha)$-good on $V$ for any $g \in \mathcal{F}_{\leq R}.$
 \end{lemma}
 \begin{proof}
 Without loss of generality we can take $\mathbf{x}_0 =0 $ (replace $f(\mathbf{x}) $ by $f(\mathbf{x} + \mathbf{x}_0)$ for this purpose). By rescaling $\mathbf x \in U $ by $r\mathbf x$ and normalizing $f_i $ by $f_i/r',$ where $r,r' \in F$ are constants, we may and will assume all the Taylor coefficients of $f_i$'s are in $\mathscr{O}$ and $U \subseteq X^{-1}  \mathscr{O}^d \defeq\left \{X^{-1}  \mathbf{b}: \mathbf{b} \in \mathscr{O}^d \right \}.$ Note that given any $g \in \mathcal{F}_{\leq R},$ there exists $A_g \in \GL_d(\mathscr{O})$ such that all of the components of $\widetilde{\nabla}(g \circ A_g)$ are non-zero functions. To see this, we write down the expression of $\widetilde{\nabla}(g \circ A_g)$ and use the fact that $1,f_1,\dots,f_n$ are linearly independent on any open subset of $U$.  We complete the proof of this lemma in the next two steps:\\
 
 \noindent \textbf{Step 1}:   There exists a $\delta >0$ and  $A_1,\dots,A_k \in \GL_d(\mathscr{O})$ such that 
 \begin{equation*}
     \displaystyle \sup_{1 \leq j \leq k, \mathbf{x} \in U} |(\widetilde{\nabla} (g \circ A_j))_i (\mathbf{x})| \geq \delta \,\, \,\, \,\, \text{for} \,\, 1 \leq i \leq d. 
 \end{equation*}
 \textit{Proof of Step 1}:  $\,\,$ First note that  for any $g \in \mathcal{F}_{\leq R},$ there exists a $\delta >0$ and $A_g\in  \GL_d(\mathscr{O})$ such that  
 \begin{equation*}
     \displaystyle \sup_{ \mathbf{x} \in U} |(\widetilde{\nabla} (g \circ A_g))_i (\mathbf{x})| \geq \delta \,\, \,\, \,\, \text{for} \,\, 1 \leq i \leq d. 
 \end{equation*}
 If not there exist a sequence $\{g_m\}_m \subseteq \mathcal{F}_{\leq R}$ such that 
  \begin{equation*}
     \displaystyle \sup_{ \mathbf{x} \in U} |(\widetilde{\nabla} (g_m \circ A_{g_m}))_i (\mathbf{x})| \leq 1/m \,\, \,\, \,\, \text{for} \,\, 1 \leq i \leq d. 
 \end{equation*}
 
\noindent Since $\mathcal{F}_{\leq R}$ and $\GL_d(\mathscr{O})$ are compact,  going to a subsequence, we may assume that $g_m$ is converging to $g \in \mathcal F_{\leq R}.$ Then from above we have $\|(\widetilde{\nabla}(g \circ A_g))(\mathbf{x})\|=0$ for any $\mathbf{x} \in U$ and this is a contradiction to the linear independence  of $1,f_1,\dots,f_n$ on any open subset of $U$.\\  
 
\noindent  To complete the proof of this step we define $d$ number of  maps $$\phi_i : \mathcal{F}_{\leq R} \times \GL_d(\mathscr{O}) \rightarrow \{q^r:r \in \mathbb Z\} \cup \{0\}, $$ for $i=1,\dots,d$ by
 \begin{equation*}
     \phi_i(g,A)= \sup_{ \mathbf{x} \in U} |(\widetilde{\nabla} (g \circ A))_i (\mathbf{x})| \quad \text{for} \,\, g \in \mathcal{F}_{\leq R} \,\, \text{and}\,\, A   \in  \GL_d(\mathscr{O}) .
 \end{equation*}
\noindent Now for any $g_0 \in \mathcal{F}_{\leq R},$ consider that $A_{g_0} \in \GL_d(\mathscr{O})$ such that 
\begin{equation*}
 \sup_{ \mathbf{x} \in U} |(\widetilde{\nabla} (g_0 \circ A_{g_0}))_i (\mathbf{x})| \geq \delta 
\end{equation*}
From the  joint continuity of $\phi_i,$ there exists neighbourhood $N(g_0)$ and $N(A_{g_0}) $ of $g_0$ and $A_{g_0}$ respectively such that for all $g \in N(g_0) $ and $A \in N(A_{g_0}),$ $  \sup_{ \mathbf{x} \in U} |(\widetilde{\nabla} (g \circ A))_i (\mathbf{x})| \geq \delta .  $ We cover $\mathcal{F}_{\leq R}$ by such neighbourhoods $N(g)$ for any $g \in \mathcal{F}_{\leq R}.$ By compactness of $\mathcal{F}_{\leq R},$ we get a finite subcover $N(g_1), \dots, N(g_k).$ Consider the corresponding neighbourhoods of $A_g$'s, let they be $N(A_{g_1}),\dots,N(A_{g_k})$. \\

\noindent If we take any $g \in \mathcal{F}_{\leq R},$ there exist a $ l \in \{1,\dots,k\}$ such that $g \in N(g_l).$ This implies
\begin{equation*}
    \sup_{ \mathbf{x} \in U} |(\widetilde{\nabla} (g \circ A_{g_l}))_i (\mathbf{x})| \geq \delta 
\end{equation*}
\noindent Hence we get our desired result and this completes the proof of Step $1$. \\

\noindent \textbf{Step 2}\label{step:2_skew_gra}: There exists   $b=b_{\delta}$ with the following property: for any $g \in \mathcal{F}_{\leq R}$ there exists $1 \leq j \leq k$ so that for any $1 \leq i \leq d$ one can find a multi-index $\beta$ with $|\beta| \leq b,$ and $ |\partial_{\beta}(\widetilde{\nabla} (g \circ A_j))_i (\mathbf{0})| \geq \delta$. \\

\noindent \textit{Proof of Step 2}: To prove this we need to use contradiction argument. By contradiction argument, the proof easily follows using step $1$  and the Taylor series expansion of $(\widetilde{\nabla} (g \circ A_j))_i$ for each $i=1,\dots,d.$ Hence we omit the details here.   \\

 \noindent  To complete the proof of Lemma \ref{lem:bounded_skew_C}, we will use $\theta \equiv 0$ case of Proposition \ref{prop:grad_C}. In view of Step 2 and Proposition \ref{prop:grad_C} (with $\theta \equiv 0$), there exists a neighbourhood $V'$ of the origin, $C$ and $\alpha >0$ such that  for any $g \in \mathcal{F}_{\leq R}$ one can find $1 \leq j \leq k$ so that
 $(\widetilde{\nabla} (g \circ A_j))_i$ is $(C,\alpha)$-good on $V'$ for each $i=1,\dots,d.$ By Lemma \ref{lem:C,alpha}, we get that $\|\widetilde{\nabla}(g \circ A_j)\|$ is $(C,\alpha)$-good on $V'.$ This implies $\|\widetilde{\nabla}g\|$ is $(C,\alpha)$-good on $V'$ for any $g \in \mathcal{F}_{\leq R}.$
 \end{proof}
 To deal with the unbounded part of $\mathcal{F},$ we first need the following lemma, which is the function fields analog of \cite[Lemma 4.9]{MG1}. We have appropriately modified \cite[Lemma 4.9]{MG1} to suit the function field setup. The proof of it is just repeating the arguments of \cite[Lemma 4.9]{MG1}.
 %with first define the following. For $c=q^r \in \mathbb{R}_+,$ where $r \in \mathbb{Z},$ we define
 %\begin{equation*}
     %\lfloor c \rfloor \defeq X^{r} \in F.
 %\end{equation*}
 %Note that $|\lfloor c \rfloor|=q^{r}  = c.$ Now we need the following lemma to deal with the unbounded part of $\mathcal{F}.$
 \begin{lemma}\label{lem:skew_bound}
 Let $\mathbf{P}=(p_1,p_2,\dots,p_n),$ where $p_i \in \mathscr{O}[x_1,\dots,x_d]$ are linearly independent polynomials of degree $\leq l.$ For any $r \in \mathbb Z,$ 
 \begin{eqnarray*}
\mathbf{P}_{r}(\mathbf{x})\defeq X^{-rl} \mathbf{P}(X^{r} \mathbf{x})
 \end{eqnarray*}
 %let $\mathbf{P}_{q^r}(\mathbf{x})= \mathbf{P}(\lfloor q^r \rfloor \mathbf{x})/\lfloor q^r \rfloor^l= \mathbf{P}(X^{r} \mathbf{x})/X^{rl}.$
 Then there exists  $\gamma>0$ and $0<s<1$ such that for any $\mathbf{a}, \mathbf{a}' \in F^n$ with $\|\mathbf{a}\|=\|\mathbf{a}'\|=\|\mathbf{a} \wedge \mathbf{a}'\|=1,$ any $a_0 \in F$ with $|a_0| \geq q^l$ and $q^r < s$ one has 
 \begin{equation*}
     \| \widetilde{\nabla} \mathbf{Q}_{r}(\mathbf{x})\|_{B(\mathbf{0};1)} \geq \gamma (1+ \|\mathbf{Q}_{r}\|_{B(\mathbf{0};1)}),
 \end{equation*}
 where $\mathbf{Q}_{r} = (\mathbf{a} \cdot \mathbf{P}_{r}, \mathbf{a}' \cdot \mathbf{P}_{r} +  X^{-rl} a_0)$.
 \end{lemma}
Now we recall an important lemma from \cite{BKM} and \cite{MG1} and the proof of it is essentially the same as the proofs of Lemma 3.7 and Lemma 4.10 of \cite{BKM} and \cite{MG1} respectively, hence we omit the proof here.
\begin{lemma}\label{lem:reduction_C_alpha}
Let $B \subseteq F^d$ be an open ball of radius $r$ containing the point $\mathbf{y}_0 \in F^d$ and let $\tilde{B}$ be the ball containing the same point $\mathbf{y}_0$ and of radius $(q+1) \cdot r.$ Let $f$ be a continuous function on $\tilde{B},$ and suppose $C, \alpha >0$ and $ 0 < \delta < 1$ are such that 
\begin{equation*}
    \left| \left \{\mathbf{x} \in B' : |f(\mathbf{x})| < \varepsilon \cdot \sup_{\mathbf{x} \in B'} |f(\mathbf{x})| \right \}\right| \leq C \varepsilon^{\alpha}|B'|,
\end{equation*}
for any ball $B' \subseteq \tilde{B}$ and any $\varepsilon \geq \delta. $ Then $f$ is $(C,\alpha')$-good on $B$ whenever $0 < \alpha' < \alpha$ and $Cq\delta^{\alpha - \alpha'} \leq 1.$
\end{lemma}
Now we complete the proof of Theorem \ref{thm:skew-calpha} by proving the following lemma.
\begin{lemma}\label{lem:unbounded_skew_C}
Let $\mathbf{x}_0, U$ and $f$ be as in the Theorem \ref{thm:skew-calpha} and define $\mathcal{F}_{\geq R}$ as follows
 \begin{equation*}
     \mathcal{F}_{\geq R} \defeq \{(\mathbf{a} \cdot f, \mathbf{a'} \cdot f + a_0) : \|\mathbf{a}\|=\|\mathbf{a'}\|=\|\mathbf{a} \wedge \mathbf{a'}\|=1, \,\, \mathbf{a},\mathbf{a'} \in F^n, \,\, a_0 \in F , \,\, |a_0| \geq R \}
 \end{equation*}
 Then for sufficiently large $R$ there exists a neighbourhood $V$ of $\mathbf{x}_0,$ $C$ and $\alpha >0$ such that $\|\widetilde{\nabla}g\| $ is $(C,\alpha)$-good on $V$ for any $g \in \mathcal{F}_{\geq R}.$
\end{lemma}
\begin{proof}
Without loss in generality we assume that $\mathbf{x}_0 =0$ and the Taylor coefficients of each $f_i$ is in $\mathscr{O}$ and $\|\mathbf{x}\| \leq 1.$ Let $p_i$ be the $l$-th degree Taylor polynomial of $f_i$ at $0.$ Then we have $|f_i(\mathbf{x})-p_i(\mathbf{x})| \leq \|\mathbf{x}\|^{l+1}.$ We take $l$ large enough so that $1,p_1,\dots,p_n$ are linearly independent. Let $q^{r_0} < s$ be small enough such that
\begin{equation*}
    2qC_{d,2l-2}\left(\frac{8q^{r_0}}{\gamma}\right)^{\frac{1}{d(2l-1)(2l-2)}} \leq 1,
\end{equation*}
where $s,\gamma$ are given by Lemma \ref{lem:skew_bound} and $C_{d,2l-2}$ is as in Lemma \ref{lem:poly-C,alpha}. Now take $R \geq q^l$ consider $g=(\mathbf{a} \cdot f, \mathbf{a'} \cdot f + a_0)$ and $\tilde{g}=\left(\mathbf{a} \cdot (p_1,\dots,p_n), \mathbf{a'} \cdot (p_1,\dots,p_n)+ a_0)\right).$ In view of Lemma \ref{lem:reduction_C_alpha}, it is enough to show that, for any ball $B(\mathbf{x}_1;q^r) \subseteq B(\mathbf{0};q^{r_0}), $ for $8q^{r_0}/\gamma \leq \varepsilon \leq 1,$ and any $g \in \mathcal{F}_R',$ one has
\begin{equation}\label{equ:skew_C}
    \left| \left\{ \mathbf{x} \in B(\mathbf{x}_1;q^r) : \|\widetilde{\nabla}g(\mathbf{x})\| < \varepsilon \left( \sup_{\mathbf{x} \in B(\mathbf{x_1};q^r)} \|\widetilde{\nabla}g(\mathbf{x})\| \right)\right \}\right| \leq 2 C_{d,2l-2} \varepsilon^{\frac{1}{d(2l-2)}} |B(\mathbf{x}_1;q^r)|
\end{equation}
Now we scale, translate and normalize $g$ and $\tilde{g}$ as follows: let $g_{r}(\mathbf{x})=X^{-rl}g( X^r \mathbf{x} + \mathbf{x}_1 )$ and $\tilde{g}_{r}(\mathbf{x})= X^{-rl}\tilde{g}(X^r  \mathbf{x} + \mathbf{x}_1).$  Note that (\ref{equ:skew_C}) holds iff
\begin{equation*}
     \left| \left\{ \mathbf{x} \in B(\mathbf{0};1) : \|\widetilde{\nabla}g_{r}(\mathbf{x})\| < \varepsilon \left( \sup_{\mathbf{x} \in B(\mathbf{0};1)} \|\widetilde{\nabla}g_{r}(\mathbf{x})\|\right) \right \}\right| \leq 2 C_{d,2l-2} \varepsilon^{\frac{1}{d(2l-2)}} |B(\mathbf{0};1)|.
\end{equation*}
Observe that for any $\mathbf{x} \in B(\mathbf{0};1),$ $\|g_{r}(\mathbf{x})-\tilde{g}_{r}(\mathbf{x})\| < q^r$ and $\|\nabla g_{r}(\mathbf{x}) - \nabla \tilde{g}_{r}(\mathbf{x})\| < q^r.$ Hence
\begin{eqnarray*}
   \|\widetilde{\nabla} g_{r}(\mathbf{x}) - \widetilde{\nabla} \tilde{g}_{r}(\mathbf{x})\| & \leq & q^r(q^r +2) (1 + \|\tilde{g}_{r}(\mathbf{x})\|) \\ [2ex]
   & \leq & 3q^r (1 + \|\tilde{g}_{r}(\mathbf{x})\|) \\ [2ex]
   & \leq & \frac{3q^r}{\gamma} \sup_{\mathbf{x} \in B(\mathbf{0};1)} \|\widetilde{\nabla} \tilde{g}_{r} (\mathbf{x})\|,
\end{eqnarray*}
using the last lemma.
And so  $\left\{ \mathbf{x} \in B(\mathbf{0};1) : \|\widetilde{\nabla}g_{r}(\mathbf{x})\| < \varepsilon \cdot \sup_{\mathbf{x} \in B(\mathbf{0};1)} \|\widetilde{\nabla}g_{r}(\mathbf{x})\| \right \}$ is contained in 
 
\begin{equation*}
    \left \{ \mathbf{x} \in B(\mathbf{0};1): \|\widetilde{\nabla}\tilde{g}_{r}(\mathbf{x})\| - \frac{3q^r}{\gamma} \sup_{\mathbf{x} \in B(\mathbf{0};1)} \|\widetilde{\nabla} \tilde{g}_{r}(\mathbf{x})\| < \varepsilon \left( 1 + \frac{3q^r}{\gamma}\right) \sup_{\mathbf{x} \in B(\mathbf{0};1)}  \|\widetilde{\nabla} \tilde{g}_{r}(\mathbf{x})\| \right \}
\end{equation*}
\begin{equation*}
    =\left \{ \mathbf{x} \in B(\mathbf{0};1): \|\widetilde{\nabla}\tilde{g}_{r}(\mathbf{x})\|  < \left(\varepsilon \left(1+ \frac{3q^r}{\gamma}\right) + \frac{3q^r}{\gamma} \right) \sup_{\mathbf{x} \in B(\mathbf{0};1)}  \|\widetilde{\nabla} \tilde{g}_{r}(\mathbf{x})\| \right \}
\end{equation*}
\begin{equation*}
    \subseteq \left \{ \mathbf{x} \in B(\mathbf{0};1): \|\widetilde{\nabla}\tilde{g}_{r}(\mathbf{x})\|  < 2 \varepsilon \sup_{\mathbf{x} \in B(\mathbf{0};1)}  \|\widetilde{\nabla} \tilde{g}_{r}(\mathbf{x})\| \right \}.
\end{equation*}
Now since each component of $\widetilde{\nabla}\tilde{g}_{r}(\mathbf{x})$ is a polynomial of degree at most $2l-2,$ and $\widetilde{\nabla}\tilde{g}_{r}$ is not zero, we get our desired result and this completes the proof.
\end{proof}
Combining Lemma \ref{lem:bounded_skew_C} and Lemma \ref{lem:unbounded_skew_C}, we complete the proof of part \ref{ii} of Theorem \ref{thm:skew-calpha}.
\end{proof}
\section{Estimate for the big gradient part}\label{section:estimate_big_grad}
The proof of Theorem \ref{thm:inhomo-conv} is based on delicate estimates for measures of sets given by certain Diophantine inequalities. To be more specific, we deal with the set of 
$\mathbf{x} \in U$ such that for some polynomial vector $\mathbf{a} \in \Lambda^n$, the function $F_{\theta}(\mathbf x)=f(\mathbf x) \cdot \mathbf a + \theta(\mathbf x)$ becomes very close to a polynomial. To this end, as mentioned earlier in \ref{stra}, we split it into two parts depending upon whether the gradient $\nabla F_{\theta}(\mathbf x)=\nabla (f(\mathbf x) \cdot \mathbf a + \theta(\mathbf x))$ is big or small. At the outset, we consider the analogous homogeneous sets, i.e., the ones corresponding to $\theta \equiv 0$, and provide a measure estimate. Then from those (homogeneous) estimates with the help of some other tools, we get the final estimates for our desired sets. In this section, we find a measure estimate for the homogeneous big gradient set.

 \begin{theorem}\label{thm:big-grad}
 Suppose $U$ is an open subset of $F^d$ and $f: U  \rightarrow F^n$ satisfies \ref{II}, \ref{III} and \ref{IV}. Let $ 0 < \varepsilon < \frac{1}{2}$, $\delta>0$ and $t_1,\dots,t_n \in \mathbb{Z}_{\geq 0}.$ Let $\mathcal{A}$ be

\begin{equation}\label{big gra}
 \left \{\mathbf{x}\in U: \exists \,\, \mathbf{a}=(a_1,\dots,a_n) \in \Lambda^n, \,\, a_0 \in \Lambda \,\, \,\,
  \left|
\begin{array}{lcl}
\left| f(\mathbf{x}) \cdot \mathbf{a} + a_0 \right| < \delta q^{-(\sum_1^n t_i)}\\[2ex]
\|\nabla( f(\mathbf{x})\cdot \mathbf{a}) \|\geq \|\mathbf{a}\|^{1-\varepsilon} \\[2ex] 
|a_i|=q^{t_i}, \quad i=1,\dots,n
\end{array}\right.\right\},
\end{equation} 
then $|\mathcal{A}| < C \delta |U|,$ for  a universal constant $C.$ 
\end{theorem}
\begin{proof}
    The proof follows the approach of \cite[Theorem 1.2]{MG1}.  Fix $\mathbf{a}=(a_1,\dots,a_n) \in \Lambda^n$ such that  $|a_i|=q^{t_i}$ for $i=1,\dots ,n$, and define $T=q^{\sum_i t_i}.$ Let $0 < \varepsilon < \frac{1}{2}$ be fixed. Let $g(\mathbf{x})=f(\mathbf{x}) \cdot \mathbf{a}$ for any $\mathbf{x} \in U,$ and consider the following set
$$\mathcal{A}_{\mathbf{a}}=\{\mathbf{x} \in \mathcal{A} :  \text{the   hypothesis  of the  theorem holds  for} \,\, \mathbf{a}=(a_1,\dots,a_n)\}.$$
If  $\mathbf x \in \mathcal{A}_{\mathbf{a}}.$ Then we have:
\begin{enumerate}[label=({{\arabic*}})]
    \item\label{1} For some $a_0 \in \Lambda,$ $|g(\mathbf{x})+a_0| < \delta T^{-1}$
    \item\label{2} $\|\nabla g(\mathbf{x})\| \geq  \|\mathbf{a}\|^{1-\varepsilon}$
\end{enumerate}
Also  since $f$ satisfies \ref{IV}, $g$ satisfies the following conditions:
\begin{enumerate}[label=({{\alph*}})]
    \item\label{a} $\|\nabla g(\mathbf{x})\| \leq \|\mathbf{a}\|$ for any $\mathbf{x} \in U.$
    \item\label{b} For any $1\leq i,j \leq d,$ $\mathbf{y_1},\mathbf{y_2},\mathbf{y_3} \in U,$ $|\bar\Phi_{ij}(g)(\mathbf{y_1},\mathbf{y_2},\mathbf{y_3} )| \leq \|\mathbf{a}\|,$ where $\bar\Phi_{ij}(\cdot)=\bar\Phi_{\beta}(\cdot)$ with $\beta$ being the multiindex whose $i$-th and $j$-th coordinates are $1$ and all other coordinates are zero. 
\end{enumerate}
First note that it is enough to show that $|\mathcal{A}_{\mathbf{a}}| < C' \delta T^{-1}|U|$ for some constant $C'>0.$ Because if we have the above estimate, summing over all possible $\mathbf{a}$'s will give us the desired estimate of the measure of $\mathcal{A}$. 

\noindent Let $B(\mathbf{x})$ be a neighbourhood of $\mathbf{x}$ defined by:
$$B(\mathbf{x})\defeq B\left(\mathbf{x};\frac{1}{2\|\nabla g(\mathbf{x})\|}\right).$$

For $\mathbf{x} \in \mathcal{A}_{\mathbf{a}},$ let $a_0$ be an element in $\Lambda$ which satisfies \ref{1}.\\

\noindent By considering the Taylor expansion of $g$ about $\mathbf x,$  it is easy to observe that,  $\forall \mathbf{y} \in B(\mathbf{x}),$ $B(g(\mathbf{y});1/2) \cap \Lambda \subseteq \{a_0\},$ i.e., $a_0$ is the only possible $1/2$ approximation of $g(\mathbf{y})$ by an element of $\Lambda$.\\

%\noindent \textit{First step}:  We first prove that,  $\forall \mathbf{y} \in B(\mathbf{x}),$ $B(g(\mathbf{y}),1/2) \cap \Lambda \subseteq \{a_0\},$ i.e., $a_0$ is the only possible $1/2$ approximation of $g(\mathbf{y})$ by an element of $\Lambda$.\\

 %Let $a'_0 \in B(g(\mathbf{y}),1/2) \cap \Lambda.$ Also assume that $a_0 \neq a'_0.$ Now we consider the Taylor expansion of $g$ about $\mathbf{x}$,
%$$a_0+g(\mathbf{y})=a_0+g(\mathbf{x}) + \nabla g(\mathbf{x}) \cdot (\mathbf{x}-\mathbf{y})+\sum_{i,j} \bar\Phi_{ij}(g)(\cdot)(\mathbf{x}^{(i)}-\mathbf{y}^{(i)})(\mathbf{x}^{(j)}-\mathbf{y}^{(j)}),$$
%where the arguments of $\bar\Phi_{ij}(g)$ are some of the components of $\mathbf{x}$ and $\mathbf{y}.$ Note that the following inequalities holds:
%\begin{itemize}    
%\item $|g(\mathbf{x}) + a_0| \leq  1/2 ,$ by (\ref{1}),  
  %\item $| \nabla g(\mathbf{x}) \cdot (\mathbf{x}-\mathbf{y})| \leq 1/2,$ because of the definition of $B(\mathbf{x}),$
    %\item $ |\sum_{i,j} \bar\Phi_{ij}(g)(\cdot)(\mathbf{x}^{(i)}-\mathbf{y}^{(i)})(\mathbf{x}^{(j)}-\mathbf{y}^{(j)})| \leq 1/2,$ by (\ref{2}), (\ref{b}) and the definition of $B(\mathbf{x}).$
%\end{itemize}
%Now by using the ultrametric triangle inequality and these above inequalities, we compare the maximum possible values of each of the three terms in the Taylor expansion, and we get that $|g(\mathbf{y}) + a_0| \leq 1/2.$  This implies $|a_0 -a'_0| \leq 1/2$ and so $a_0=a'_0,$ since $a_0,a'_0 \in \Lambda.$\\

\noindent\textit{First step}: Now we show that  $\forall \mathbf{y} \in B(\mathbf{x}),$ one has 
$\|\nabla g(\mathbf{y})- \nabla g(\mathbf{x})\| < \frac{\|\nabla g(\mathbf{x})\|}{2}$. \\

Throughout let $\bar{\Phi}_j(\cdot)$ stand for  $\bar{\Phi}_{\beta} (\cdot)$ with $\beta$ being the multiindex whose $j$-th coordinate is $1$ and all other coordinates are zero. Now we consider the Taylor expansion of $\partial_i g$ about $\mathbf{x}.$ So let $\mathbf{y}=\mathbf{x} + \mathbf{z}$.
  Now we have 

\begin{eqnarray*}
   \partial_i g(\mathbf{y}) & = & \partial_i g(\mathbf{x}) + \sum_j \bar\Phi_j (\partial_i g)(\cdot)\mathbf{z}^j \\ [2ex]
   & = &  \partial_i g(\mathbf{x}) + \sum_j \left (\bar\Phi_{ji} ( g)(\cdot)+\bar\Phi_{ji} ( g)(\cdot) \right)\mathbf{z}^j,
\end{eqnarray*}
where the arguments of $\bar\Phi_{ij}(g)$ and $\bar\Phi_j(\partial_i g)$  are some of the components of $\mathbf{x}$ and $\mathbf{y}.$ In view of above, we get 
$$|\partial_i g(\mathbf{y}) - \partial_i g(\mathbf{x}) | \leq \max_j \{|\bar\Phi_{ij}(g)(\cdot)| \|\mathbf{z}\|^j\}  \leq \|\mathbf{a}\| \|\mathbf{z}\| \leq \frac{\|\mathbf{a}\|}{2\|\nabla g(\mathbf{x})\|} < \frac{\|\nabla g(\mathbf{x})\|}{2},$$
 by \ref{b},   \ref{2}, the definition of $B(\mathbf{x}),$ and using the fact that $\|\mathbf{z}\| <1$ and $0 < \varepsilon < 1/2.$ Hence we are done with this step.\\

\noindent \textit{Second step}: We claim that $ |\mathcal{A}_{\mathbf{a}} \cap B(\mathbf{x})| \leq C' \delta T^{-1}|B(\mathbf{x})|$.\\
 
 Without loss in generality, we may and will assume that  $\|\nabla g(\mathbf{x})\|=|\partial_1 g(\mathbf{x})|.$ First we calculate the measure of the considered set in the $e_1$ direction and then use Fubini's theorem to complete the proof. Now let $\mathbf{y}, \mathbf{y}' \in \mathcal{A}_{\mathbf{a}} \cap B(\mathbf{x}),$ and $\mathbf{y}'=\mathbf{y} + \alpha e_1.$ By the first step and \ref{1}, we have $|g(\mathbf{y})+a_0| < \delta T^{-1}$ and $|g(\mathbf{y}')+a_0| < \delta T^{-1}$ for some $a_0 \in \Lambda,$ and so
\begin{equation}\label{equ:thin-set}
    |g(\mathbf{y})-g(\mathbf{y}')| \leq \delta T^{-1}
\end{equation}
Again by the Taylor expansion we get 
\begin{equation*}
    g(\mathbf{y} + \alpha e_1) - g(\mathbf{y})= \partial_1 g(\mathbf{y}) \alpha + \Phi_{11} g(\cdot) \alpha^2,
\end{equation*}
this implies
\begin{equation}\label{equ:mean}
    |g(\mathbf{y} + \alpha e_1) - g(\mathbf{y})|=|\partial_1 g(\mathbf{y})| |\alpha|,
\end{equation}
by norm comparison of the two terms on the right-hand side of the above Taylor's expansion and using the equality case of the ultrametric triangle inequality.

Now we slice our concerned set by fixing the last $d-1$ entries. Observe that the measure of each slice is at most $C''\delta T^{-1}/\|\nabla g(\mathbf{x})\|=C'' (\delta T^{-1})2 \times \text{radius of}\,\, B(\mathbf{x}),$ by using  \eqref{equ:mean}, \eqref{equ:thin-set} and the second step. We complete the proof of this step by using Fubini's theorem.\\

\noindent \textit{Final step}:  $\{B(\mathbf{x})\}_{\mathbf{x} \in \mathcal{A}_{\mathbf{a}}}$ is a covering of $\mathcal{A}_{\mathbf{a}}.$ Using the covering property of ultrametric spaces and the third step, we get 
$$|\mathcal{A}_{\mathbf{a}}| \leq C' (\delta T^{-1})|U|.$$
We now sum up over all possible $\mathbf{a}$'s and get $|\mathcal A| < C\delta |U|,$ as desired. 
 \\
\end{proof}

As an immediate corollary of the above theorem, we get the following.
\begin{corollary}\label{coro:big_grad_thm}
Suppose $U$ is an open subset of $F^d$ and $f: U  \rightarrow F^n$ satisfies \ref{II}, \ref{III} and \ref{IV}. Let $ 0 < \varepsilon < \frac{1}{2}$, $\delta>0$ and $t \in \mathbb{Z}_{\geq 0}$. Let $\mathcal{A}'$ be

\begin{equation}
 \left \{\mathbf{x}\in U: \exists \,\, \mathbf{a}=(a_1,\dots,a_n) \in \Lambda^n, \,\, a_0 \in \Lambda \,\, \,\,
  \left|
\begin{array}{lcl}
\left| f(\mathbf{x}) \cdot \mathbf{a} + a_0 \right| < \delta q^{-nt}\\[2ex]
\|\nabla( f(\mathbf{x})\cdot \mathbf{a}) \|\geq \|\mathbf{a}\|^{1-\varepsilon} \\[2ex] 
\|\mathbf{a}\|=q^{t}, \quad i=1,\dots,n
\end{array}\right.\right\};
\end{equation} 
then $|\mathcal{A}'| < C' \delta |U|,$ for a universal constant $C'.$ 
\end{corollary}
\noindent The proof of it is exactly the same as the proof of Theorem \ref{thm:big-grad} and hence we omit. 

\section{Estimate for the small gradient part}\label{section:estimate_small_grad}
Here we continue with the same theme as mentioned in \S \ref{section:estimate_big_grad} and obtain a measure estimate for homogeneous small gradient set, which is kind of a complementary set to $\mathcal A$ considered in the hypothesis of Theorem \ref{thm:big-grad}. Later using the inhomogeneous transference principle, we provide a measure estimate for analogous inhomogeneous small gradient set. 

\begin{theorem}\label{thm:small-grad}
Suppose $U$ is an open subset of $F^d$ and $f: U  \rightarrow F^n$ satisfies \ref{II}, \ref{III} and \ref{IV}. Then for any $\mathbf{x}_0 \in U,$ one can find a neighbourhood $V \subseteq U$ of $\mathbf{x}_0$ and $\alpha >0$ with the following property: for any ball $B \subseteq V,$ there exists $E>0$ such that for any choice of $t,t',t_1,\dots,t_n \in \mathbb Z$ with $t \geq 0,$  $t_1,\dots,t_n \geq 1,$ and  $t' + \sum_i t_i - t - \max_i t_i < 0$ one has 

\begin{equation}\label{equ:small grad}
 \left|\left \{\mathbf{x}\in B : \exists \,\, \tilde{\mathbf{a}}=(\tilde{a_1},\dots,\tilde{a_n}) \in \Lambda^n,\,\, \tilde{a_0} \in \Lambda \,\,
  \left|
\begin{array}{lcl}
\left| f(\mathbf{x}) \cdot \tilde{\mathbf{a}} + \tilde{a_0}\right| < q^{-t} \\[2ex]
\|\nabla (f(\mathbf{x}) \cdot \tilde{\mathbf{a}}) \| < q^{t'} \\[2ex]  
|\tilde{a_i}| <  q^{t_i}, \,\, \,\,  i=1,\dots,n
\end{array}\right.\right\}\right| \leq E \varepsilon^{\alpha}|B|,
\end{equation} 
where $ \varepsilon \defeq \max \left( q^{-t}, q^{\frac{t' + \sum_i t_i - t - \max_i t_i}{n+1}} \right)$.

\end{theorem}

\subsection{The proof of Thoerem \ref{thm:small-grad} }
We prove Theorem \ref{thm:small-grad} by converting it into a dynamical problem following \cite{BKM}. In this section, $m$ will stand for $n+d+1.$ Denote the standard basis of ${F}^m$ by $\{\mathbf{e}_0,\mathbf{e}_1^*,\dots,\mathbf{e}_d^*,\mathbf{e}_1,\dots,\mathbf{e}_n\}.$ Also denote by $\Gamma $ the $\Lambda$-submodule generated by $\mathbf{e_0},\mathbf{e_1}, \dots, \mathbf{e_n},$ i.e.
\begin{equation}\label{equ:gamma}
\Gamma = \left\{\begin{pmatrix}
\tilde{a_0}\\
\mathbf 0\\
\tilde{\mathbf{a}}
\end{pmatrix} :  \tilde{a_0} \in \Lambda, \tilde{\mathbf{a}} \in \Lambda^{n} \right\}.
\end{equation}
Now take $t,$ $t',$ $t_i$'s, and the function $f$ as in Theorem \ref{thm:small-grad}, and let 
\begin{equation}\label{equ:u_x}
H_{f(\mathbf{x})} = \begin{pmatrix}
1 & 0 & f(\mathbf{x})\\
0 & I_d & \nabla f(\mathbf{x})\\

0 & 0 & I_n
\end{pmatrix} \in \SL_m(F).
\end{equation}
Note that 
\begin{equation}
H_{f(\mathbf{x})} \begin{pmatrix}
\tilde{a_0}\\
0\\
\tilde{\mathbf{a}}
\end{pmatrix}
= \begin{pmatrix}
f(\mathbf{x}) \cdot \tilde{\mathbf{a}}+\tilde{a_0}\\
\nabla (f(\mathbf{x}) \cdot \tilde{\mathbf{a}})\\
\tilde{\mathbf{a}}
\end{pmatrix} 
\end{equation}
Hence the existence of a nonzero solution in the system of inequalities in the left hand side of \eqref{equ:small grad} implies that the lattice $H_{f(\mathbf{x})} \Gamma$ have nonzero element inside certain parallelepiped in $F^m$. Next, we consider an appropriate diagonal matrix $D\in \GL_m(F)$ so as to transform that into a small cube. This would ensure that the solvability of the above system of inequalities will be same as finding short vectors of the lattice $DH_{f(\mathbf{x})} \Gamma$. Finally, we use the `quantitative nondivergence' estimate to get a measure estimate of the set of $\mathbf{x} \in B$ for which the above phenomena happen. 
Specifically, take $t,t', t_1,\dots,t_n, $ as in Theorem \ref{thm:small-grad} and   $\varepsilon= \max \left( q^{-t}, q^{\frac{t' + \sum_i t_i - t - \max_i t_i}{n+1}} \right)$ be same as in Theorem \ref{thm:small-grad}. For this $\varepsilon,$ we define $\lceil \varepsilon \rceil \in F$ as follows:
\begin{equation}
    \lceil \varepsilon \rceil \defeq\left \{ \begin{array}{lcl}\displaystyle  X^{-t} & \mbox{if} \,\, t < \frac{ t+\max_{1\leq i \leq n}t_i-t'-\sum_{i}t_i}{n+1} \\
                          X^{\left[\frac{t'+\sum t_i -t - \max t_i}{n+1}\right]+1}            &\mbox{otherwise},
     \end{array} \right.
\end{equation}
where for any $x \in \mathbb{R},$ $[x]$ denotes the greatest integer not greater than $x.$ Observe that $\lceil \varepsilon \rceil \neq 0$ and $|\lceil \varepsilon \rceil| \geq \varepsilon.$ Now we define $D \in \GL_m(F)$ as follows:
\begin{equation}
    D = \diag (a_0^{-1},a_*^{-1},\dots,a_*^{-1},a_1^{-1},\dots,a_n^{-1}),
\end{equation}
where
\begin{equation}\label{equ:a_0_a_1}
    a_0= \frac{X^{-t}}{\lceil \varepsilon \rceil},\,\, a_*=\frac{X^{t'}}{\lceil \varepsilon \rceil},\,\,a_i=\frac{X^{t_i}}{\lceil \varepsilon \rceil},\,\, i=1,\dots,n.
\end{equation}
It can be easily seen that the set within the bracket of left hand side of (\ref{equ:small grad}) is exactly the same as the following set
\begin{equation}
    \{\mathbf{x} \in B : \|DH_{f(\mathbf{x})}\mathbf{v}\| < |\lceil \varepsilon \rceil| \,\, \text{for some} \,\, \mathbf{v} \in \Gamma \setminus \{0\} \}.
\end{equation}
Therefore it is enough to estimate the above set. Considering this, we state the following theorem which completes the proof of Theorem \ref{thm:small-grad}.
\begin{theorem}\label{thm:small grad main}
Let $U$, $\mathbf{x}_0,$ d, n and $f$ be as in Theorem \ref{thm:small-grad}. Take $\Gamma$ as in (\ref{equ:gamma}) and $H_{f(\mathbf{x})}$ as in (\ref{equ:u_x}). Then there exists a neighbourhood $V \subseteq U$ of $\mathbf{x}_0,$ and a positive number $\alpha$ with the following property: for any ball $B \subseteq V$ there exists $E>0$ such that for any diagonal matrix $D = \diag (a_0^{-1},a_*^{-1},\dots,a_*^{-1},a_1^{-1},\dots,a_n^{-1}) \in \GL_m(F)$ with
\begin{enumerate}[label=({{\roman*}})]
    \item $0 < |a_0| \leq 1 \leq |a_1| \leq \dots \leq |a_n|,$ and 
    \item $0 < |a_*| \leq |a_0 a_1 \dots a_{n-1}|^{-1}$
\end{enumerate}
and for any positive number $\varepsilon,$ one has 
\begin{equation}\label{equ:quan_esti_modulo}
    \left|  \{\mathbf{x} \in B : \|DH_{f(\mathbf{x})}\mathbf{v}\| < \varepsilon \,\, \text{for some} \,\, \mathbf{v} \in \Gamma \setminus \{0\} \} \right| \leq E \varepsilon^{\alpha} \left |B \right|.
\end{equation}
\end{theorem}

\begin{proof}[Proof of Theorem \ref{thm:small-grad} modulo Theorem \ref{thm:small grad main}.]  Using a permutation without loss in generality we can assume that $t_1 \leq t_2 \leq \dots\leq t_n.$ Now take $\varepsilon$ as in Theorem \ref{thm:small-grad} and define $a_0,a_*,a_1,\dots,a_n$ as in (\ref{equ:a_0_a_1}). Using the fact that $|\lceil \varepsilon \rceil| \geq \varepsilon$, it is easy to verify that $a_0,a_*,a_1,\dots,a_n$ satisfy the constraints $(i)$ and $(ii)$ of Theorem \ref{thm:small grad main}. Hence applying Theorem \ref{thm:small grad main} in this particular case, we get our desired estimate. \end{proof}  
\subsubsection{Proof of Theorem \ref{thm:small grad main}}
 
First, we recall the ``Quantitative nondivergence theorem'' from \cite{G-JNT}, which is a generalization of \cite[Theorem 6.1]{KT}. To simplify the things coming up, we specialize \cite[Theorem 6.1]{G-JNT} to our setup. \\

Let  $F, \Lambda$ be as in \S\ref{subsection:mainres} and $\Gamma$ be given by \eqref{equ:gamma}. For a $\Lambda$ submodule $\Delta$ of ${F}^m$, ${F}\Delta\defeq \Span_{F}(\Delta)$. Clearly $\Gamma$ is  a  $ \Lambda$ submodule of $F^m$ with rank $n+1$ and given a submodule $\Delta $ of $\Gamma,$  $\Delta$ is said to be primitive in $\Gamma$ if $\Delta={F} \Delta \cap \Gamma$. We denote the set of all nonzero primitive $\Lambda$ submodules of $\Gamma$ by $\mathfrak{P}(\Gamma)$. 
 A function $\|\cdot\|: \bigwedge ({F}^m) \rightarrow \mathbb R_+$ is said to be $\textit{submultiplicative}$ if
 \begin{enumerate} [label=(S.{{\arabic*}})]
     \item it is continuous with respect to the natural topology of $\bigwedge({F}^m)$,
     \item $\|t \mathbf{w}\|=|t|\|\mathbf{w}\|,$  $\forall \mathbf{w} \in \bigwedge({F}^m)$ and $t \in {F}$, and 
     \item for any $\mathbf{v}, \mathbf{w} \in \bigwedge ({F}^m),$ one always has the inequality $\|\mathbf{v} \wedge \mathbf{w}\| \leq \|\mathbf{v}\| \cdot \|\mathbf{w}\|.$ 
 \end{enumerate} Now given a  ball $B=B(\mathbf{x}_0;r_0) \subseteq F^d, $ where $\mathbf{x}_0 \in F^d$ and $r_0>0$, a continuous function $h :  B(\mathbf{x}_0;3^{n+1} r_0) \rightarrow \GL_m({F})$ and a submultiplicative function $\|\cdot \|,$ we consider the function $\phi : \mathfrak{P}(\Gamma) \rightarrow C\left(B(\mathbf{x}_0;3^{n+1} r_0)\right)$ defined as follows: for $\Delta \in \mathfrak{P}(\Gamma),$ $\phi_{\Delta}$ is the map $\mathbf y \rightarrow \|h(\mathbf y) \Delta\|,$ $\forall \mathbf y \in B(\mathbf{x}_0;3^{n+1} r_0).$  Clearly every such $\phi_{\Delta}$ is a continuous function on $B(\mathbf{x}_0;3^{n+1} r_0)$.

\begin{theorem}{\cite[Theorem 5.2]{G-JNT}}
Let ${F},n,m,\Gamma,B=B(\mathbf{x}_0;r_0),\|\cdot\|$ and $\phi$ be as above. Suppose that $C,\alpha >0$ and $\rho \in (0,1]$ are constants such that the following holds:
\begin{enumerate}[label=({{\alph*}})]
    \item $\forall \Delta \in \mathfrak{P}(\Gamma),$ $\phi_{\Delta}$ is $(C,\alpha)$-good on $B(\mathbf{x}_0;3^{n+1} r_0),$
    \item $\forall \Delta \in \mathfrak{P}(\Gamma)$, $\|\phi_{\Delta}\|_{\mu,B} \geq \rho;$ and
    \item $\forall \mathbf{y} \in  B(\mathbf{x}_0;3^{n+1} r_0)$, $\#\{\Delta \in \mathfrak{P}(\Gamma): \|h(\mathbf{y})\Delta\| < \rho\}$ is finite.
\end{enumerate}
Then for any $\varepsilon \leq \rho,$ we have
\begin{equation}\label{equ:quan_easti}
    \left|\{ \mathbf y \in B : \|h(\mathbf y)\gamma\|< \varepsilon, \,\, for \,\, some \,\, \gamma \in \Gamma \setminus \{0\}\}\right| \leq (n+1)C \left(\frac{\varepsilon}{\rho}\right)^{\alpha}  |B|.
\end{equation}
\end{theorem}

%We will apply the above theorem to our setting. In our case:
%\begin{itemize}
%\item $\mathcal{D}= \Lambda $
    %\item $\mathcal{R}=$ F
    %\item $X= F^d$ and $\Theta=\Gamma$
%\end{itemize}
We need to define a submultiplicative function $\|\cdot\|,$ a family of functions $\mathcal{H},$ and verify these conditions of the above theorem for our choice of any function $h \in \mathcal{H}.$ We consider
\begin{itemize}
    \item $U,f, \Gamma, H_{f(\mathbf{x})}$ as in Theorem \ref{thm:small grad main} and $m=n+d+1.$ Denote the standard basis of  $F^m$ by $\{\mathbf{e}_0,\mathbf{e}_1^*,\dots,\mathbf{e}_d^*,\mathbf{e}_1,\dots,\mathbf{e}_n\}.$ Let $W^*$ be the $F$-linear span of $\mathbf{e}_1^*,\dots,\mathbf{e}_d^*$ and $W$ be the $F$-linear span of $\mathbf{e}_1,\dots,\mathbf{e}_{n}$
    %\item $\mathfrak{P}=$ the poset of primitive $\Lambda$-submodules of $\Gamma$
    \item $\mathcal{H}$ be the family of functions $h : U \rightarrow \GL_m(F)$ given by $h(\mathbf{x})=DH_{f(\mathbf{x})},$ where $H_{f(\mathbf{x})}$ is as in (\ref{equ:u_x}) and $D= \text{diag} (a_0^{-1},a_*^{-1},\dots,a_*^{-1},a_1^{-1},\dots,a_n^{-1})$ satisfying  conditions of Theorem \ref{thm:small grad main}.
    \item Now we choose a special submultiplicative function on $\bigwedge(F^m)$ as follows. Let 
    $$\mathcal S \defeq \left\{I \subseteq \{0, *1,\dots, *d, 1,\dots, n\}: \# I \cap \{*1,\dots, *d\} \leq 1 \right\}$$ and $\mathcal S'$ be the complement of $\mathcal S$ in the power set of $\{0, *1,\dots, *d, 1,\dots, n\}$. Clearly, we have the following direct sum decomposition of $\bigwedge(F^m)$ as:  $\Span \{\mathbf{e}_I : I \in \mathcal S\} \bigoplus \Span\{\mathbf{e}_I : I \in \mathcal S'\}$, where for $I=\{*i_1,\dots,*i_k,j_1,\dots,j_{\ell}\},$ $\mathbf{e}_I \defeq \mathbf{e}_{i_1}^{*} \wedge \dots \wedge \mathbf{e}_{i_k}^{*} \wedge \mathbf{e}_{j_1} \wedge \dots \wedge \mathbf{e}_{j_{\ell}}.$  Let $\pi$ be the projection operator to the first summand with respect to this decomposition. Define $\|\cdot\| : \bigwedge(F^m) \rightarrow \mathbb R_+$ as follows: for any $\mathbf{w} \in \bigwedge(F^m),$ $\|\mathbf{w}\|$ is defined as the supremum norm of $\pi(\mathbf{w}).$ In other words, if $\mathbf{w}$ is written as a sum of exterior products of base vectors $\mathbf{e}_i$ and $\mathbf{e}_i^*,$ to compute $\|\mathbf{w}\|$ we ignore components containing $\mathbf{e}_i^* \wedge \mathbf{e}_j^*,$ $1 \leq i \neq j \leq d,$ and take the supremum norm of the rest. It is easy to see that $\|\cdot\|$ is indeed a submultiplicative function on $\bigwedge(F^m).$ 
\end{itemize}
Since the restriction of the above submultiplicative function $\|\cdot\|$ on $F^m$ agrees with the supremum norm, to prove Theorem \ref{thm:small grad main}, it suffices to find a neighbourhood $V$ of $\mathbf{x}_0$ such that 
\begin{enumerate}[label=({{\Alph*}})]
    \item\label{A} there exist $C,\alpha >0$ such that all the functions $\mathbf{x} \mapsto \|h(\mathbf{x}) \Delta \|,$ where $h \in \mathcal{H}$ and $\Delta \in \mathfrak{P}(\Gamma)$ are $(C,\alpha)$-good on $V,$ 
    \item\label{B} for all $\mathbf{x} \in V$ and $h \in \mathcal{H}$ one has \#$\{\Delta \in \mathfrak{P}(\Gamma) : \|h(\mathbf{x}) \Delta  \| \leq 1 \} < \infty,$ and
    \item\label{C}for every ball $B \subseteq V,$ there exists $\rho >0$ such that $\displaystyle
    \sup_{\mathbf{x} \in B} \|h(\mathbf{x}) \Delta\| \geq \rho$ for all $h \in \mathcal{H}$ and $\Delta \in \mathfrak{P}(\Gamma)$.
\end{enumerate}
We define the notion of orthonormality of vectors similar to that of \cite{MG1} in the following.  
\begin{definition}
   Let $\ell \in \mathbb N$ with $\ell \leq m.$ A set of vectors $\mathbf{x}_1,\dots,\mathbf{x}_{\ell}$ in $F^m,$ is called \emph{orthonormal} if $\|\mathbf{x}_1\|=\dots=\|\mathbf{x}_{\ell}\|=\|\mathbf{x}_1 \wedge \dots \wedge \mathbf{x}_{\ell}\|=1$.\end{definition}
   %or equivalently, if $\mathbf{x}_1+\mathcal{P}^m, \dots ,\mathbf{x}_l + \mathcal{P}^m$ are linearly independent in where $\mathcal{P}=\{x\in F :|x|<1\}$ is the unique maximal ideal inside the ring of integers $\mathcal{O}=\{x\in F:|x|\leq 1\}$.
\begin{remark}$\ $
\begin{enumerate}[label=({{\roman*}})]
    \item The above definition of orthonormality is made in view of the geometric phenomenon in the Euclidean space that the volume of a parallelepiped equals to the product of the lengths of its sides if and only if it is rectangular, i.e., sides are mutually orthogonal. 
    \item \label{equiv-ortho}Note that the notion of orthonormality can also be defined purely algebraically as follows. Consider the vectors $\mathbf{x}_1+\mathcal{P}^m, \dots ,\mathbf{x}_{\ell} + \mathcal{P}^m$ in $\mathscr{O}^m/\mathcal{P}^m \cong \mathbb{F}_q^m$, where $\mathcal{P}=\{x\in F :|x|<1\}$ is the unique maximal ideal inside the ring of integers $\mathscr{O}$. Then $\mathbf{x}_1,\dots,\mathbf{x}_{\ell}$ is orthonormal if and only if $\mathbf{x}_1+\mathcal{P}^m, \dots ,\mathbf{x}_{\ell} + \mathcal{P}^m$ are linearly independent in the vector space  $\mathscr{O}^m/\mathcal{P}^m$ over $\mathbb{F}_q$. 
    \item In view of the equivalent definition mentioned above in \ref{equiv-ortho}, it is clear that any orthonormal subset of $F^m $ can be extended to an orthonormal basis of $F^m$.
\end{enumerate}
\end{remark}
We now verify conditions \ref{A}-\ref{C} which will finish the proof of the Theorem \ref{thm:small grad main}. We follow the approach of \cite{BKM} and \cite{MG1}. \\

$\bullet$ \textbf{Checking \ref{A}:} Let $\text{rank}_{\Lambda} \Delta=k. $ The case $k=0$ is trivial, so we take $1 \leq k \leq n+1.$ Also let $D\Delta_{F}\defeq \Span_F D\Delta$.  Clearly, $\dim D\Delta_{F}=k.$ Now we choose an orthonormal set of vectors $\mathbf{v}_1,\dots,\mathbf{v}_{k-1} \in D\Delta_{F} \,\, \cap W $. Observe that, if $\mathbf{e}_0\in D\Delta_F$ then $\{\mathbf{e}_0, \mathbf{v}_1, \dots, \mathbf{v}_{k-1}\}$ is an orthonormal basis of $D\Delta_F$. If $\mathbf{e}_0\notin D\Delta_F$ then we extend $\{\mathbf{v}_1,\dots, \mathbf{v}_{k-1}\}$ to an orthonormal basis of $D\Delta_F\oplus F\mathbf{e}_0$ of the form $\{\mathbf{e}_0,\mathbf{v}_0, \mathbf{v}_1, \dots, \mathbf{v}_{k-1}\}$. Then  a standard argument using orthonormality yields a vector $\alpha\mathbf{e}_0+\beta\mathbf{v}_0\in D\Delta_F$ such that $\{\alpha\mathbf{e}_0+\beta \mathbf{v}_0, \mathbf{v}_1, \dots, \mathbf{v}_{k-1}\}$ is an orthonormal basis of $D\Delta_F$. Thus in either case (if $\mathbf{e}_0\in D\Delta_F$, we interpret $\mathbf{v}_0 $ as $0$) we see that any for any $\mathbf{w} \in \bigwedge^k(F^m)$ representing $\Delta $, one has 
\begin{equation}\label{equ:delta-basis}
    D\mathbf{w}= a \mathbf{e}_0 \wedge \mathbf{v}_1 \wedge \dots \wedge \mathbf{v}_{k-1} + b \mathbf{v}_0 \wedge \mathbf{v}_1 \wedge \dots \wedge  \mathbf{v}_{k-1}
\end{equation}
for some $a,b \in F.$ We write
\begin{equation}
    h(\mathbf{x})\mathbf{w}= DH_{f(\mathbf{x})}D^{-1}(D\mathbf{w})
\end{equation}
and define the $m$-tuple of functions 
\begin{equation}
    \hat{f}(\mathbf{x})=\left(1,0,\dots,0,\frac{a_1}{a_0}f_1(\mathbf{x}),\dots, \frac{a_n}{a_0}f_n(\mathbf{x})\right)
\end{equation}
Let us define $\nabla^* \bar{f}(\mathbf{x})= \sum_{i=1}^{d} \partial_i \bar{f}(\mathbf{x}) \mathbf{e}_i^*,$ for any $C^1$ function $\bar{f}$ from an open subset of $F^d$ to $F.$ Also let $g_1$, $g_2$ are two functions from an open subset of $F^d$ to $F,$ and for a map $g=(g_1,g_2),$ we define $\tilde{\nabla}^*g $ by
\begin{equation*}
    \tilde{\nabla}^*g(\mathbf{x})= g_1(\mathbf x) \nabla^* g_2(\mathbf x) - g_2(\mathbf x) \nabla^* g_1(\mathbf{x}).
\end{equation*}
Then it is easy to see that $DH_{f(\mathbf{x})}D^{-1} \mathbf{e}_0=\mathbf{e}_0$ and
\begin{equation*}
    DH_{f(\mathbf{x})}D^{-1} \mathbf{v}= \mathbf{v} + \left(\hat{f}(\mathbf{x}) \mathbf{v}\right) \mathbf{e}_0 + \frac{a_0}{a_*} \nabla^* (\hat{f}(\mathbf{x}) \mathbf{v})
\end{equation*}
whenever $\mathbf{v}$ is in $W.$

\begin{equation*}
    DH_{f(\mathbf{x})}D^{-1}(\mathbf{e}_0 \wedge \mathbf{v}_1 \wedge \dots \wedge \mathbf{v}_{k-1} )= \mathbf{e}_0 \wedge \mathbf{v}_1 \wedge \dots \wedge \mathbf{v}_{k-1} + \frac{a_0}{a_*} \sum_{i=1}^{k-1} \left(\pm \nabla^* (\hat{f}(\mathbf{x})\mathbf{v}_i) \mathbf{e}_0 \wedge \bigwedge_{s \neq i} \mathbf{v}_s\right) + \mathbf{w}_1^*
\end{equation*}
and 
\begin{align*}
    DH_{f(\mathbf{x})}D^{-1}(\mathbf{v}_0 \wedge \mathbf{v}_1 \wedge \dots \wedge \mathbf{v}_{k-1} )&=\mathbf{v}_0 \wedge  \dots \wedge \mathbf{v}_{k-1} + \displaystyle \sum_{i=1}^{k-1} \left(\pm  (\hat{f}(\mathbf{x})\mathbf{v}_i) \mathbf{e}_0 \wedge \bigwedge_{s \neq i} \mathbf{v}_s\right)\\ &+\frac{a_0}{a_*} \displaystyle \sum_{i=1}^{k-1} \left(\pm \nabla^* (\hat{f}(\mathbf{x})\mathbf{v}_i)  \wedge \bigwedge_{s \neq i} \mathbf{v}_s \right) \\& + \frac{a_0}{a_*} \displaystyle \sum_{0\leq i<j\leq k-1} \left(\pm \tilde{\nabla}^* \left(\hat{f}(\mathbf{x})\mathbf{v}_i,\hat{f}(\mathbf{x})\mathbf{v}_j \right) \wedge \mathbf{e}_0 \wedge \bigwedge_{s \neq i,j} \mathbf{v}_s \right)+ \mathbf{w}_2^*,
\end{align*}
where $\mathbf{w}_1^*$ and $\mathbf{w}_1^*$ belong to $\bigwedge^2 (W^*) \wedge \bigwedge (W).$ Therefore, we have
\begin{eqnarray}\label{equ:pihxw}
   \pi(h(\mathbf{x}) \mathbf{w}) & = &  (a+b \hat{f}(\mathbf{x})\mathbf{v}_0)\,\, \mathbf{e}_0 \wedge \mathbf{v}_1 \wedge \dots \wedge \mathbf{v}_{k-1} + b \,\, \mathbf{v}_0 \wedge \mathbf{v}_1 \wedge \dots \wedge  \mathbf{v}_{k-1} \nonumber \\ [2ex]
   & + &   b \displaystyle \sum_{i=1}^{k-1} \left(\pm  (\hat{f}(\mathbf{x})\mathbf{v}_i) \mathbf{e}_0 \wedge \bigwedge_{s \neq i} \mathbf{v}_s \right) + b \frac{a_0}{a_*} \displaystyle \sum_{i=0}^{k-1} \left(\pm \nabla^* (\hat{f}(\mathbf{x})\mathbf{v}_i)  \wedge \bigwedge_{s \neq i} \mathbf{v}_s \right)\nonumber \\ [2ex]
   & + & \frac{a_0}{a_*} \displaystyle \sum_{i=1}^{k-1} \left(\pm \tilde{\nabla}^* \left(\hat{f}(\mathbf{x})\mathbf{v}_i,a + b \hat{f}(\mathbf{x})\mathbf{v}_0 \right) \wedge \mathbf{e}_0 \wedge \bigwedge_{s \neq 0,i} \mathbf{v}_s  \right)\nonumber \\ [2ex]
   & + & b \frac{a_0}{a_*} \displaystyle \sum_{1 \leq i<j \leq k-1 } \left(\pm \tilde{\nabla}^* \left(\hat{f}(\mathbf{x})\mathbf{v}_i,\hat{f}(\mathbf{x})\mathbf{v}_j \right) \wedge \mathbf{e}_0 \wedge \bigwedge_{s \neq i,j} \mathbf{v}_s \right)
\end{eqnarray}
We are now ready to prove condition \ref{A}. Since the supremum of a family of $ (C,\alpha) $-good functions is again a $(C, \alpha)$-good function, it suffices to show that the norms of each of these summands of (\ref{equ:pihxw})  are $(C,\alpha)$-good functions. Note that the components in the first two lines of (\ref{equ:pihxw}) are linear combination of either $1,f_1,\dots,f_n$ or $\nabla^* f_1,\dots,\nabla^*f_n,$ and so in view of Corollary \ref{coro:c-alpha-grad-homo}, we can find a neighbourhood $V_1$ of $\mathbf{x}_0,$ $C_1$ and $\alpha_1 >0$ such that all these norms of the first two lines are $(C_1,\alpha_1)$-good on $V_1.$ Now by applying Theorem \ref{thm:skew-calpha}, we get a neighbourhood $V_2,$ $C_2$ and $\alpha_2>0$ such that the norms of the remaining part of (\ref{equ:pihxw})  are $(C_2,\alpha_2)$-good on $V_2.$ Now combining these two and applying Lemma \ref{lem:C,alpha}, we get our desired condition \ref{A}.\\

$\bullet$ \textbf{Checking \ref{B}:} In view of the first line of (\ref{equ:pihxw}), we get that
\begin{equation*}
    \| h(\mathbf{x})\Delta \| \geq \max (|a+b \hat{f}(\mathbf{x}) \mathbf{v}_0|,|b|).
\end{equation*}
Thus $\| h(\mathbf{x}) \Delta \| \leq 1$ implies that $|b| \leq 1$ and $|a| \leq 1+ \|\hat{f}(\mathbf{x})\|$. From \eqref{equ:delta-basis}, we have $\|D\mathbf{w}\|=\max\{|a|,|b|\}.$ Now in view of the discreteness of $\bigwedge (D\Gamma)$ in $\bigwedge(F^m),$  it is easy to see that  given any $M>0$ the set of  $\Delta \in \mathfrak{P}(\Gamma)$ such that both $|a|$ and $|b|$ are bounded above by $M$ is finite. Hence we conclude that condition \ref{B} holds.\\

$\bullet$ \textbf{Checking \ref{C}:} Let $V$ be the neighbourhood of $\mathbf{x}_0$ given by Theorem \ref{thm:skew-calpha} and $B \subseteq V$ be a ball containing $\mathbf{x}_0.$ Now we define $\rho_1, \rho_2, \rho_3, M $ as follows:
\begin{equation*}
    \rho_1 = \inf \{ |f(\mathbf{x}) \mathbf{v} + v_0| \,\,  : \,\, \mathbf{x} \in B,\,\, \mathbf{v} \in F^n, \,\, v_0 \in F, \,\, \|\mathbf{v}\|=1   \},
\end{equation*}
\begin{equation*}
    \rho_2 = \inf \left\{ \sup_{\mathbf{x}\in B}\|\nabla f(\mathbf{x}) \mathbf{v}\| \,\, : \,\, \mathbf{v} \in F^n,  \,\, \|\mathbf{v}\|=1  \right\}
\end{equation*}
and $\rho_3 >0$ given by Theorem \ref{thm:skew-calpha}(a), and 
\begin{equation*}
    M \defeq \max \,\, ( \,\, \sup_{\mathbf{x} \in B} \|f(\mathbf{x})\|,\,\,  \sup_{\mathbf{x} \in B} \|\nabla f(\mathbf{x})\| \,\,  )
\end{equation*}
From the linear independence of the functions $1,f_1, \dots, f_n$, it follows that $\rho_1 >0$.

First we consider the case $\text{rank}_{\Lambda} \Delta =1. $ Then $\Delta$ can be represented by a vector $$\mathbf{v}=(v_0,0,\dots,0,v_1,\dots,v_n)^T,$$ with $v_i \in \Lambda$ for $i=0,1,\dots,n.$ Note that the first coordinate of $h(\mathbf x) \mathbf{v}$ is equal to
\begin{equation*}
    \frac{1}{a_0}(v_0 + v_1 f_1 (\mathbf{x}) + \dots + v_n f_n (\mathbf{x}) )
\end{equation*}
This implies $\|h(\mathbf{x})\Delta\|\geq \rho_1,$ since $|a_0| \leq 1.$\\

Now assume that $\text{rank}_{F} \Delta =k>1.$ Let $\mathbf{v}_1,\dots,\mathbf{v}_{k-2}$ be an orthonormal set in $(W \setminus F\mathbf{e}_n) \cap \Delta_{F},$ where $\Delta_{F}$ is the $F$ span of $\Delta.$ Now by adding $\mathbf{v}_{k-1},$ we extend $\{\mathbf{v}_1,\dots,\mathbf{v}_{k-2}\}$ to an orthonormal set in $W \cap \Delta_{F}.$ Now, if necessary choose a vector $\mathbf{v}_0$ to complete $\{\mathbf{e}_0,\mathbf{v}_1,\dots,\mathbf{v}_{k-1}\}$ to an orthonormal basis of $\Delta_{F} \oplus F \mathbf{e}_0.$ Now similar to the earlier discussion and equation (\ref{equ:delta-basis}), we will represent $\Delta$ by $\mathbf{w} \in \bigwedge^k(F^m)$ of the form
\begin{equation}\label{equ:w}
    \mathbf{w}= a \mathbf{e}_0 \wedge \mathbf{v}_1 \wedge \dots \wedge \mathbf{v}_{k-1} + b \mathbf{v}_0 \wedge \mathbf{v}_1 \wedge \dots \wedge  \mathbf{v}_{k-1}
\end{equation}
for some $a,b \in F$ with $\max \{|a|,|b|\} \geq 1.$ Since $D$ leaves $W ,$ $W^*$ and $F \mathbf{e}_0$  invariant, we have
\begin{equation*}
    \|DH_{f(\mathbf{x})} \Delta \|= \|DH_{f(\mathbf{x})} \mathbf{w}\|= \|\pi(DH_{f(\mathbf{x})} \mathbf{w})\|=\|D\pi(H_{f(\mathbf{x})} \mathbf{w})\|
\end{equation*}

Now consider the $m$-tuple of functions
\begin{equation*}
    \tilde{f}(\mathbf{x})\defeq(1,0,\dots,0,f_1(\mathbf{x}),\dots,f_n(\mathbf{x})),
\end{equation*}
and note that $H_{f(\mathbf{x})} \mathbf{e}_0=\mathbf{e}_0$ and
\begin{equation*}
    H_{f(\mathbf{x})} \mathbf{v}= \mathbf{v} + (\tilde{f}(\mathbf{x})\mathbf{v}) \mathbf{e}_0 + \nabla^*(\tilde{f}(\mathbf{x}) \mathbf{v})
\end{equation*}
whenever $\mathbf{v}$ is in $W.$ Using this and (\ref{equ:w}), we also get
\begin{eqnarray*}
   \pi(H_{f(\mathbf{x})} \mathbf{w}) & = &  (a+b (\tilde{f}(\mathbf{x})\mathbf{v}_0))\,\, \mathbf{e}_0 \wedge \mathbf{v}_1 \wedge \dots \wedge \mathbf{v}_{k-1} + b \,\, \mathbf{v}_0 \wedge \mathbf{v}_1 \wedge \dots \wedge  \mathbf{v}_{k-1} \nonumber \\ [2ex]
   & + &   b \displaystyle \sum_{i=1}^{k-1} \left(\pm  (\tilde{f}(\mathbf{x})\mathbf{v}_i) \mathbf{e}_0 \wedge \bigwedge_{s \neq i} \mathbf{v}_s \right) + b \displaystyle \sum_{i=0}^{k-1} \left(\pm \nabla^* (\tilde{f}(\mathbf{x})\mathbf{v}_i)  \wedge \bigwedge_{s \neq i} \mathbf{v}_s \right)+ \mathbf{e_0} \wedge \tilde{\mathbf{w}}(\mathbf{x}),  
\end{eqnarray*}
where 
\begin{eqnarray*}
  \tilde{\mathbf{w}}(\mathbf{x}) & = &  \displaystyle \sum_{i=1}^{k-1} \left(\pm \tilde{\nabla}^* \left(\tilde{f}(\mathbf{x})\mathbf{v}_i,a + b \tilde{f}(\mathbf{x})\mathbf{v}_0 \right)  \wedge \bigwedge_{s \neq 0,i} \mathbf{v}_s \right) \nonumber \\ &+& b \displaystyle \sum_{1 \leq i<j \leq k-1} \left(\pm \tilde{\nabla}^* \left(\tilde{f}(\mathbf{x})\mathbf{v}_i,\tilde{f}(\mathbf{x})\mathbf{v}_j \right) \wedge \bigwedge_{s \neq i,j} \mathbf{v}_s \right) .
\end{eqnarray*}
To prove that $\|h(\mathbf{x}) \mathbf{w}\|$ is not less than $\rho$ for some $\mathbf{x} \in B,$ it is enough to show that $\sup_{\mathbf{x} \in B} \|D  \tilde{\mathbf{w}}(\mathbf{x}) \|$
is not less than $\rho |a_0|.$ 

Now consider the product $\mathbf{e}_n \wedge  \tilde{\mathbf{w}}(\mathbf{x}).$ We claim that it is enough to show that
\begin{equation}\label{equ:tilde_w}
    \|\mathbf{e}_n \wedge  \tilde{\mathbf{w}}(\mathbf{x})\| \geq \rho \,\, \text{for some} \,\,  \mathbf{x} \in B.
\end{equation}
First, we finish the proof assuming (\ref{equ:tilde_w}), and then we will prove (\ref{equ:tilde_w}). Now since $\mathbf{e}_n$ is an eigenvector of $D$ with eigenvalue $a_n^{-1},$ we have $\|D(\mathbf{e}_n \wedge  \tilde{\mathbf{w}}(\mathbf{x}))\| \leq |a_n|^{-1} \|D\tilde{\mathbf{w}}(\mathbf{x})\|.$ Also since the eigenvalue with the smallest norm of $D$ on $W^* \wedge (\bigwedge^{k-2}(W \oplus F \mathbf{e}_0))$ is equal to $(a_* a_{n-k+1} \dots a_n)^{-1},$ one has
\begin{eqnarray*}
   \|D \tilde{\mathbf{w}}(\mathbf{x})\| & \geq & |a_n| \|D(\mathbf{e}_n \wedge  \tilde{\mathbf{w}}(\mathbf{x}))\| \\ [2ex]
   & \geq & \frac{|a_n|}{|a_* a_{n-k+1} \dots a_n|} \|\mathbf{e}_n \wedge \tilde{\mathbf{w}}(\mathbf{x})  \| \\ [2ex]
   & \geq & \frac{|a_0| \rho}{|a_0 a_* a_1 \dots a_{n-1}|} \geq |a_0| \rho.
\end{eqnarray*}
Thus it remains to prove (\ref{equ:tilde_w}). To prove this we consider the term containing $\mathbf{v}_1 \wedge \dots \wedge \mathbf{v}_{k-2},$ and then we have
\begin{eqnarray}
   \mathbf{e}_n \wedge \tilde{\mathbf{w}}(\mathbf{x}) & = & \pm \mathbf{v}^*(\mathbf{x}) \wedge \mathbf{e}_n \wedge \mathbf{v}_1 \wedge \dots \wedge \mathbf{v}_{k-2} \nonumber \\ [2ex] & + & \text{other terms where one or two } \mathbf{v}_i \,\,  \text{are missing},
\end{eqnarray}
where
\begin{equation*}
   \mathbf{v}^*(\mathbf{x}) = \tilde{\nabla}^* (\tilde{f}(\mathbf{x})\mathbf{v}_{k-1},a + b \tilde{f}(\mathbf{x})\mathbf{v}_0 ) = b \tilde{\nabla}^* (\tilde{f}(\mathbf{x})\mathbf{v}_{k-1},  \tilde{f}(\mathbf{x})\mathbf{v}_0 ) - a {\nabla}^* (\tilde{f}(\mathbf{x})\mathbf{v}_{k-1} )
\end{equation*}
Note that $\| \mathbf{e}_n \wedge \tilde{\mathbf{w}}(\mathbf{x}) \| \geq \| \mathbf{v}^*(\mathbf{x})\|.$ Now using the first expression for $\mathbf{v}^*(\mathbf{x}),$ we get $\sup_{\mathbf{x} \in B} \|\mathbf{v}^*(\mathbf{x})\| \geq \rho_3 |b|$ and the second expression gives  $\sup_{\mathbf{x} \in B} \|\mathbf{v}^*(\mathbf{x})\| \geq \rho_2 |a| - 2 M^2 |b|.$ It is easy to see that there exists $\rho_0$ such that 
\begin{equation*}
    \max \{\rho_2 |a|-2M^2 |b|, \rho_3 |b|\} \geq \rho_0 \cdot \max \{|a|,|b|\} \geq \rho_0.
\end{equation*}
Thus $\rho=\rho_0$ satisfies condition \ref{C} and this completes the proof of Theorem \ref{thm:small grad main}, and hence of Theorem \ref{thm:small-grad}.
 
\begin{remark}
    Note that the $\rho$ appearing in the condition \ref{C} needs to be independent of the family  $\mathcal{H}$ in order to have \eqref{equ:quan_esti_modulo}  for all diagonal matrix $D$ (satisfying certain conditions).  Hence we need the $\rho$ to be uniform in \eqref{equ:quan_easti}.
\end{remark}

\section{Inhomogeneous convergence case}\label{section:inhomoconv}
Combining the previous estimates, namely,  Theorem \ref{thm:big-grad} and \ref{thm:small-grad}, we now prove the convergence case of the Khitchine-Groshev theorem in the inhomogeneous arena. This amounts to estimating the size of analogous inhomogeneous big gradient and small gradient sets.\\

\noindent Recall that,
\begin{equation*}
    \mathcal W_f(\Psi,\theta) = \{\mathbf{x} \in U : f(\mathbf{x}) \,\, \text{is} \,\, (\Psi,\theta)-\text{approximable} \} 
\end{equation*}
We write $\mathcal W_f(\Psi,\theta)$ as
\begin{equation*}
    \mathcal W_f(\Psi,\theta)= \displaystyle \limsup_{\mathbf{a} \rightarrow \infty} \mathcal W_f(\mathbf{a},\Psi,\theta),
\end{equation*}
where 
for $\mathbf{a} \in \Lambda^n$
\begin{equation*}
    \mathcal W_f(\mathbf{a},\Psi,\theta)\defeq \{ \mathbf{x} \in U : |a_0 + \mathbf{a} \cdot f(\mathbf{x})+ \theta(\mathbf{x})| < \Psi(\mathbf{a}) \,\, \text{for some } \,\, a_0 \in \Lambda \}.
\end{equation*}
 For each $\mathbf{a} \in \Lambda^n \setminus \{0\},$ we decompose the set $ \mathcal W_f(\mathbf{a},\Psi,\theta)$ into two subsets based on `big gradient' and `small gradient' as follows:
 \begin{equation}\label{equ:w_large_f_a}
      \mathcal W_f^{\text{big}}(\mathbf{a},\Psi,\theta) \defeq \left\{\mathbf{x} \in  \mathcal W_f(\mathbf{a},\Psi,\theta) : \|\nabla (f(\mathbf{x}) \cdot \mathbf{a} + \theta(\mathbf{x}))\| \geq \|\mathbf{a}\|^{1-\varepsilon} \right \}
 \end{equation}
and 
\begin{equation}\label{equ:w_small_f_a}
      \mathcal W_f^{\text{small}}(\mathbf{a},\Psi,\theta) \defeq \left\{\mathbf{x} \in  \mathcal W_f(\mathbf{a},\Psi,\theta) : \|\nabla (f(\mathbf{x}) \cdot \mathbf{a} + \theta(\mathbf{x}))\| < \|\mathbf{a}\|^{1-\varepsilon} \right \},
 \end{equation}
 where $0 < \varepsilon < 1/2$ is a fixed number.

\noindent Now it is easy to see that
\begin{equation*}
     \mathcal W_f(\Psi,\theta) =  \mathcal W_f^{\text{big}}(\Psi,\theta) \cup  \mathcal W_f^{\text{small}}(\Psi,\theta),
\end{equation*}
where $\mathcal W_f^{\text{big}}(\Psi,\theta) = \displaystyle \limsup_{\mathbf{a} \rightarrow \infty }  \mathcal W_f^{\text{big}}(\mathbf{a},\Psi,\theta)$  and $\mathcal W_f^{\text{small}}(\Psi,\theta) = \displaystyle\limsup_{\mathbf{a} \rightarrow \infty }  \mathcal W_f^{\text{small}}(\mathbf{a},\Psi,\theta).$ To prove our desired statement $|\mathcal W_f(\Psi,\theta)|=0$ (i.e, Theorem \ref{thm:inhomo-conv}), it is enough to show that $|\mathcal W_f^{\text{big}}(\Psi,\theta)|=0$ and $|\mathcal W_f^{\text{small}}(\Psi,\theta)|=0.$ We consider these two cases separately.

\subsection{Big gradient} This subsection is devoted to show that $|\mathcal W_f^{\text{big}}(\Psi,\theta)|=0$. For this, we first prove the following extension of Theorem \ref{thm:big-grad} to the inhomogeneous case. 
\begin{proposition}\label{cor:big_grad_inhomo}
Suppose $U$ is an open subset of $F^d$ and $f: U  \rightarrow F^n$ satisfies \ref{II}, \ref{III}, \ref{IV}, and $\theta: U \rightarrow F$ satisfy condition \ref{VI}. Let $t_1,\dots,t_n \in \mathbb{Z}_{\geq 0}$ and $\mathcal{A}_{\theta}$ be
\begin{equation}\label{big_grad_inhomo}
 \left \{\mathbf{x}\in U: \exists \,\, \mathbf{a}=(a_1,\dots,a_n) \in \Lambda^n, a_0 \in \Lambda \,\, \text{ s.t.}
  \left|
\begin{array}{lcl}
\left| f(\mathbf{x}) \cdot \mathbf{a} + \theta(\mathbf{x}) +a_0 \right| < \delta  q^{-(\sum_1^n t_i)} \\[2ex]
\|\nabla ( f(\mathbf{x}) \cdot \mathbf{a} + \theta(\mathbf{x}) ) \|\geq \|\mathbf{a}\|^{1-\varepsilon} \\[2ex]  
|a_i|=q^{t_i}, \,\, \,\,  i=1,\dots,n
\end{array}\right.\right\}.
\end{equation} 
Then $|\mathcal{A}_{\theta}| < C \delta |U|,$ for a universal constant $C>0$.
\end{proposition}
\begin{proof}
Let $g=(f_1,\dots,f_n,\theta): U \rightarrow F^{n+1}$ and $\mathbf{a}'=(a_1,\dots,a_n,1) \in \Lambda^{n+1}.$ Note that $g$ satisfies the hypothesis of Theorem \ref{thm:big-grad}. Therefore by applying Theorem \ref{thm:big-grad} for the function $g,$ we get that 
\begin{equation*}
|\mathcal{A}_{\theta}| < C \delta |U|,
\end{equation*}
for a universal constant $C>0$.
\end{proof}

Now in particular apply Prosition \ref{cor:big_grad_inhomo} for $\delta = q^{\sum_1^n t_i} \Psi(X^{t_1},\dots,X^{t_n}),$ where $t_i$'s are same as in the above proposition. Since $|a_i|= q^{t_i},$ $\Psi (\mathbf{a}) \geq \Psi(X^{t_1 +1},\dots,X^{t_n+1})$ and $\Psi (\mathbf{a}) \leq \Psi(X^{t_1 },\dots,X^{t_n}).$ Therefore using Proposition \ref{cor:big_grad_inhomo}, we have
\begin{equation}\label{equ:w_large_a}
   \left| \bigcup_{|a_i|=q^{t_i}} \mathcal W_f^{\text{big}}(\mathbf{a},\Psi,\theta) \right | < C q^{\sum_1^n t_i} \Psi(X^{t_1},\dots,X^{t_n}) |U|.
\end{equation}
Note that 
    $\displaystyle \displaystyle \sum_{t_1,\dots,t_n}  \Psi(X^{t_1},\dots,X^{t_n}) q^{\sum_1^n t_i}\leq \sum_{\mathbf{a}} \Psi(\mathbf{a}) <\infty$. By (\ref{equ:w_large_a}) and  the Borel-Cantelli lemma we now get our desired result $|W_f^{\text{large}}(\Psi,\theta)|=0.$

\subsection{Inhomogeneous transference principle} Here we first recall  the inhomogeneous transference principle from \cite[Section 5]{BV}. Suppose that $(X,d)$ is a locally compact metric space. Given two countable indexing set $\mathcal{A}$ and $\mathbf{T},$ let $H$ and $I$ be two maps from $\mathbf{T} \times \mathcal{A} \times \mathbb{R}_+$ into the set of open subsets of $X$ such that
\begin{equation}\label{equ:H}
    H \,\, : \,\,  (\mathbf{t},\alpha, \lambda) \in \mathbf{T} \times \mathcal{A} \times \mathbb{R}_+ \rightarrow H_{\mathbf{t}}(\alpha,\lambda)
\end{equation}
and
\begin{equation}\label{equ:I}
     I \,\, : \,\, (\mathbf{t},\alpha, \lambda) \in \mathbf{T} \times \mathcal{A} \times \mathbb{R}_+ \rightarrow I_{\mathbf{t}}(\alpha,\lambda).
\end{equation}
Furthermore, let
\begin{equation}\label{equ:H_t_I_t}
    H_{\mathbf{t}}(\lambda)\defeq \bigcup_{\alpha \in \mathcal{A}} H_{\mathbf{t}}(\alpha,\lambda) \quad \text{and} \quad I_{\mathbf{t}}(\lambda)\defeq  \bigcup_{\alpha \in \mathcal{A}} I_{\mathbf{t}}(\alpha,\lambda).
\end{equation}
Let $\Phi$ denote a set of functions $\phi : \mathbf{T} \rightarrow \mathbb{R}_+.$ For $\phi \in \Phi,$ consider the limsup sets
\begin{equation}\label{equ:lambda_H_I}
    \Lambda_H(\phi) = \displaystyle \limsup_{\mathbf{t} \in \mathbf{T}} H_{\mathbf{t}}(\phi(\mathbf{t})) \quad \text{and} \quad \Lambda_I(\phi) = \displaystyle \limsup_{\mathbf{t} \in \mathbf{T}} I_{\mathbf{t}}(\phi({\mathbf{t}})).
\end{equation}
We call the sets associated with the map $H$ as homogeneous sets and those associated with the map $I$ as inhomogeneous sets. Now we discuss two important properties, which are the key ingredients for the inhomogeneous transference principle. \\

\noindent$\textbf{The intersection property.} $ We say that the triple $(H,I,\Phi)$  satisfy the \textit{intersection property} if for any $\phi \in \Phi,$ there exists $\phi^* \in \Phi$ such that for all but finitely many $\mathbf{t} \in T$ and all distinct $\alpha$ and $\alpha'$ in $\mathcal{A},$ we have that
\begin{equation}\label{equ:inter_property}
    I_{\mathbf{t}}(\alpha,\phi(\mathbf{t})) \cap I_{\mathbf{t}}(\alpha',\phi(\mathbf{t})) \subseteq H_{\mathbf{t}}(\phi^*(\mathbf{t})).
\end{equation}

\noindent $\textbf{The contraction property.}$ We suppose that $\mu$ is a non-atomic finite doubling measure supported on a bounded subset $\mathbf{S}$ of $X.$ We say that $\mu $  is doubling if there exists a constant $\lambda >1$ such that $\mu(B(\mathbf{x};2r)) \leq \lambda \mu(B(\mathbf{x};r))$ for every ball $B(\mathbf{x};r)$ with $\mathbf x \in \mathbf{S}$ and $r>0.$

\noindent  $\mu$ is said to be $\textit{contracting with respect}$ to $(I, \Phi)$ if for any $\phi \in \Phi$ there exists $\phi^+ \in \Phi$ and a sequence of positive numbers $\{ k_{\mathbf{t}} \}_{\mathbf{t} \in \mathbf{T}}$ satisfying
\begin{equation}
    \displaystyle \sum_{\mathbf{t} \in \mathbf{T}} k_{\mathbf{t}} < \infty,
\end{equation}
and for all but finitely many $\mathbf{t} \in \mathbf{T}$ and all $\alpha \in \mathcal{A}$ there exists a collection $C_{\mathbf{t},\alpha}$ of balls $B=B(\mathbf x;r)$ with  $\mathbf x \in \mathbf{S}$ satisfying the following conditions:
\begin{equation}\label{equ:cont_prop_i}
    \mathbf{S} \cap I_{\mathbf{t}}(\alpha,\phi({\mathbf{t}})) \subseteq \bigcup_{B \in C_{\mathbf{t},\alpha}} B,
\end{equation}
\begin{equation}\label{equ:cont_prop_ii}
    \mathbf{S} \cap \bigcup_{B \in C_{\mathbf{t},\alpha}} B \subseteq I_{\mathbf{t}}(\alpha,\phi^+(\mathbf{t}))
\end{equation}
and
\begin{equation}\label{equ:cont_prop_iii}
    \mu(B(\mathbf x;5r) \cap I_{\mathbf{t}}(\alpha,\phi(\mathbf{t}))) \leq k_{\mathbf{t}} \mu(B(\mathbf x;5r)).
\end{equation}
We now recall the following transference theorem from \cite[Theorem 5]{BV}.
\begin{theorem}\label{thm:transf}
Let the triple $(H,I,\Phi)$ satisfies the intersection property and $\mu$ is contracting with respect to $(I,\Phi).$ Then
\begin{equation}
    \mu(\Lambda_H(\phi))=0 \,\, \,\, \forall \,\, \,\, \phi \in \Phi \quad \implies \quad \mu(\Lambda_I(\phi))=0 \,\, \,\, \forall \,\, \,\, \phi \in \Phi.
\end{equation} 
\end{theorem}

\subsection{Small gradient}
We use the inhomogeneous transference principle to show that  $|\mathcal W_f^{\text{small}}(\Psi,\theta)|=0.$
Since $\Psi$ satisfies condition \ref{V} and $\sum \Psi(\mathbf{a}) < \infty,$ without loss of generality we may assume that $\Psi(\mathbf{a}) < \Psi_0(\mathbf{a}) \defeq c_0 \prod_{a_i \neq 0, i=1}^n |a_i|^{-1}$  for all but finitely many $\mathbf{a}=(a_1,\dots,a_n) \in \Lambda^n \setminus \{0\}$ and for some $c_0 >0.$ For the sake of convenience we will work with $c_0=1$ (all the computations we are going to do will follow immediately for any fixed $c_0>0 $ with minor changes).  In view of the above inequality, $\mathcal W_f^{\text{small}}(\Psi,\theta) \subseteq \mathcal W_f^{\text{small}}(\Psi_0,\theta)$ and therefore it is enough to show that $|\mathcal W_f^{\text{small}}(\Psi_0,\theta)|=0.$\\

\noindent Following the terminology of the last section (inhomogeneous transference principle), let us take 
\begin{itemize}
    \item $\mathcal{A}= \Lambda \times \Lambda^n \setminus \{0\}$ and $\mathbf{T}= \mathbb{Z}^{n}_{\geq 0}$
    
    \item For $0 < \delta < 1,$ $\phi_{\delta} : \mathbf{T} \rightarrow \mathbb{R}_+$ defined by $\phi_{\delta}(\mathbf{t})=q^{\delta |\mathbf{t}|}.$ Here for $\mathbf{t}=(t_1,\dots,t_n) \in \mathbf{T},$ we define $|\mathbf{t}|\defeq\max_{i} t_i.$
    
    \item $\Phi \defeq \{\phi_{\delta } : 0 < \delta < \varepsilon/2\}$
    
    \item For every $\lambda >0,$ $\mathbf{t}=(t_1,\dots,t_n) \in \mathbf{T}$ and $\alpha=(a_0,\mathbf{a}) \in \mathcal{A},$ we define the sets $I_{\mathbf{t}}(\alpha,\lambda)$ and $H_{\mathbf{t}}(\alpha,\lambda)$ as follows:
    
 \begin{equation}\label{equ:I_t_def}
 I_{\mathbf{t}}(\alpha,\lambda)=\left \{\mathbf{x}\in U
  \left|
\begin{array}{lcl}
\left| a_0 + f(\mathbf{x}) \cdot \mathbf{a} + \theta(\mathbf{x})  \right| < \lambda  \Psi_0(X^{t_1},\dots,X^{t_n}) \\[2ex]
\|\nabla ( f(\mathbf{x}) \cdot \mathbf{a} + \theta(\mathbf{x})  \| < \lambda q^{|\mathbf{t}|(1-\varepsilon)}  \\[2ex]  
\max \{1,|a_i|\}\leq q^{t_i}, \,\, \,\,  i=1,\dots,n
\end{array}\right.\right\}
\end{equation} 

and 
\begin{equation}\label{equ:H_t_def}
 H_{\mathbf{t}}(\alpha,\lambda)=\left \{\mathbf{x}\in U
  \left|
\begin{array}{lcl}
\left| a_0 + f(\mathbf{x}) \cdot \mathbf{a}  \right| < \lambda  \Psi_0(X^{t_1},\dots,X^{t_n}) \\[2ex]
\|\nabla ( f(\mathbf{x}) \cdot \mathbf{a} ) \| <  \lambda q^{|\mathbf{t}|(1-\varepsilon)}  \\[2ex]  
|a_i|\leq q^{t_i}, \,\, \,\,  i=1,\dots,n
\end{array}\right.\right\}
\end{equation} 
\end{itemize}
This gives us the functions (\ref{equ:H}) and (\ref{equ:I}) required in the transference principle. We also get $H_{\mathbf{t}}(\lambda),\,\,  I_{\mathbf{t}}(\lambda),\,\, \Lambda_H(\phi) $ and $\Lambda_I(\phi)$ as defined in the transference principle by equations (\ref{equ:H_t_I_t}), (\ref{equ:H_t_I_t}), (\ref{equ:lambda_H_I}) and (\ref{equ:lambda_H_I}) respectively.

\noindent It is easy to note that $\mathcal W_f^{\text{small}}(\Psi_0,\theta) \subseteq \Lambda_I(\phi_{\delta})$ for every $\delta \in (0,\varepsilon/2).$ Hence to prove our desired result it is enough to show that 
\begin{equation}\label{equ:lambda_I_phi_0}
    |\Lambda_I(\phi_{\delta})|=0 \,\, \text{for some} \,\, \delta \in (0,\varepsilon/2).
\end{equation}
To prove (\ref{equ:lambda_I_phi_0}), we will use the inhomogeneous transference principle. Suppose that $(H,I,\Phi)$ satisfies the intersection property, and the measure $|\cdot|$ is contracting with respect to $(I,\Phi).$ Then in view of Theorem \ref{thm:transf}, to establish (\ref{equ:lambda_I_phi_0}) it enough to show that 
\begin{equation}\label{equ:lambda_H_phi_0}
    |\Lambda_H(\phi_{\delta})|=0 \,\, \text{for all} \,\, \delta \in (0,\varepsilon/2).
\end{equation}
Before going further we fix a notation. Let us denote the LHS of (\ref{equ:small grad}) by $S(t,t',t_1,\dots,t_n).$ Now note that 

\begin{equation*}
    \Lambda_H(\phi_{\delta}) = \displaystyle \limsup_{t \in \mathbf{T}} \,\, \bigcup_{\alpha \in \mathcal{A}} H_t(\alpha,\phi_{\delta}(\mathbf{t})).
\end{equation*}
Using Theorem \ref{thm:small-grad}, we will show that
\begin{equation*}
    \displaystyle \sum_{\mathbf{t} \in T} \left| \bigcup_{\alpha \in \mathcal{A}} H_t(\alpha,\phi_{\delta}(\mathbf{t}))\right| < \infty
\end{equation*}
for all $0 < \delta < \varepsilon/2.$ In view of the Borel-Cantelli lemma, this implies that $|\Lambda_H(\phi_{\delta})|=0.$\\

  Fix any $\delta \in (0,\varepsilon/2).$ In view of (\ref{equ:H_t_def}), we have
  \begin{equation*}
      \left| \bigcup_{\alpha \in \mathcal{A}} H_t(\alpha,\phi_{\delta}(\mathbf{t}))\right| \leq S(-1-[\delta |\mathbf{t}|] + \sum_1^n t_i ,\,\,  2+[\delta |\mathbf{t}|]+[|\mathbf{t}|(1-\varepsilon)], \,\, t_1 +1  ,\dots,t_n+1 ),
  \end{equation*}
where for any real number $x$, $[x]=$ the greatest integer but not greater than $x.$ This implies that 
\begin{equation*}
    \left| \bigcup_{\alpha \in \mathcal{A}} H_t(\alpha,\phi_{\delta}(\mathbf{t}))\right| \leq E \varepsilon_1^{\alpha_1} |U|,
\end{equation*}
where 
\[\varepsilon_1 \defeq \max \left\{q^{1+[\delta |\mathbf{t}|]- \sum_1^n t_i}, q^{\frac{n+3+2[\delta |\mathbf{t}|]+[|\mathbf{t}|(1-\varepsilon)]-|\mathbf{t}|}{n+1}}\right\} \leq \max \left\{q^{1-(1-\delta)|\mathbf{t}|},q^{\frac{n+3-(\varepsilon-2\delta)|\mathbf{t}|}{n+1}}\right\}\]and $\alpha_1 >0$. Observe that for all large $\mathbf{t} \in \mathbb{Z}_{\geq 0}$ both terms in the parenthesis are less than $c_1 q^{-\gamma |\mathbf{t}|}$ for some constants $c_1,\gamma >0$ (we are using the fact that $0 < \delta <\varepsilon/2$, $\sum_i t_i \geq |\mathbf{t}|$ and $[x] \leq x$). Hence we have
\begin{equation*}
    \left| \bigcup_{\alpha \in \mathcal{A}} H_t(\alpha,\phi_{\delta}(\mathbf{t}))\right| \ll q^{-\gamma |\mathbf{t}|},
\end{equation*}
for some $\gamma >0.$ Therefore
\begin{equation*}
    \displaystyle \sum_{\mathbf{t} \in T} \left| \bigcup_{\alpha \in \mathcal{A}} H_t(\alpha,\phi_{\delta}(\mathbf{t}))\right| \ll \displaystyle \sum_{\mathbf{t} \in T} q^{-\gamma |\mathbf{t}|} < \infty.
\end{equation*}
To complete the proof of Theorem \ref{thm:inhomo-conv}, we just need to verify the intersection and contracting properties.

\subsection{Verifying the intersection property}
Let $\mathbf{t} \in T$  with $|\mathbf{t}| \geq 1.$ To  show the intersection property, we suppose that $\mathbf{x} \in  I_{\mathbf{t}}(\alpha, \phi_{\delta}(\mathbf{t})) \cap  I_{\mathbf{t}}(\alpha', \phi_{\delta}(\mathbf{t})) $ for some distinct $\alpha=(a_0,\mathbf{a})$ and $\alpha'=(a_0',\mathbf{a}')$ in $\mathcal{A}.$ Then by Definition (\ref{equ:I_t_def}) we have

\begin{equation}
	\left \{\begin{array}{lcl} |a_0 + f(\mathbf{x}) \cdot \mathbf{a} + \theta(\mathbf{x})| < \phi_{\delta}(\mathbf{t}) \Psi_0(X^{t_1},\dots,X^{t_n})   \\[3ex] \|\nabla ( f(\mathbf{x}) \cdot \mathbf{a} + \theta(\mathbf{x}))\| \ < \phi_{\delta}(\mathbf{t}) q^{|\mathbf{t}|(1-\varepsilon)} \\[2ex]
	|a_i|\leq q^{t_i}, \quad i=1,\dots,n \end{array}\right.
\end{equation} 
and

\begin{equation}
	\left \{\begin{array}{lcl} |a_0' + f(\mathbf{x}) \cdot \mathbf{a}' + \theta(\mathbf{x})| < \phi_{\delta}(\mathbf{t}) \Psi_0(X^{t_1},\dots,X^{t_n})   \\[3ex] \|\nabla ( f(\mathbf{x}) \cdot \mathbf{a}' + \theta(\mathbf{x}))\| \ < \phi_{\delta}(\mathbf{t}) q^{|\mathbf{t}|(1-\varepsilon)} \\[2ex]
	|a_i'|\leq q^{t_i}, \quad i=1,\dots,n \end{array}\right.
\end{equation} 
where $\mathbf{a}=(a_1,\dots,a_n)$ and $\mathbf{a}'=(a_1',\dots,a_n').$ Now subtracting the first inequality from the second within each of the above two systems gives

\begin{equation}
	\left \{\begin{array}{lcl} |a_0'' + f(\mathbf{x}) \cdot \mathbf{a}'' | < \phi_{\delta}(\mathbf{t}) \Psi_0(X^{t_1},\dots,X^{t_n})   \\[3ex] \|\nabla  f(\mathbf{x}) \cdot \mathbf{a}'' \| \ < \phi_{\delta}(\mathbf{t}) q^{|\mathbf{t}|(1-\varepsilon)} \\[2ex]
	|a_i''| \leq q^{t_i}, \quad i=1,\dots,n, \end{array}\right.
\end{equation} 
where $\mathbf{a}''=\mathbf{a}' - \mathbf{a}\defeq(a_1'',\dots,a_n'')$ and $a_0''\defeq a_0'-a_0.$ Note that $\mathbf{a}'' \neq 0,$ because otherwise
\begin{equation*}
    1 \leq |a_0''| < q^{\delta |\mathbf{t}|-\sum_1^n t_i} \leq q^{|\mathbf{t}|(\delta - 1)} < 1,
\end{equation*}
since $0 < \delta < 1.$ Hence $\mathbf{a}'' \neq 0$ and it follows that $\alpha''=(a_0'',\mathbf{a}'') \in \mathcal{A}.$ Therefore $\mathbf{x} \in H_{\mathbf{t}}(\alpha'', \phi_{\delta}(\mathbf{t}))$ and (\ref{equ:inter_property}) is satisfied with $\phi_{\delta}^*=\phi_{\delta}.$ So the verification of intersection property is done.

\subsection{Verifying the contracting property}
Let us consider sufficiently small open ball $V=V(\mathbf{x}_0;r)$ containing $\mathbf{x}_0$ of radius $r>0$ such that Corollary \ref{coro:c-alpha-grad} is valid on $V(\mathbf{x}_0;5r).$ Then there exist positive numbers $M_0,k$ and $C$ such that for any $\mathbf{t} \in T$ and $\alpha=(a_0,\mathbf{a}) \in \mathcal{A}$ with $\|\mathbf{a}\| \geq M_0$, both $|a_0+f(\mathbf{x}) \cdot \mathbf{a}+\theta(\mathbf{x})|$ and $\|\nabla (f(\mathbf{x}) \cdot \mathbf{a} + \theta(\mathbf{x}))\|$ are $(C,1/k)$-good on $V(\mathbf{x}_0;5r).$ Hence for any $\mathbf{t} \in T$ and $\alpha=(a_0,\mathbf{a}) \in \mathcal{A}$ with $\|\mathbf{a}\| \geq M_0$,
\begin{equation}
    G_{\mathbf{t},\alpha}(\mathbf{x}) \defeq \max \left\{\Psi_0^{-1}(X^{t_1},\dots,X^{t_n})q^{|\mathbf{t}|(1-\varepsilon)} |a_0 + f(\mathbf{x}) \cdot \mathbf{a} + \theta(\mathbf{x})|, \|\nabla (f(\mathbf{x}) \cdot \mathbf{a} + \theta(\mathbf{x}))\| \right\}
\end{equation}
is $(C,\frac{1}{k})$-good on $V(\mathbf{x}_0;5r)$. Now observe that 

\begin{equation}\label{equ:I_t_G_t}
 I_{\mathbf{t}}(\alpha, \phi_{\delta}(\mathbf{t})) =\left \{\mathbf{x}\in U
  \left|
\begin{array}{lcl}
G_{\mathbf{t},\alpha}(\mathbf{x}) < \phi_{\delta}(\mathbf{t}) q^{|\mathbf{t}|(1-\varepsilon)} \\[2ex]
\max \{1,|a_i|\}\leq q^{t_i}, \,\, i =1,\dots,n \\[2ex]  
\end{array}\right.\right\}.
\end{equation} 
Also, note that 
\begin{equation*}
    I_{\mathbf{t}}(\alpha, \phi_{\delta}(\mathbf{t})) \subseteq I_{\mathbf{t}}(\alpha, \phi_{\delta}^+(\mathbf{t})),
\end{equation*}
where $\phi_{\delta}^+(\mathbf{t})\defeq\phi_{\frac{\delta}{2}+\frac{\varepsilon}{4}} (\mathbf{t}) \geq \phi_{\delta}(\mathbf{t})$ for all $\mathbf{t} \in T.$ Note that $\phi_{\delta}^+(\mathbf{t})\defeq\phi_{\frac{\delta}{2}+\frac{\varepsilon}{4}} (\mathbf{t}) \in \Phi,$ because $\frac{\delta}{2}+\frac{\varepsilon}{4} < \frac{\varepsilon}{2}.$\\

We now construct the collection $C_{\mathbf{t},\alpha}$ of balls containing points of $\overline{V}$ that satisfy the conditions (\ref{equ:cont_prop_i}), (\ref{equ:cont_prop_ii}) and (\ref{equ:cont_prop_iii})  for a suitable sequence $\{k_{\mathbf{t}}\}$. If $I_{\mathbf{t}}(\alpha, \phi_{\delta}(\mathbf{t})) = \emptyset,$ then the properties holds trivially. Hence Without loss in generality we may assume that $I_{\mathbf{t}}(\alpha, \phi_{\delta}(\mathbf{t})) \neq \emptyset.$ It is easy to observe that  $\phi_{\delta}^+(\mathbf{t}) \Psi_0(X^{t_1},\dots,X^{t_n}) < q^{-(1-\frac{\varepsilon}{2})|\mathbf{t}|}$ for every $\phi_{\delta} \in \Phi.$ In view of this we have 
\begin{equation*}
    I_{\mathbf{t}}(\alpha, \phi_{\delta}^+(\mathbf{t})) \subseteq \left\{\mathbf{x} \in U : |a_0 + f(\mathbf{x}) \cdot \mathbf{a} + \theta(\mathbf{x})| < q^{-(1-\frac{\varepsilon}{2})|\mathbf{t}|}\right\}.
\end{equation*}
By Proposition \ref{prop:non_zero_c_alpha}, there exists a $M_0' >1$ such that 
\begin{equation*}
    \inf_{\tiny{\begin{array}{cc}
     (\mathbf{a},a_0) \in F^{n+1}  \\
         \|\mathbf{a}\| \geq M_0'
    \end{array}}}
\left(\sup_{\mathbf{x} \in V(\mathbf{x}_0;5r)} |a_0+\mathbf{a} \cdot f(\mathbf{x})+\theta(\mathbf{x})|\right)>0.
\end{equation*}
Therefore,
\begin{equation}\label{equ:c_alpha_positivity_con}
    \inf_{\tiny{\begin{array}{cc}
     (\mathbf{a},a_0) \in \Lambda^{n+1}  \\
         \|\mathbf{a}\| \geq M_0'
    \end{array}}}\left(\sup_{\mathbf{x} \in V(\mathbf{x}_0;5r)} \,\, |a_0+\mathbf{a} \cdot f(\mathbf{x})+\theta(\mathbf{x})|\right)>0.
\end{equation}
Recall that   the function $a_0 + f(\mathbf{x}) \cdot \mathbf{a} + \theta(\mathbf{x})$ is $(C,\frac{1}{k})$-good on $V(\mathbf{x}_0;5r)$ for sufficiently large $\|\mathbf{a}\|$. Hence by the definition of $(C,1/k)$-good and the above positivity condition (\ref{equ:c_alpha_positivity_con}), we have that
\begin{eqnarray*}
\left|I_{\mathbf{t}}(\alpha, \phi_{\delta}^+(\mathbf{t})) \cap V \right| & \leq   &  \left |\left\{\mathbf{x} \in V: |a_0 + f(\mathbf{x}) \cdot \mathbf{a} + \theta(\mathbf{x})| < q^{-(1-\frac{\varepsilon}{2})|\mathbf{t}|}\right\}\right| \nonumber \\[2ex]
&  \ll   &  q^{-\frac{1}{k}(1-\frac{\varepsilon}{2})|\mathbf{t}|} |V| \nonumber \end{eqnarray*} 
for all sufficiently large $|\mathbf{t}|.$ Hence $I_{\mathbf{t}}(\alpha, \phi_{\delta}^+(\mathbf{t})) \nsubseteq V$ for all sufficiently large $|\mathbf{t}|.$ We take $S=\overline{V}.$ Note that $I_{\mathbf{t}}(\alpha, \phi_{\delta}^+(\mathbf{t}))$ is open, since $G_{\mathbf{t},\alpha}$ is continuous. Applying this and the fact that  $I_{\mathbf{t}}(\alpha, \phi_{\delta}(\mathbf{t})) \subseteq I_{\mathbf{t}}(\alpha, \phi_{\delta}^+(\mathbf{t})),$ there exists a ball $B'(\mathbf{x})$ containing $\mathbf{x}$ for every $\mathbf{x} \in S \cap I_{\mathbf{t}}(\alpha, \phi_{\delta}(\mathbf{t}))$ such that
\begin{equation}\label{equ:Bx_S}
    B'(\mathbf{x}) \cap S \subseteq I_{\mathbf{t}}(\alpha, \phi_{\delta}^+(\mathbf{t})).
\end{equation}
Now combining $I_{\mathbf{t}}(\alpha, \phi_{\delta}^+(\mathbf{t})) \not\subseteq V$,
(\ref{equ:Bx_S}) and the fact that $V$ is bounded, we get that there exists a constant $\tau \geq 1$ such that the ball $B=B(\mathbf{x})=\tau B'(\mathbf{x})$ satisfies
\begin{equation}
    5B(\mathbf{x}) \subseteq 5V
\end{equation} 
and
\begin{equation}\label{equ:B_S_I_t}
    B(\mathbf{x}) \cap S \subseteq  I_{\mathbf{t}}(\alpha, \phi_{\delta}^+(\mathbf{t})) \nsupseteq S \cap 5B(\mathbf{x})
\end{equation}
holds for all but  finitely many $\mathbf{t}.$ We now take 
\begin{equation}
    C_{\mathbf{t},\alpha} \defeq \{B(\mathbf{x}): \mathbf{x} \in S \cap I_{\mathbf{t}}(\alpha, \phi_{\delta}(\mathbf{t})) \}.
\end{equation}
From the last two set inclusions and the definition of $C_{\mathbf{t},\alpha},$ it is clear that conditions (\ref{equ:cont_prop_i}) and (\ref{equ:cont_prop_ii}) are automatically satisfied. Hence we just need to check condition (\ref{equ:cont_prop_iii}). For that, consider any ball $B \in C_{\mathbf{t},\alpha}.$ By (\ref{equ:B_S_I_t}) and (\ref{equ:I_t_G_t}), we have 
\begin{equation}
    \sup_{\mathbf{x} \in 5B} G_{\mathbf{t},\alpha}(\mathbf{x}) \geq \sup_{\mathbf{x} \in 5B \cap S} G_{\mathbf{t},\alpha}(\mathbf{x}) \geq \phi_{\delta}^+(\mathbf{t}) q^{|\mathbf{t}|(1-\varepsilon)}.
\end{equation}
On the other hand, we have  
\begin{eqnarray}\label{equ:sup_G_t_alpha}
 \sup_{\mathbf{x} \in 5B \cap I_{\mathbf{t}}(\alpha, \phi_{\delta}(\mathbf{t})) } G_{\mathbf{t},\alpha} (\mathbf{x}) & \leq   &  \phi_{\delta}(\mathbf{t}) q^{|\mathbf{t}|(1-\varepsilon)} \nonumber\\[2ex] &  = & \phi_{\delta}^+(\mathbf{t}) q^{|\mathbf{t}|(1-\varepsilon)} q^{|\mathbf{t}|(\frac{\delta}{2}-\frac{\varepsilon}{4})}  \nonumber \\[2ex]
&  \leq  &  q^{|\mathbf{t}|(\frac{\delta}{2}-\frac{\varepsilon}{4})} \sup_{\mathbf{x} \in 5B} G_{\mathbf{t},\alpha}(\mathbf{x}) . \end{eqnarray} 

\noindent Therefore for all large $|\mathbf{t}|$ and $\alpha \in \mathcal{A}$ we have
\begin{eqnarray}
 \left | 5B \cap I_{\mathbf{t}}(\alpha, \phi_{\delta}(\mathbf{t})) \cap V \right| & \leq   &    \left | 5B \cap I_{\mathbf{t}}(\alpha, \phi_{\delta}(\mathbf{t}))  \right| \nonumber \\[2ex]
&  \leq  &  \left| \left\{\mathbf{x} \in 5B : |G_{\mathbf{t},\alpha} (\mathbf{x})| \leq  q^{|\mathbf{t}|(\frac{\delta}{2}-\frac{\varepsilon}{4})} \sup_{\mathbf{x} \in 5B} G_{\mathbf{t},\alpha}(\mathbf{x}) \right \}  \right|  \nonumber \\[2ex]
& \leq & C  q^{|\mathbf{t}|\frac{1}{k}(\frac{\delta}{2}-\frac{\varepsilon}{4})} |5B| \nonumber \\[2ex]
&  \leq   &   C C'q^{|\mathbf{t}|\frac{1}{k}(\frac{\delta}{2}-\frac{\varepsilon}{4})}  |5B \cap V|, \end{eqnarray}
since $5B \subseteq 5V,$ the measure is doubling and the the ball $5B$ contains some point of $\overline{V}.$ Note that $C'$ depends only on $d.$ Therefore condition (\ref{equ:cont_prop_iii}) verifies with $k_{\mathbf{t}}=  C C'q^{|\mathbf{t}|\frac{1}{k}(\frac{\delta}{2}-\frac{\varepsilon}{4})}. $ Now since $(\frac{\delta}{2}-\frac{\varepsilon}{4})<0,$ the series $\sum k_{\mathbf{t}} < \infty$   and hence we are done.

\noindent This completes the verification of the contracting property.

\section{Inhomogeneous divergence case}\label{section:divg}
In this section, we prove Theorem \ref{thm:inhomo-div} using the technique of ubiquitous systems from \cite{BBV,BDV}. Now we recall a few definitions from \cite{BBV}. Let $U$ be an open ball in $F^d$ and $f: U \rightarrow F^n $ is an analytic map satisfying condition \ref{II}. For $\delta >0$ and $t \in \mathbb N $, we define
\begin{equation}
 \Phi^f(t,\delta)\defeq\left \{\mathbf{x}\in U: \exists \,\, \mathbf{a} \in \Lambda^n \setminus \{0\}, a_0 \in \Lambda \,\,
  \left|
\begin{array}{lcl}
\left| f(\mathbf{x}) \cdot \mathbf{a} + a_0 \right| < \delta q^{-nt}\\[2ex]
\|\mathbf{a}\|=q^t \\[2ex]  
\end{array}\right.\right\}.
\end{equation} 
We recall the notion of a nice function from \cite{BBV}.
\begin{definition}\label{def:nice}
  $f$ is said to be $\textit{nice}$ at $\mathbf{x}_0 \in U$ if  there exists a neighbourhood $U_0 \subseteq U$ of $\mathbf{x}_0$ and constants $0<\delta,w<1$ such that for any sufficiently small ball $B \subseteq U_0$ we have that 
 \begin{equation}
     \limsup_{t \rightarrow \infty} |\Phi^f(t,\delta) \cap B| \leq w|B|.
 \end{equation}
\end{definition}
\noindent The map $f$ is said to be nice if it is nice at almost every point in $U.$

\subsection{Ubiquitous Systems in $F^d$}  We recall the ubiquity framework from \cite{BDV,BBV}. Let $U \subseteq F^d$ be an open ball and $\mathcal{R} = (R_{\alpha})_{\alpha \in J}$ be a family of subsets $R_{\alpha} \subseteq F^d$ indexed by a countable set $J$. The sets $R_{\alpha}$ are referred to as $\textit{resonant sets}.$ Throughout $\rho : \mathbb{R}_+ \rightarrow \mathbb{R}_+$ will denote a function such that $\rho(r) \rightarrow 0$ as $r \rightarrow \infty.$ Given a set $A \subseteq U,$ let 
\begin{equation*}
    \Delta(A,r)\defeq \{\mathbf{x} \in U : \text{dist}(\mathbf{x},A) < r \}
\end{equation*}
where $\text{dist}(\mathbf{x},A)\defeq \inf \{ \|\mathbf{x}-\mathbf{a}\| : \mathbf{a} \in A \}.$ Moreover, let $\beta : J \rightarrow \mathbb{R}_+$ such that $\alpha \mapsto \beta_{\alpha}$ be a positive function on  $J.$  Thus the function $\beta$ attaches a `weight' $\beta_{\alpha}$ to the set $R_{\alpha}.$  We will assume that the set $J_t = \{\alpha \in J : \beta_{\alpha} \leq q^t \}$ is finite for all $t \in \mathbb N$.\\

\noindent \textbf{The intersection  conditions:} 
There exists a constant $\gamma$ with $0 \leq \gamma \leq d$  such that for any sufficiently large $t$ and  $\alpha \in J_t,$ $c \in {R}_{\alpha}$ and $0 < \lambda \leq \rho(q^t)$ the following conditions are satisfied:
\begin{equation}\label{equ:ubi_int_pro_i}
    \left| B\left(c;\frac{1}{2} \rho(q^t)\right) \cap \Delta(R_{\alpha},\lambda)\right| \geq c_1 \left|B(c;\lambda)\right| \left(\frac{\rho(q^t)}{\lambda}\right)^{\gamma }
\end{equation} 
\begin{equation}\label{equ:ubi_int_pro_ii}
    \left|B \cap B(c;3\rho(q^t)) \cap \Delta(R_{\alpha},3\lambda)\right| \leq c_2 |B(c;\lambda)| \left(\frac{r(B)}{\lambda}\right)^{\gamma},
\end{equation}
where $B$ is an arbitrary ball containing some point of a resonant set with radius $r(B) \leq 3 \rho(q^t).$ The constants $c_1$ and $c_2$ are positive and absolute. 

\begin{remark} The constant $\gamma$ mentioned just above is called the \emph{common dimension} of the resonant sets. The rationale behind the nomenclature is as follows. In practical applications, it turns out that the resonant set $R_{\alpha}$'s are submanifolds of the same dimension. One can then show that the intersection conditions are fulfilled with $\gamma=\dim R_{\alpha}$. For more details see \cite[\S 2.3]{BDV}.

%are typically contained in smooth submanifolds sharing the same dimension. 
%Although  so that $\Delta (R_{\alpha},\lambda)=B(c;\lambda)$, the aforementioned conditions are trivially met with $\gamma=0$; however,  
%In fact, under such circumstances,   
\end{remark}

\begin{definition}
 Suppose that there exists a ubiquitous function $\rho$ and an absolute constant $k >0$ such that for any ball $B \subseteq U$
 \begin{equation}\label{equ:ubi_cover_pro}
     \liminf_{t \rightarrow \infty }  \left| \bigcup_{\alpha \in J_t} \Delta(R_{\alpha},\rho(q^t)) \cap B \right|  \geq k |B| . 
 \end{equation}
 Furthermore, suppose that the intersection properties (\ref{equ:ubi_int_pro_i}) and (\ref{equ:ubi_int_pro_ii}) are satisfied. Then the system $(\mathcal{R},\beta)$ is called $\textit{locally ubiquitous}$ in $U$ relative to $\rho.$
\end{definition}

Let $(\mathcal{R},\beta)$ be a locally ubiquitous system in $U$ relative to $\rho$ and $\phi$ be an approximating function, i.e., $\phi : \mathbb R_{+} \rightarrow \mathbb{R}_{+}$ is a decreasing function. Let $\Lambda(\phi)$ be the set of points $\mathbf{x} \in U$ such that the inequality 
\begin{equation}
    \text{dist}(\mathbf{x},R_{\alpha}) < \phi(\beta_{\alpha})
\end{equation}
holds for infinitely many $\alpha \in J.$

\noindent Now we recall the following ubiquity lemma from \cite{BBV}, which is the key to prove the divergent part.
\begin{lemma}[Ubiquity lemma]\label{lem:ubiquity}
Let $\phi$ be an approximating function and $(\mathcal{R},\beta)$ be a locally ubiquitous system in $U$ relative to $\rho.$ Suppose that there is a $\lambda \in \mathbb{R},$ $0 < \lambda < 1$ such that $\rho(q^{t+1}) < \lambda \rho(q^t)$  for all $t \in \mathbb N.$ Then for any $s> \gamma,$
\begin{equation}
    \mathscr{H}^s(\Lambda(\phi))= \mathscr{H}^s(U) \text{ if }\sum_{t=1}^{\infty} \frac{{\phi(q^t)}^{s-\gamma}} {{\rho(q^t)}^{d-\gamma}}  = \infty.
\end{equation}
\end{lemma}
In order to prove Theorem \ref{thm:inhomo-div}, first we construct a locally ubiquitous system such that $\Lambda(\phi) \subseteq  \mathcal W_f(\Psi,\theta)$. Then using Lemma \ref{lem:ubiquity}, we get that $ \mathscr{H}^s(\Lambda(\phi))= \mathscr{H}^s(U)$, which in turn yields that $ \mathscr{H}^s(\mathcal W_f(\Psi,\theta))= \mathscr{H}^s(U)$.

\subsection{Minkowski's convex body theorem for function fields}
Convex bodies naturally appear in our construction of locally ubiquitous system and we need the function field analog of Minkowski's convex body theorem proved by K. Mahler in \cite[\S 9]{M}. We first recall the definition of a convex body in function fields. 
\begin{definition}
    By a convex body $\mathcal{C}$ of $F^n$, we mean a free $\mathscr{O}$-submodule of $F^n$ of rank $n.$
\end{definition} 
Given a convex body $
\mathcal{C}$ of $F^n$ and $i=1,\dots,n$, the $i$-th successive minima is defined as
\begin{equation*}
    \mu_i=\min\{|\mathbf{\rho}|: \mathbf{\rho}\in F \setminus \{0\}, \mathbf{\rho} \mathcal C \,\, \text{contains}\,\, i \,\, \text{linearly independent elements of} \,\, \Lambda^n\},
\end{equation*}
where $\mathbf{\rho} \mathcal C$ denotes the dilated convex body defined by $\{\mathbf{\rho} \mathbf{y}: \mathbf{y} \in \mathcal C\}.$
%Note that $\rho \mathcal C$ only depend on the class $\rho \mathcal{O}^{\times}$ in $F^{\times}/\mathcal{O}^{\times}.$ Hence we may restrict to elements of the form $\rho^l,$ where $l \in \mathbb Z.$
 Minkowski's convex body theorem in the context of function fields is as follows. 
\begin{theorem}{\cite[\S 9]{M}}\label{thm:mink_conv_body}
    Let $\mathcal C$ be a convex body of $F^n$ and $\mu_i$ be the $i$-th successive minima of $\mathcal C$ for $i=1,\dots,n.$ Then we have the following
    \begin{equation*}
        \mu_1 \dots \mu_n |\mathcal C|=1.
    \end{equation*}
\end{theorem}

\subsection{The proof of Theorem \ref{thm:inhomo-div}}We first deduce the following corollary of Theorem \ref{thm:big-grad} and \ref{thm:small-grad}, which says that the functions satisfying the hypothesis of Theorem \ref{thm:inhomo-div} are nice.
\begin{corollary}\label{cor:nice}
Let $U \subseteq F^d$ is open, $f : U \rightarrow F^n$ satisfies \ref{II}, \ref{III}, \ref{IV} and $\mathbf{x}_0 \in U.$ Then there exists a sufficiently small ball $B_0 \subseteq U$ containing $\mathbf{x}_0$  and positive constants $C,$ $\alpha_1$ such that for any ball $B \subseteq B_0$ and any $0 < \delta <1$  we have 
\begin{equation}
    |\Phi^f(t,\delta) \cap B| \leq C \delta^{\alpha_1} |B|,
\end{equation}
for sufficiently large $t.$
\end{corollary}
\begin{proof}
Observe that $\Phi^f(t,\delta) \cap B = \left(\Phi_{\text{small}}^f(t,\delta) \cap B\right)  \cup \left(\Phi_{\text{big}}^f(t,\delta) \cap B \right),$ where for fixed $0 < \varepsilon < 1/2 $

\begin{equation}
 \Phi_{\text{big}}^f(t,\delta) \cap B\defeq\left \{\mathbf{x}\in B: \exists \,\, \mathbf{a} \in \Lambda^n \setminus \{0\}, a_0 \in \Lambda \,\,
  \left|
\begin{array}{lcl}
\left| f(\mathbf{x}) \cdot \mathbf{a} + a_0 \right| < \delta q^{-nt}\\[2ex]
\|\nabla(f(\mathbf{x}) \cdot \mathbf{a})\| \geq \|\mathbf{a}\|^{1-\varepsilon} \\ [2ex]
\|\mathbf{a}\|=q^t \\[2ex]  
\end{array}\right.\right\}
\end{equation} 
and
\begin{equation}
 \Phi_{\text{small}}^f(t,\delta) \cap B \defeq\left \{\mathbf{x}\in B: \exists \,\, \mathbf{a} \in \Lambda^n \setminus \{0\}, a_0 \in \Lambda \,\,
  \left|
\begin{array}{lcl}
\left| f(\mathbf{x}) \cdot \mathbf{a} + a_0 \right| < \delta q^{-nt}\\[2ex]
\|\nabla(f(\mathbf{x}) \cdot \mathbf{a})\| < \|\mathbf{a}\|^{1-\varepsilon} \\ [2ex]
\|\mathbf{a}\|=q^t \\[2ex]  
\end{array}\right.\right\}.
\end{equation} 
Now choose $t'' \in \mathbb N$ such that $q^{-n(t''+1)} \leq \delta < q^{-nt''}.$ In view to this 
\begin{equation*}
 \Phi_{\text{small}}^f(t,\delta) \cap B \subseteq \left \{\mathbf{x}\in B: \exists \,\, \mathbf{a} \in \Lambda^n \setminus \{0\}, a_0 \in \Lambda \,\,
  \left|
\begin{array}{lcl}
\left| f(\mathbf{x}) \cdot \mathbf{a} + a_0 \right| <  q^{-n(t+t'')}\\[2ex]
\|\nabla(f(\mathbf{x}) \cdot \mathbf{a})\| < q^{[t(1-\varepsilon)]+1} \\ [2ex]
|a_i| < q^{t+1}, \quad  \text{for} \,\, i=1,\dots,n \\[2ex]  
\end{array}\right.\right\}
\end{equation*} 
By  Theorem \ref{thm:small-grad}, we get
\begin{equation*}
    | \Phi_{\text{small}}^f(t,\delta) \cap B| \leq E \varepsilon^{\alpha} |B |,
\end{equation*}
where the constant $E$ is uniform, $\alpha >0$ and $\varepsilon= \max \left\{q^{-n(t+t'')},q^{\frac{n+[t(1-\varepsilon)]-t-nt''}{n+1}}\right\}.$ Now observe that $\varepsilon \leq \delta^{\frac{1}{n+1}}  \max \left\{q^{\frac{n-nt}{n+1}},q^{\frac{2n-t\varepsilon}{n+1}} \right\} \leq \delta^{\frac{1}{n+1}}$ for large enough $t,$ since $q^{-n(t''+1)} \leq \delta.$

\noindent By  Corollary \ref{coro:big_grad_thm}, we have
\begin{equation*}
    | \Phi_{\text{big}}^f(t,\delta) \cap B| \leq C' \delta |B|,
\end{equation*}
for large enough $t$ and a universal constant $C'.$

\noindent Now choose $\alpha_1=\min\{\frac{\alpha}{n+1},1\}>0.$ Therefore combining the above two estimates, we get that
\begin{equation*}
     |\Phi^f(t,\delta) \cap B| \leq C \delta^{\alpha_1} |B|
\end{equation*}
for sufficiently large $t$, where $C$ is a positive constant.

\end{proof}
In view of Corollary \ref{cor:nice}, Theorem \ref{thm:inhomo-div} reduces to the following theorem. 
\begin{theorem}\label{thm:inhomo_div_main}
Let $f : U \subseteq F^d \rightarrow F^n$ is nice and satisfies \ref{II}, \ref{III}, and \ref{IV}. Let $\Psi(\mathbf{a})= \psi(\|\mathbf{a}\|),$ for $\mathbf{a} \in \Lambda^{n} \setminus \{\mathbf{0}\}$, where $\psi : \mathbb{R}_{+} \rightarrow \mathbb{R}_+$ is a decreasing function and $\theta : U \rightarrow F $ be an analytic map satisfying \ref{VI}. Then for any $s > d-1,$ we have 
\begin{equation}
    \mathscr{H}^s(\mathcal W_f(\Psi,\theta) \cap U)=\mathscr{H}^s(U) \,\,  \,\,  \,\, if \,\, \,\, \,\, \displaystyle \sum_{\mathbf{a} \in \Lambda^n  \setminus \{0\} }  \|\mathbf{a}\| \left(\frac{\Psi (\mathbf{a})}{\|\mathbf{a}\|}\right)^{s+1-d} = \infty.
\end{equation}
\end{theorem}
As mentioned earlier, we use ubiquitous systems to prove Theorem \ref{thm:inhomo_div_main}. To apply the ubiquity lemma, first, we construct a locally ubiquitous system as deems necessary. Recall that $$f=(f_1,\dots,f_n) : U \subseteq F^d \rightarrow F^n$$ is  a nice map satisfying  \ref{II}, \ref{III}, and \ref{IV}. Also suppose that $\theta : U \rightarrow F $ is an analytic map satisfying \ref{VI}. Let
\begin{equation}
  \mathcal{F}_n \defeq\left \{  \mathbf{f} : U \rightarrow F 
  \left|
\begin{array}{lcl}
\mathbf{f}(\mathbf{x})=  a_0 + a_1 f_1(\mathbf{x}) + \dots + a_n f_n(\mathbf{x}), \\
\text{ with }\tilde{\mathbf{a}}=(a_0,a_1,\dots,a_n) \in \Lambda^{n+1} \setminus \{0\} \\
\end{array}\right.\right\}. 
\end{equation}
Now given $\mathbf{f} \in \mathcal{F}_n,$ consider
\begin{equation}
    \widetilde{R_{\mathbf{f}}} \defeq \{ \mathbf{x} \in U : (\mathbf{f}+ \theta)(\mathbf{x})=0\} \,\, \text{and} \,\, H(\mathbf{f}) \defeq \max_{1 \leq i \leq n} |a_i|=\|\mathbf{a}\|,
\end{equation}
where $\mathbf{a}=(a_1,\dots,a_n).$ The following ubiquity statement is the key to proving Theorem \ref{thm:inhomo_div_main}.
\begin{theorem}\label{thm:ubiqui_system}
Let $\mathbf{x}_0 \in U\subseteq F^d$ be open. Suppose that  $f: U \rightarrow F^n$ is nice at $\mathbf{x}_0$ and satisfies \ref{II}, \ref{III} and \ref{IV}. Consider an analytic map $\theta : U \rightarrow F$  satisfying \ref{VI}.  Then there exist
 a neighbourhood $U_0$ of $\mathbf{x}_0$, constants $k_0,k_1 >0$, and a family $\mathcal{R} \defeq(R_{\mathbf{f}})_{\mathbf{f} \in \mathcal{F}_n}$ of sets $R_{\mathbf{f}} \subseteq  \widetilde{R_{\mathbf{f}}} \cap U_0$ such that the system $(\mathcal{R},\beta)$ is locally ubiquitous in $U_0$ relative to $\rho(r)=k_1 r^{-(n+1)}$ with common dimension $\gamma = d-1,$ where $\beta : \mathcal{F}_n \rightarrow \mathbb{R}_+$ is defined by
\begin{equation*}
     \beta(\mathbf{f}) \defeq\beta_{\mathbf{f}}=k_0 H(\mathbf{f})= k_0 \|{\mathbf{a}}\|.
\end{equation*}

\end{theorem}
\begin{proof}
Let $\pi : F^d \rightarrow F^{d-1}$ be the projection map defined by
\begin{equation*}
    \pi(x_1,\dots,x_d)=(x_2,\dots,x_d).
\end{equation*}
Let $ U_0$ be any open subset of $U$ with $\mathbf{x}_0 \in U_0$. For $\mathbf{f} \in \mathcal{F}_n$, let 
\begin{equation}
    \widetilde{V}\defeq\pi(\widetilde{R_{\mathbf{f}}} \cap U_0), \quad V\defeq\bigcup_{\tiny{\begin{array}{rcl}
         B\subseteq \tilde{V} \text{ of }\\
         \text{radius }3\rho(\beta_{\mathbf{f}})
    \end{array}} } \frac{1}{2} B
\end{equation}
and 
\begin{equation}   
R_{\mathbf{f}}\defeq \begin{cases}\displaystyle  \pi^{-1}(V) \cap \widetilde{R}_{\mathbf{f}} & \text{ if } |\partial_1(\mathbf{f}+\theta)(\mathbf{x})| > \lambda \|\nabla(\mathbf{f}+\theta)(\mathbf{x})\|, \, \forall \mathbf{x} \in U_0 \\
     \emptyset            &\text{otherwise} 
\end{cases},\end{equation} 
where $0 < \lambda < 1$ is fixed. \\

We claim that $R_{\mathbf{f}}$ are resonant sets. Note that the intersection conditions can be checked in a similar way as done in \cite{BBV}. We only need to observe that in \cite{BBV} the implicit function theorem for $C^l(U)$ functions in $\mathbb{R}^n$ was used and the implicit function theorem in the function field holds for analytic maps and all our concerned maps are assumed to be analytic. So we can easily check the intersection property in our setup by adapting the proof of  \cite[Proposition 5]{BBV}.  Thus to prove ubiquity we only need to check the covering property (\ref{equ:ubi_cover_pro}).\\

Let $U_0$ be the neighbourhood of $\mathbf{x}_0$ that arises from Definition \ref{def:nice}. Without loss in generality, we will always assume that the ball $U_0$ satisfies
\begin{equation}
    \text{diam}\,\, U_0 \leq \frac{1}{q^2}.
\end{equation}
Since $f$ is nice at $\mathbf{x}_0,$ there exist fixed $0 < \delta,w < 1$ such that for any arbitrary ball $B \subseteq U_0$ 
\begin{equation}
     \limsup_{t \rightarrow \infty} \left|\Phi^f(t,\delta) \cap \frac{1}{2} B \right| \leq w \left|\frac{1}{2}B \right|.
 \end{equation}
 Hence for sufficiently large $t$ we have 
 \begin{equation*}
    \left |\frac{1}{2}B \setminus \Phi^f(t,\delta)\right| \geq (1-w) \left|\frac{1}{2}B \right|=c(1-w)|B|,
 \end{equation*}
 for some constant $c>0.$
 
\noindent Therefore it is enough to show that 
 \begin{equation}\label{equ:phi_sub_del_R}
     \frac{1}{2}B \setminus \Phi^f(t,\delta) \subseteq \bigcup_{\mathbf{f} \in \mathcal{F}_n,\,\, \beta_{\mathbf{f}} \leq q^t} \Delta (R_{\mathbf{f}},\rho(q^t)) \cap B 
 \end{equation}
Let $\mathbf{x} \in \frac{1}{2}B \setminus \Phi^f(t,\delta)$ and consider the following system  
 \begin{equation}\label{equ:convex_set}
     \left \{ \begin{array}{lcl}\displaystyle  |a_0 + a_1 f_1(\mathbf{x})+ \dots + a_n f_n(\mathbf{x}) | < q^{-nt} \\
        |a_i| \leq q^t, \quad i=1,\dots,n.          .
                      
                        \end{array} \right.
 \end{equation}
 The set of $(a_0,a_1,\dots,a_n) \in F^{n+1}$ satisfying (\ref{equ:convex_set}) give rise to a convex body $\mathcal{C}$ in $F^{n+1},$ which is symmetric about the origin. Let $\mu_i$ be the $i$-th successive minima of $\mathcal C$ for $i=1,\dots,n+1.$ Now by using Minkowski's convex body theorem over function fields (i.e., Theorem \ref{thm:mink_conv_body}), we have
 \begin{equation*}
     \begin{array}{lcl}
     \displaystyle  \quad \quad \mu_1 \mu_2 \dots \mu_{n+1} |\mathcal C|=1 \\
        \displaystyle   \implies  \mu_{n+1}=\frac{1}{\mu_1 \mu_2 \dots \mu_{n}} < \delta^{-n},
         \end{array}
 \end{equation*}\\
 since $   \mu_1  > \delta $ (as $\mathbf{x} \notin \Phi^f(t,\delta)$), $|\mathcal C|=1$ and $\mu_1 \leq \dots \leq \mu_{n+1}).$
 
 Before going further we choose $t' \in \mathbb{N}$ such that $q^{-t'} \leq \delta < q^{-(t'-1)}.$  By the definition of $\mu_{n+1},$ there exists $n+1$ linearly independent vectors $\mathbf{a}_j =(a_{j,0},\dots,a_{j,n}) \in \Lambda^{n+1}$ ($0 \leq j \leq n$) such that the functions $g_j$ given by 
 \begin{equation*}
     g_j(\mathbf{x})= a_{j,0}+a_{j,1}f_1(\mathbf{x})+ \dots + a_{j,n}f_n(\mathbf{x})
 \end{equation*}
 satisfy
 \begin{equation}\label{equ:g_j_x}
	\left \{\begin{array}{lcl}|g_j(\mathbf{x})| \leq 
	\delta^{-n} q^{-nt}\\[2ex]  |a_{j,i}| \leq  \delta^{-n} q^{t} \quad \text{for} \,\, i=1,\dots,n.
	\end{array}\right.\,
\end{equation}

\noindent Now consider the following system of linear equations
\begin{equation}\label{equ:system_eta_g}
    \left \{\begin{array}{lcl}
         \eta_0 g_0(\mathbf{x}) +  \eta_1 g_1(\mathbf{x}) + \dots +  \eta_n g_n(\mathbf{x})+\theta(\mathbf{x})=0  \\
         \eta_0 \partial_1g_0(\mathbf{x}) +  \eta_1 \partial_1g_1(\mathbf{x}) + \dots +  \eta_n \partial_1g_n(\mathbf{x})+\partial_1\theta(\mathbf{x})=X^{nt'+t+1} \\
          \eta_0 a_{0,j} +   \dots +  \eta_n a_{n,j}=0 \quad (2 \leq j \leq n).
    \end{array}\right.
\end{equation}
Since $f_1(\mathbf{x})=x_1,$ the determinant of this system is $\text{det}(a_{i,j}) \neq 0.$ Therefore there exists a unique solution to the system, say $(\eta_0,\dots,\eta_n) \in F^{n+1}.$ Now let $r_i=[\eta_i]=$ the polynomial part of the Laurent series $\eta_i,$ for $i=0,\dots,n.$ Then we have 
\begin{equation}\label{equ:r_i_eta_i}
    |r_i - \eta_i| < 1 \quad (0 \leq i \leq n)
\end{equation}
Let 
\begin{equation}\label{equ:g_g_i_f_i}
    g(\mathbf{x}) :  = r_0 g_0(\mathbf{x})+ \dots + r_n g_n(\mathbf{x}) =a_0+a_1f_1(\mathbf{x})+\dots + a_n f_n(\mathbf{x}),
\end{equation}
where 
\begin{equation}\label{equ:a_i}
    a_i \defeq r_0a_{0,i}+ \dots + r_n a_{n,i}, \quad \text{for} \,\, i=0,\dots,n.
\end{equation}

We claim the following:
\begin{enumerate}[label=(B.{{\arabic*}})]
    \item\label{item:B1} $g$ defined above satisfies:
    $|\partial_1(g+\theta)(\mathbf{x})| > \lambda \|\nabla(g +\theta)(\mathbf{x})\|, \,\, \forall \,\,  \mathbf{x} \in U_0,$ where $0 < \lambda < 1.$
    \item\label{item:B2} $g$ also satisfies the height condition $\beta_g \leq q^t.$
    \item\label{item:B3} $\mathbf{x} \in \Delta(R_g,\rho(q^t)).$
\end{enumerate}
If we can prove these claims, (\ref{equ:phi_sub_del_R}) follows and we are done.\\

\noindent $\bullet \,\, $  \textbf{Checking \ref{item:B2}}: Here first we establish a few bounds, which will be useful in checking these three claims. In view of (\ref{equ:a_i}), (\ref{equ:system_eta_g}), (\ref{equ:r_i_eta_i}) and (\ref{equ:g_j_x}) we get that 

\begin{equation}\label{equ:bound_a_j_geq_2}
    |a_j| \leq \delta^{-n} q^t \leq q^{nt'+t} \quad (2 \leq j \leq n)
\end{equation}
and 
\begin{equation}\label{equ:bound_g_theta}
    |g(\mathbf{x})+\theta(\mathbf{x})|=|(r_0 - \eta_0)g_0(\mathbf{x})+\dots +(r_n - \eta_n)g_n(\mathbf{x})| \leq \delta^{-n} q^{-nt} \leq q^{nt'-nt}
\end{equation}
Using (\ref{equ:g_j_x}) and \ref{IV}, we get
\begin{equation}\label{equ:bound_partial_g_j}
    |\partial_1 g_j(\mathbf{x})| \leq \delta^{-n}q^t \leq q^{nt'+t} \quad \text{for} \,\, j=0,\dots,n.
\end{equation}

\noindent Now using (\ref{equ:system_eta_g}),  (\ref{equ:g_g_i_f_i}), we get 
\begin{equation}
    \partial_1(g + \theta)(\mathbf{x})=(r_0 - \eta_0)\partial_1 g_0(\mathbf{x}) + \dots + (r_n - \eta_n)\partial_1 g_n(\mathbf{x}) + X^{nt'+t+1}
\end{equation}
Therefore using  (\ref{equ:bound_partial_g_j}) and (\ref{equ:r_i_eta_i}), we get
\begin{equation}\label{equ:lower_bound_d_g_theta}
    | \partial_1(g + \theta)(\mathbf{x})| \geq q^{nt'+t+1} - q^{nt'+t}=q^{nt'+t}(q-1) \geq q^{nt'+t},
\end{equation}
since  $q \geq 2$. Also using (\ref{equ:r_i_eta_i}) and (\ref{equ:bound_partial_g_j}) we have
\begin{equation}\label{equ:upper_bound_d_g_theta}
    | \partial_1(g + \theta)(\mathbf{x})|  \leq q^{nt'+t+1}
\end{equation}
Also since $f_1(\mathbf{x})=x_1,$ we can write
\begin{equation}\label{equ:a_1_partial_1_g_theta}
    a_1=\partial_1(g+\theta)(\mathbf{x}) - \partial_1 \theta(\mathbf{x}) - \sum_{j=2}^{n} a_j \partial_1 f_j(\mathbf{x})
\end{equation}
and 
\begin{equation}
    a_0=(g+\theta)(\mathbf{x}) -  \theta(\mathbf{x}) - \sum_{j=1}^{n} a_j f_j(\mathbf{x})
\end{equation}
In view of these two equations, (\ref{equ:bound_a_j_geq_2}), (\ref{equ:upper_bound_d_g_theta}), (\ref{equ:bound_g_theta}), \ref{IV} and \ref{VI}, we get 
\begin{equation}\label{equ:bound_a_1}
    |a_1| \leq q^{nt'+t+1}
\end{equation}
and
\begin{equation}\label{equ:bound_a_0}
    |a_0| \leq q^{nt'+t+1}.
\end{equation}
Now in view of (\ref{equ:bound_a_j_geq_2}), (\ref{equ:bound_a_1}) and  (\ref{equ:lower_bound_d_g_theta}) we get 
\begin{equation}\label{equ:bound_beta_g}
    k_0^* q^t < \beta_g : = q^{-(nt'+1)} \|(a_1,\dots,a_n)\| \leq q^t,
\end{equation}
for some constant $0 <k_0^* <1.$ We take $k_0=q^{-(nt'+1)}.$

\noindent $\bullet \,\, $  \textbf{Checking \ref{item:B1}}:
In view of Taylor's formula, for any $\mathbf{y} \in U_0$ we have
\begin{eqnarray}
 \partial_1 (g + \theta)(\mathbf{x}) & = & \partial_1 (g + \theta)(\mathbf{y}) + \displaystyle \sum_{j=1}^{d} \bar{\Phi}_{j}( \partial_1 (g + \theta))( \cdot )(x_{j}-y_{j}) \nonumber \\[2ex]
 & = &  \partial_1 (g + \theta)(\mathbf{y}) + \displaystyle \sum_{j=1}^{d} \bar{\Phi}_{j} \circ \bar{\Phi}_{1} (g + \theta)(\cdot)(x_{j}-y_{j}), \nonumber 
\end{eqnarray}
where the argument of $\bar{\Phi}_{j}( \partial_1 (g + \theta))$ is some of the components of $\mathbf x$ and $\mathbf y.$

\noindent Now using \ref{IV}, \ref{VI}, (\ref{equ:lower_bound_d_g_theta}), (\ref{equ:bound_a_j_geq_2}), (\ref{equ:bound_a_1}) and the above equation, we get that
\begin{eqnarray}\label{equ:d_g_theta_y_lower_bound}
 |\partial_1 (g + \theta)(\mathbf{y})| & \geq & |\partial_1 (g + \theta)(\mathbf{x})| - \left|  \displaystyle \sum_{j=1}^{d} \bar{\Phi}_{j} \circ \bar{\Phi}_{1} (g + \theta)(\cdot)(x_{j}-y_{j})\right| \nonumber \\[2ex] & \geq & q^{nt'+t} - \text{diam}\,\, U_0 \,\, q^{nt'+t+1} \nonumber  \\[2ex] & \geq & q^{nt'+t}\left(1-\frac{1}{q}\right)  \quad \text{for any} \,\, \mathbf{y} \in U_0,
\end{eqnarray}
since $\text{diam} \,\, U_0 \leq \frac{1}{q^2}$ and $t,t' \in \mathbb N.$ Also using (\ref{equ:bound_a_j_geq_2}), (\ref{equ:bound_a_1}), \ref{IV} and \ref{VI} we get that 
\begin{equation}\label{equ:gradient_g_theta_y_bound}
    \|\nabla(g+\theta)(\mathbf{y})\| \leq q^{nt'+t+1} \quad \text{for any} \,\, \mathbf{y} \in U_0.
\end{equation}
In view of (\ref{equ:d_g_theta_y_lower_bound}) and (\ref{equ:gradient_g_theta_y_bound}), we get 
\begin{equation*}
    |\partial_1(g+\theta)(\mathbf{y})| > \lambda \|\nabla(g +\theta)(\mathbf{y})\|\,\,  \,\, \forall \,\,  \mathbf{y} \in U_0,
\end{equation*}
with $0 < \lambda=\frac{1}{q}-\frac{1}{q^2} < 1.$\\

\noindent $\bullet \,\, $  \textbf{Checking \ref{item:B3}}: Here we will show that $\mathbf{x} \in \Delta(R_g,\rho(q^t)).$ Let $\mathbf{x}=(x_1,x_2,\dots,x_d)$ and consider the set
\begin{equation*}
    I= \{\eta \in F : (x_1 + \eta,x_2,\dots,x_d) \in B \}.
\end{equation*}
Now consider the function $h: I \rightarrow F$ defined by
\begin{equation*}
    h(\eta) \defeq (g+\theta)(x_1 + \eta,x_2,\dots,x_d), \quad \text{for}\,\, \eta \in I.
\end{equation*}
Using (\ref{equ:bound_g_theta}) and (\ref{equ:lower_bound_d_g_theta}), we have
\begin{equation}\label{equ:bound_h_h'}
    |h(0)| \leq q^{nt'-nt} \quad \text{and} \quad |h'(0)| \geq q^{nt'+t}(q-1)
\end{equation}
Now applying the local invertibilty theorem (Theorem $27.5$, \cite{WHS}), there exists $\tilde{\eta} \in F$ such that 
\begin{equation}\label{equ:bound_eta_tilde}
    h(\tilde{\eta})=0  \quad \text{and} \quad |\tilde{\eta}| < \frac{q^{nt'-nt} }{q^{nt'+t}(q-1)}=\frac{1}{q^{nt+t}(q-1)}
\end{equation}
Then $\mathbf{x}_{\tilde{\eta}}=(x_1 + \tilde{\eta},x_2,\dots,x_d) \in B$ satisfies
\begin{equation}\label{equ:bound_x_x_eta_tilde}
    (g+\theta)(\mathbf{x}_{\tilde{\eta}})=0 \quad \text{and} \quad \|\mathbf{x} - \mathbf{x}_{\tilde{\eta}}\| \leq \frac{1}{q^{nt+t}(q-1)}.
\end{equation}
By the Mean value theorem, we get 
\begin{equation}
    |(g+\theta)(\mathbf{y})| \ll q^{-nt} \quad \text{for any} \,\, \|\mathbf{y}-\mathbf{x}_{\tilde{\eta}}\| \ll q^{-nt}
\end{equation}
We define $k_1 \defeq q^{-nt'}$ and recall that $\rho(r)=k_1 r^{-(n+1)}.$ Then by (\ref{equ:bound_beta_g}) we have $\rho(\beta_g) \ll \rho(q^t),$ since $\rho   $ is monotone decreasing. Now by using the above argument, we get for sufficiently large $t$ the ball of radius $3\rho(\beta_g)$ containing the point $\pi(\mathbf{x}_{\tilde{\eta}})$ is contained in $\widetilde{V}.$ This gives $\mathbf{x}_{\tilde{\eta}} \in R_g.$ And this implies $\mathbf{x} \in \Delta(R_g,\rho(q^t))$ for some $g \in \mathcal{F}_n$ such that $\beta_g \leq q^t,$ in view of (\ref{equ:bound_x_x_eta_tilde}). This completes the proof of Theorem \ref{thm:ubiqui_system}.
\end{proof}

\subsection{The proof of Theorem \ref{thm:inhomo_div_main}} Here we proof the divergence case of inhomogeneous Khitchine-Groshev type theorem by proving Theorem \ref{thm:inhomo_div_main}, using Theorem \ref{thm:ubiqui_system} and the ubiquity lemma.

Fix $\mathbf{x}_0 \in U$ and let $U_0$ be  as in Theorem \ref{thm:ubiqui_system}. Now we need to show that
\begin{equation*}
     \mathscr{H}^s(\mathcal W_f(\Psi,\theta) \cap U_0)=\mathscr{H}^s(U_0) \,\,  \,\,  \,\, if \,\, \,\, \,\, \displaystyle \sum_{\mathbf{a} \in \Lambda^n  \setminus \{0\} } \|\mathbf{a}\| \left(\frac{\Psi (\mathbf{a})}{\|\mathbf{a}\|}\right)^{s+1-d} = \infty.
\end{equation*}
We consider $\phi(r)\defeq k_0 r^{-1}\psi(k_0^{-1}r).$ First we show that
\begin{equation}
    \Lambda(\phi) \subseteq \mathcal W_f(\Psi,\theta).
\end{equation}
For that, let $\mathbf{x} \in \Lambda(\phi).$ Then there exists infinitely many $g=a_0 + a_1 f_1(\mathbf{x}) + \dots + a_n f_n(\mathbf{x}) \in \mathcal{F}_n$ such that $\text{dist}(\mathbf{x},R_g) < \phi(\beta_g).$ Now for such $g \in \mathcal{F}_n$ there exists $\mathbf{z} \in U_0 $ such that $(g+\theta)(\mathbf{z})=0$ and $\|\mathbf{x}-\mathbf{z}\| < \phi(\beta_g).$ By the mean value theorem, we get
\begin{eqnarray*}
 (g+\theta)(\mathbf{x})=(g+\theta)(\mathbf{z}) + \nabla (g+\theta)(\mathbf{z}) \cdot (\mathbf{x}-\mathbf{z}) \\[2ex]
 + \displaystyle \sum_{i,j} \bar {\Phi}_{ij}(g+\theta)(\mathbf{t}_{i,j})(x_i -z_i)(x_j-z_j),
\end{eqnarray*}
where $\mathbf{t}_{i,j}$ comes from the components of $\mathbf{x}$ and $\mathbf{z}.$ Now in view of the above equation, we have
\begin{equation}
    |(g+\theta)(\mathbf{x})| \leq \|\mathbf{a}\| \|\mathbf{x}-\mathbf{z}\|< \|\mathbf{a}\| \phi(\beta_g)= \|\mathbf{a}\| \phi(k_0 \|\mathbf{a}\|)=\psi(\|\mathbf{a}\|)=\Psi(\mathbf{a}),
\end{equation}
where $\mathbf{a}=(a_1,\dots,a_n).$ Hence $\mathbf{x} \in \mathcal W_f(\Psi,\theta) $ and thus we get $\Lambda(\phi) \subseteq \mathcal W_f(\Psi,\theta).$ Now to complete the proof we just need to check the divergent condition, i.e., 
\begin{equation*}
    \sum_{t=1}^{\infty} \frac{{\phi(q^t)}^{s-\gamma}} {{\rho(q^t)}^{d-\gamma}}  = \infty, \quad \text{where} \,\, \gamma =d-1.
\end{equation*}
Note that 
 \begin{eqnarray*}
 \sum_{t=1}^{\infty} \frac{{ \phi(q^t)}^{s-d+1}} {{\rho(q^t)}}  \asymp  \sum_{t=1}^{\infty} \frac{(q^{-t}{\psi(k_0^{-1}q^t)})^{s-d+1}} {{\rho(q^t)}} & \asymp & \sum_{t=1}^{\infty} (q^{-t}{\psi(k_0^{-1}q^t)})^{s-d+1} q^{t(n+1)} \\[2ex] & \gg & \sum_{t=1}^{\infty} \sum_{\|\mathbf{a}\|= k_0^{-1}q^{t}} \|\mathbf{a}\| \left(\frac{\psi(k_0^{-1}q^t)}{\|\mathbf{a}\|}\right)^{s-d+1}  \\[2ex]
 & \gg & \sum_{t=1}^{\infty} \sum_{ \|\mathbf{a}\|= k_0^{-1}q^{t}} \|\mathbf{a}\| \left(\frac{\psi(\|\mathbf{a}\|)}{\|\mathbf{a}\|}\right)^{s-d+1} \\[2ex]
 & \asymp & \sum_{\mathbf{a} \in \Lambda^n  \setminus \{0\} } \|\mathbf{a}\|  \left(\frac{\Psi (\mathbf{a})}{\|\mathbf{a}\|}\right)^{s+1-d} 
 = \infty,
\end{eqnarray*}
since $\psi$ is an approximation function, $k_0=q^{-(nt'+1)},$ and the number of $\mathbf{a} \in \Lambda^n$ such that $\|\mathbf{a}\|=k_0^{-1}q^t$ is comparable to $q^{nt}$. This completes the proof of the theorem.

\end{document}